\documentclass[12pt]{amsart}
\usepackage[utf8]{inputenc}
\usepackage{amssymb}
\usepackage{hyperref}
\usepackage[final]{showkeys}           

\input xy
\xyoption{all}

\theoremstyle{definition}
\newtheorem{mydef}{Definition}[section]
\newtheorem{lem}[mydef]{Lemma}
\newtheorem{thm}[mydef]{Theorem}
\newtheorem{cor}[mydef]{Corollary}
\newtheorem{claim}[mydef]{Claim}
\newtheorem{question}[mydef]{Question}
\newtheorem{hypothesis}[mydef]{Hypothesis}
\newtheorem{prop}[mydef]{Proposition}
\newtheorem{defin}[mydef]{Definition}
\newtheorem{example}[mydef]{Example}
\newtheorem{remark}[mydef]{Remark}
\newtheorem{notation}[mydef]{Notation}
\newtheorem{fact}[mydef]{Fact}
\newtheorem{conjecture}[mydef]{Conjecture}

\newcommand{\fct}[2]{{}^{#1}#2}

\newcommand{\ba}{\bar{a}}
\newcommand{\bb}{\bar{b}}
\newcommand{\bc}{\bar{c}}

\newcommand{\bigN}{\widehat{N}}

\newcommand{\sea}{\mathfrak{C}}

\newcommand{\cf}[1]{\text{cf} (#1)}
\newcommand{\seq}[1]{\langle #1 \rangle}
\newcommand{\rest}{\upharpoonright}
\newcommand{\bkappa}{\bar \kappa}
\newcommand{\s}{\mathfrak{s}}
\newcommand{\snsp}[1]{\s_{#1\text{-ns}}}
\newcommand{\sns}{\s_{\text{ns}}}
\newcommand{\insp}[1]{\is_{#1\text{-ns}}}
\newcommand{\ins}{\is_{\text{ns}}}
\newcommand{\is}{\mathfrak{i}}
\newcommand{\isch}[1]{\is_{#1\text{-ch}}}
\newcommand{\ts}{\mathfrak{t}}

\newcommand{\id}{\text{id}}

\def\lta{<}
\def\lea{\le}

\def\gea{\ge}

\def\ltu{\lta_{\text{univ}}}
\def\leu{\lea_{\text{univ}}}

\newcommand{\ltl}[2]{\lta_{#1,#2}}
\newcommand{\lel}[2]{\lea_{#1,#2}}

\def\ltg{\vartriangleleft}
\def\leg{\trianglelefteq}

\newbox\noforkbox \newdimen\forklinewidth
\forklinewidth=0.3pt \setbox0\hbox{$\textstyle\smile$}
\setbox1\hbox to \wd0{\hfil\vrule width \forklinewidth depth-2pt
 height 10pt \hfil}
\wd1=0 cm \setbox\noforkbox\hbox{\lower 2pt\box1\lower
2pt\box0\relax}
\def\unionstick{\mathop{\copy\noforkbox}\limits}

\setbox0\hbox{$\textstyle\smile$}
\setbox1\hbox to \wd0{\hfil{\sl /\/}\hfil} \setbox2\hbox to
\wd0{\hfil\vrule height 10pt depth -2pt width
               \forklinewidth\hfil}
\newbox\doesforkbox
\setbox\doesforkbox\hbox{\lower 0pt\box1 \lower
2pt\box2\lower2pt\box0\relax}

\def\nunionstick{\mathop{\copy\doesforkbox}\limits}

\newcommand{\nf}{\unionstick}
\newcommand{\fork}{\nunionstick}

\newcommand{\nfs}[4]{#2 \nf_{#1}^{#4} #3}
\newcommand{\nnfs}[4]{#2 \fork_{#1}^{#4} #3}

\def\1nf{\unionstick^{(1)}}
\def\2nf{\unionstick^{(2)}}

\newcommand{\gtp}{\text{gtp}}

\newcommand{\gS}{\text{gS}}

\newcommand{\Sbs}{S^\text{bs}}

\newcommand{\hanf}[1]{h (#1)}

\newcommand{\Ll}{\mathbb{L}}

\newcommand{\cl}{\text{cl}}
\newcommand{\pre}{\text{pre}}
\newcommand{\F}{\mathcal{F}}

\newcommand{\LS}{\text{LS}}

\newcommand{\plus}[1]{{#1}_r}
\newcommand{\kappap}{\plus{\kappa}}

\newcommand{\Ksatpp}[2]{{#1}^{#2\text{-sat}}}
\newcommand{\Ksatp}[1]{\Ksatpp{K}{#1}}
\newcommand{\Ksat}{\Ksatp{\lambda}}

\newcommand{\Kmhp}[1]{\Ksatp{#1}}
\newcommand{\Kmh}[1]{\Kmhp{\lambda}}
\newcommand{\Ktuqp}[1]{K_{#1}^{3, \text{uq}}}

\newcommand{\K}{K}
\newcommand{\NF}{\text{NF}}

\newcommand{\Kupp}[1]{{#1}^{\text{up}}}
\newcommand{\Kup}{\Kupp{K}}

\newcommand{\islongp}[1]{{#1}^{\text{long}}}
\newcommand{\islong}{\islongp{\is}}

\newcommand{\slc}[1]{\bkappa_{#1}}

\newcommand{\clc}[1]{\kappa_{#1}}

\newcommand{\ssp}{\text{superstable}^+}

\title[Building independence relations in AECs]{Building independence relations in abstract elementary classes}
\makeindex

\author{Sebastien Vasey}
\thanks{This material is based upon work done while the author was supported by the Swiss National Science Foundation under Grant No.\ 155136.}
\email{sebv@cmu.edu}
\urladdr{http://math.cmu.edu/\textasciitilde svasey/}
\address{Department of Mathematical Sciences, Carnegie Mellon University, Pittsburgh, Pennsylvania, USA}

\date{\today }


\begin{document}

\begin{abstract} 
We study general methods to build forking-like notions in the framework of tame abstract elementary classes (AECs) with amalgamation. We show that whenever such classes are categorical in a high-enough cardinal, they admit a good frame: a forking-like notion for types of singleton elements. 

\begin{thm}[Superstability from categoricity]
  Let $K$ be a $(<\kappa)$-tame AEC with amalgamation. If $\kappa = \beth_\kappa > \LS (K)$ and $K$ is categorical in a $\lambda > \kappa$, then:

  \begin{itemize}
    \item $K$ is stable in any cardinal $\mu$ with $\mu \ge \kappa$.
    \item $K$ is categorical in $\kappa$.
    \item There is a type-full good $\lambda$-frame with underlying class $K_\lambda$.
  \end{itemize}
\end{thm}

Under more locality conditions, we prove that the frame extends to a global independence notion (for types of arbitrary length). 

\begin{thm}[A global independence notion from categoricity]\label{abstract-global}
  Let $K$ be a densely type-local, fully tame and type short AEC with amalgamation. If $K$ is categorical in unboundedly many cardinals, then there exists $\lambda \ge \LS (K)$ such that $K_{\ge \lambda}$ admits a global independence relation with the properties of forking in a superstable first-order theory.
\end{thm}

As an application, we deduce (modulo  an unproven claim of Shelah) that Shelah's eventual categoricity conjecture for AECs (without assuming categoricity in a successor cardinal) follows from the weak generalized continuum hypothesis and a large cardinal axiom.

\begin{cor}
  Assume $2^{\lambda} < 2^{\lambda^+}$ for all cardinals $\lambda$, as well as an unpublished claim of Shelah. If there exists a proper class of strongly compact cardinals, then any AEC categorical in \emph{some} high-enough cardinal is categorical in \emph{all} high-enough cardinals.
\end{cor}

\end{abstract}

\maketitle

\tableofcontents

\section{Introduction}

Independence (or forking) is a central notion of model theory. In the first-order setup, it was introduced by Shelah \cite{shelahfobook78} and is one of the main devices of his book. One can ask whether there is such a notion in the nonelementary context. In homogeneous model theory, this was investigated in \cite{hyttinen-lessmann} for the superstable case and \cite{buechler-lessmann} for the simple and stable cases. Some of their results were later generalized by Hyttinen and Kesälä \cite{finitary-aec} to tame and $\aleph_0$-stable finitary abstract elementary classes (AECs). For general\footnote{For a discussion of how the framework of tame AECs compare to other non first-order frameworks, see the introduction of \cite{sv-infinitary-stability-v6-toappear}.} AECs, the answer is still a work in progress. 

In \cite[Remark 4.9.1]{sh394} it was asked whether there is such a notion as forking in AECs.  In his book on AECs \cite{shelahaecbook}, Shelah introduced the concept of good $\lambda$-frames (a local independence notion for types of singletons) and some conditions are given for their existence. Shelah's main construction (see \cite[Theorem II.3.7]{shelahaecbook}) uses model-theoretic and set-theoretic assumptions: categoricity in two successive cardinals and principles like the weak diamond\footnote{Shelah claims to construct a good frame in ZFC in \cite[Theorem IV.4.10]{shelahaecbook} but he has to change the class and still uses the weak diamond to show his frame is $\omega$-successful.}. It has been suggested\footnote{The program of using tameness and amalgamation to prove Shelah's results in ZFC is due to Rami Grossberg and dates back to at least \cite{tamenessone}, see the introduction there.} that replacing Shelah's strong local model-theoretic hypotheses by the global hypotheses of amalgamation and tameness (a locality property for types introduced by Grossberg and VanDieren \cite{tamenessone}) should lead to better results with simpler proofs. Furthermore, one can argue that any ``reasonable'' AEC should be tame and have amalgamation, see for example the discussion in Section 5 of \cite{bg-v9}, and the introductions of \cite{tamelc-jsl} or \cite{tamenessone}. In particular, they follow from a large cardinal axiom and categoricity:

\begin{fact}\label{boney-lc}
  Let $\K$ be an AEC and let $\kappa > \LS (\K)$ be a strongly compact cardinal. Then:

  \begin{enumerate}
  \item \cite{tamelc-jsl} $\K$ is $(<\kappa)$-tame (in fact fully $(<\kappa)$-tame and short).
  \item \cite[Proposition 1.13]{makkaishelah}\footnote{This is stated there for the class of models of an $\Ll_{\kappa, \omega}$ theory but Boney \cite{tamelc-jsl} argues that the argument generalizes to any AEC $\K$ with $\LS (\K) < \kappa$.} If $\lambda > \beth_{\kappa + 1}$ is such that $\K$ is categorical in $\lambda$, then $\K_{\ge \kappa}$ has amalgamation.
  \end{enumerate}
\end{fact}

Examples of the use of tameness and amalgamation include \cite{b-k-vd-spectrum} (an upward stability transfer), \cite{lieberman2011} (showing that tameness is equivalent to a natural topology on Galois types being Hausdorff), \cite{tamenesstwo} (an upward categoricity transfer theorem, which can be combined with Fact \ref{boney-lc} and the downward transfer of Shelah \cite{sh394} to prove that Shelah's eventual categoricity conjecture for a successor follows from the existence of a proper class of strongly compact cardinals) and \cite{ext-frame-jml, tame-frames-revisited-v5, jarden-tameness-apal}, showing that good frames behave well in tame classes.

\cite{ss-tame-jsl} constructed good frames in ZFC using global model-theoretic hypotheses: tameness, amalgamation, and categoricity in a cardinal of high-enough cofinality. However we were unable to remove the assumption on the cofinality of the cardinal or to show that the frame was \emph{$\omega$-successful}, a key technical property of frames. Both in Shelah's book and in \cite{ss-tame-jsl}, the question of whether there exists a \emph{global} independence notion (for longer types) was left open. In this paper, we continue working in ZFC with tameness and amalgamation, and make progress toward these problems. Regarding the cofinality of the categoricity cardinal, we show that it is possible to take the categoricity cardinal to be high-enough: (Theorem \ref{categ-frame}):

\textbf{Theorem \ref{categ-frame}.}
  Let $K$ be a $(<\kappa)$-tame AEC with amalgamation. If $\kappa = \beth_\kappa > \LS (K)$ and $K$ is categorical in a $\lambda > \kappa$, then there is a type-full good $\lambda$-frame with underlying class $K_\lambda$.

As a consequence, the class $K$ above has several superstable-like properties: for all $\mu \ge \lambda$, $K$ is stable\footnote{The downward stability transfer from categoricity is an early result of Shelah \cite[Claim 1.7]{sh394}, but the upward transfer is new and improves on \cite[Theorem 7.5]{ss-tame-jsl}. In fact, the proof here is new even when $K$ is the class of models of a first-order theory.} in $\mu$ (this is also part of Theorem \ref{categ-frame}) and has a unique limit model of cardinality $\mu$ (by e.g.\ \cite[Corollary 6.9]{tame-frames-revisited-v5} and Remark \ref{lim-uq-rmk}). Since $K$ is stable in $\lambda$, the model of size $\lambda$ is saturated. Hence using Morley's omitting type theorem for AECs (see the proof of Theorem \ref{categ-frame} for the details), we deduce a downward categoricity transfer\footnote{\cite[Conclusion 5.1]{makkaishelah} proves a stronger conclusion under stronger assumptions (namely that $K$ is the class of models of an $\Ll_{\kappa, \omega}$ sentence, $\kappa$ a strongly compact cardinal).}:

\begin{cor}\label{downward-categ-transfer}
  Let $K$ be a $(<\kappa)$-tame AEC with amalgamation. If $\kappa = \beth_\kappa > \LS (K)$ and $K$ is categorical in a $\lambda > \kappa$, then $K$ is categorical in $\kappa$.
\end{cor}

We emphasize that already \cite{sh394} deduced such results assuming that the model of size $\lambda$ is saturated (or just $\kappa$-saturated so when $\cf{\lambda} \ge \kappa$ this follows). The new part here is showing that it \emph{is} saturated, even when $\cf{\lambda} < \kappa$.

The construction of the good frame in the proof of Theorem \ref{categ-frame} is similar to that in \cite{ss-tame-jsl} but uses local character of coheir (or $(<\kappa)$-satisfiability) rather than splitting. A milestone study of coheir in the nonelementary context is \cite{makkaishelah}, working in classes of models of an $\Ll_{\kappa, \omega}$-sentence, $\kappa$ a strongly compact cardinal. Makkai and Shelah's work was generalized to fully tame and short AECs in \cite{bg-v9}, and some results were improved in \cite{sv-infinitary-stability-v6-toappear}. Building on these works, we are able to show that under the assumptions above, coheir has enough superstability-like properties to apply the arguments of \cite{ss-tame-jsl}, and obtain that coheir restricted to types of length one in fact induces a good frame.

Note that coheir is a candidate for a global independence relation. In fact, one of the main result of \cite{bgkv-v3-toappear} is that it is canonical: if there is a global forking-like notion, it must be coheir. The paper assumes additionally that coheir has the extension property. Here, we prove that coheir is canonical without this assumption (Theorem \ref{canon-coheir}). We also obtain results on the canonicity of good frames. For example, any two type-full good $\lambda$-frames with the same \emph{categorical} underlying AEC must be the same (Theorem \ref{good-frame-cor}). This answers several questions from \cite{bgkv-v3-toappear}.

Using that coheir is global and (under categoricity) induces a good frame, we can use more locality assumptions to get that the good frame is $\omega$-successful:

\textbf{Theorem \ref{good-frame-succ}.}
  Let $K$ be a fully $(<\kappa)$-tame and short AEC. If $\LS (K) < \kappa = \beth_\kappa < \lambda = \beth_\lambda$, $\cf{\lambda} \ge \kappa$, and $K$ is categorical in a $\mu \ge \lambda$, then there exists an \emph{$\omega$-successful} type-full good $\lambda$-frame with underlying class $K_\lambda$.

We believe that the locality hypotheses in Theorem \ref{good-frame-succ} are reasonable: they follow from large cardinals (Fact \ref{boney-lc}) and slightly weaker assumptions can be derived from the existence of a global forking-like notion, see the discussion in Section \ref{main-thm-sec}.

Theorem \ref{good-frame-succ} can be used to build a global independence notion (Theorem \ref{main-thm} formalizes Theorem \ref{abstract-global} from the abstract). We assume one more locality hypothesis (dense type-locality) there. We suspect it can be removed, see the discussion in Section \ref{main-thm-sec}. Without dense type-locality, one still obtains an independence relation for types of length less than or equal to $\lambda$ (see Theorem \ref{good-frame-succ}). This improves several results from \cite{bg-v9} (see Section \ref{examples-sec} for a more thorough comparison).

These results bring us closer to solving one of the main test questions in the classification theory of abstract elementary classes\footnote{A version of Shelah's categoricity conjecture already appears as \cite[Open problem D.3(a)]{shelahfobook} and the statement here appears in \cite[Conjecture N.4.2]{shelahaecbook}, see \cite{grossberg2002} or the introduction to \cite{shelahaecbook} for history and motivation.}:

\begin{conjecture}[Shelah's eventual categoricity conjecture]
  An AEC that is categorical in a high-enough cardinal is categorical on a tail of cardinals.
\end{conjecture}

The power of $\omega$-successful frames comes from Shelah's analysis in Chapter III of his book. Unfortunately, Shelah could not prove the stronger results he had hoped for. Still, in \cite[Discussion III.12.40]{shelahaecbook}, he claims the following (a proof should appear in a future publication \cite{sh842}):

\begin{claim}\label{shelah-claim}
  Assume the weak generalized continuum hypothesis\footnote{$2^{\lambda} < 2^{\lambda^+}$ for all cardinals $\lambda$.} (WGCH). Let $K$ be an AEC such that there is an $\omega$-successful good $\lambda$-frame with underlying class $K_\lambda$. Write $\Ksatp{\lambda^{+\omega}}$ for the class of $\lambda^{+\omega}$-saturated models in $\K$. Then $\Ksatp{\lambda^{+\omega}}$ is categorical in some $\mu > \lambda^{+\omega}$ if and only if it is categorical in all $\mu > \lambda^{+\omega}$.
\end{claim}

Modulo this claim, we obtain the consistency of Shelah's eventual categoricity conjecture from large cardinals. This partially answers \cite[Question 6.14]{sh702}:

\begin{thm}\label{shelah-categ-cor}
  Assume Claim \ref{shelah-claim} and WGCH. 
  \begin{enumerate}
    \item Shelah's categoricity conjecture holds in fully tame and short AECs with amalgamation.
    \item If there exists a proper class of strongly compact cardinals, then Shelah's categoricity conjecture holds.
  \end{enumerate}
\end{thm}
\begin{proof}
  Let $K$ be an AEC.
  
  \begin{enumerate}
    \item Assume $K$ is fully $\LS (K)$-tame and short and has amalgamation. Pick $\kappa$ and $\lambda$ such that $\LS (K) < \kappa = \beth_\kappa < \lambda = \beth_\lambda$ and $\cf{\lambda} \ge \kappa$. By Theorem \ref{good-frame-succ}, there is an $\omega$-successful good $\lambda$-frame on $K_\lambda$. By Claim \ref{shelah-claim}, $\Ksatp{\lambda^{+\omega}}$ is categorical in all $\mu > \lambda^{+\omega}$. By Morley's omitting type theorem for AECs (see \cite[II.1.10]{sh394}), $\K$ is categorical in all $\mu \ge \beth_{\left(2^{\lambda^{+\omega}}\right)^+}$.
    \item Let $\kappa > \LS (K)$ be strongly compact. By \cite{tamelc-jsl}, $K$ is fully $(<\kappa)$-tame and short. By the methods of \cite[Proposition 1.13]{makkaishelah}, $K_{\ge \kappa}$ has amalgamation. Now apply the previous part to $K_{\ge \kappa}$.
  \end{enumerate}
\end{proof}
\begin{remark}
  Previous works (e.g.\ \cite{makkaishelah, sh394, tamenesstwo, tamelc-jsl}) all assume categoricity in a successor cardinal, and this was thought to be hard to remove. Here, we do \emph{not} need to assume categoricity in a successor.
\end{remark}

Note that \cite[Theorem IV.7.12]{shelahaecbook} is stronger than Theorem \ref{shelah-categ-cor} (since Shelah assumes only Claim \ref{shelah-claim}, WGCH, and amalgamation): unfortunately we were unable to verify Shelah's proof. The statement contains an error as it contradicts Morley's categoricity theorem.

This paper is organized as follows. In Section \ref{prelim-sec} we review some of the background. In Sections \ref{indep-rel-sec}-\ref{indep-calc-sec}, we introduce the framework with which we will study independence. In Sections \ref{dense-subac-sec}-\ref{good-sec}, we introduce the definition of a \emph{generator} for an independence relation and show how to use it to build good frames. In Section \ref{canon-sec}, we use the theory of generators to prove results on the canonicity of coheir and good frames. In Section \ref{sec-ss}, we use generators to study the definition of superstability implicit in \cite{shvi635} (and further studied in \cite{gvv-toappear-v1_2, ss-tame-jsl}). We derive superstability from categoricity and use it to construct good frames. In Section \ref{domin-sec}, we show how to prove a good frame is $\omega$-successful provided it is induced by coheir. In Sections \ref{long-frame-sec}-\ref{long-transfer-sec}, we show how to extend such a frame to a global independence relation. In Section \ref{main-thm-sec}, some of the main theorems are established. In Section \ref{examples-sec}, we give examples (existence of large cardinals, totally categorical classes, and fully $(<\aleph_0)$-tame and short AECs) where Theorem \ref{abstract-global} can be applied to derive the existence of a global independence relation.

Since this paper was first circulated (in December 2014), several improvements and applications were discovered. Threshold cardinals for the construction of a good frame are improved in \cite{vv-symmetry-transfer-v3}. Global independence relations are studied in the framework of universal classes in \cite{ap-universal-v9} and a categoricity transfer is obtained there (later improved to the full eventual categoricity conjecture in \cite{categ-universal-2-v1}). Global independence can also be used to build prime models over sets of the form $Ma$, for $M$ a saturated models \cite{prime-categ-v5}. Several of the results of this paper are exposed in \cite{bv-survey-v3}.

This paper was written while working on a Ph.D.\ thesis under the direction of Rami Grossberg at Carnegie Mellon University and I would like to thank Professor Grossberg for his guidance and assistance in my research in general and in this work specifically. 

I also thank Andrés Villaveces for sending his thoughts on my results and Will Boney for carefully reading this paper and giving invaluable feedback. I thank the referee for a thorough report that greatly helped to improve the presentation of this paper.

\section{Preliminaries}\label{prelim-sec}

We review some of the basics of abstract elementary classes and set some notation. The reader is advised to skim through this section quickly and go back to it as needed. We refer the reader to the preliminaries of \cite{sv-infinitary-stability-v6-toappear} for more motivation on some of the definitions below.

\subsection{Set-theoretic terminology}

\begin{notation}\index{interval of cardinals} \index{$\F$|see {interval of cardinals}}
  When we say that \emph{$\F$ is an interval of cardinals}, we mean that $\F = [\lambda, \theta)$ is the set of cardinals $\mu$ such that $\lambda \le \mu < \theta$. Here, $\lambda \le \theta$ are (possibly finite) cardinals except we also allow $\theta = \infty$.
\end{notation}

We will often use the following function:

\begin{defin}[Hanf function]\label{hanf-def}\index{Hanf function} \index{$\hanf{\lambda}$|see {Hanf function}}
  For $\lambda$ an infinite cardinal, define $\hanf{\lambda} := \beth_{(2^{\lambda})^+}$.
\end{defin}

Note that for $\lambda$ infinite, $\lambda = \beth_\lambda$ if and only if for all $\mu < \lambda$, $h (\mu) < \lambda$.

\begin{defin}\label{kappa-r-def}\index{$\kappap$}
  For $\kappa$ an infinite cardinal, let $\kappap$ be the least regular cardinal which is at least $\kappa$. That is, $\kappap$ is $\kappa^+$ if $\kappa$ is singular and $\kappa$ otherwise.
\end{defin}

\subsection{Abstract classes}

An abstract class\index{abstract class} \index{AC|see {abstract class}} (AC for short) is a pair $(K, \lea)$, where $K$ is a class of structures of the same (possibly infinitary) language and $\lea$ is an ordering on $K$ extending substructure and respecting isomorphisms. The definition is due to Rami Grossberg and appears in \cite{grossbergbook}. It is replicated in \cite[Definition 2.7]{sv-infinitary-stability-v6-toappear}. We use the same notation as in \cite{sv-infinitary-stability-v6-toappear}; for example $M \lta N$ \index{$\lta$} means $M \lea N$ and $M \neq N$. 

\begin{defin}\label{r-increasing-def}\index{$R$-increasing} \index{continuous} \index{increasing|see {$R$-increasing}} \index{strictly increasing|see {$R$-increasing}}
  Let $K$ be an abstract class and let $R$ be a binary relation on $K$. A sequence $\seq{M_i : i < \delta}$ of elements of $K$ is \emph{$R$-increasing} if for all $i < j < \delta$, $M_i R M_j$. When $R = \lea$, we omit it. \emph{Strictly increasing} means $\lta$-increasing. $\seq{M_i : i < \delta}$ is \emph{continuous} if for all limit $i < \delta$, $M_i = \bigcup_{j < i} M_j$.
\end{defin}

\begin{notation}\index{$K_{\F}$}\index{$K_\lambda$|see {$K_{\F}$}}\index{$K_{\ge \lambda}$|see {$K_{\F}$}}\index{$K_{<\lambda}$|see {$K_{\F}$}}
  For $K$ an abstract class, $\F$ an interval of cardinals, we write $K_{\F} := \{M \in K \mid \|M\| \in \F\}$. When $\F = \{\lambda\}$, we write $K_\lambda$ for $K_{\{\lambda\}}$. We also use notation like $K_{\ge \lambda}$, $K_{<\lambda}$, etc.
\end{notation}
\begin{defin}\index{in $\F$}
  An abstract class $K$ is \emph{in $\F$} if $K_{\F} = K$.
\end{defin}

We now recall the definition of an abstract elementary class (AEC) in $\F$, for $\F$ an interval of cardinal. Localizing to an interval is convenient when dealing with good frames and appears already (for $\F = \{\lambda\}$) in \cite[Definition 1.0.3.2]{jrsh875}. Confusingly, Shelah earlier on called an AEC in $\lambda$ a $\lambda$-AEC (in \cite[Definition II.1.18]{shelahaecbook}).

\begin{defin}\label{mu-aec-def}\index{abstract elementary class} \index{AEC|see {abstract elementary class}} \index{abstract elementary class in $\F$|see {abstract elementary class}} \index{AEC in $\F$|see {abstract elementary class}} \index{AEC in $\lambda$|see {abstract elementary class}} \index{abstract elementary class in $\lambda$|see {abstract elementary class}} \index{coherence|see {abstract elementary class}} \index{Tarski-Vaught axioms|see {abstract elementary class}} \index{Löwenheim-Skolem-Tarski axiom|see {abstract elementary class}} \index{Löwenheim-Skolem-Tarski number|see {abstract elementary class}} \index{$\LS (K)$|see {abstract elementary class}}
  For $\F = [\lambda, \theta)$ an interval of cardinals, we say an abstract class $K$ in $\F$ is an \emph{abstract elementary class} (AEC for short) \emph{in $\F$} if it satisfies:

  \begin{enumerate}
    \item Coherence: If $M_0, M_1, M_2$ are in $K$, $M_0 \lea M_2$, $M_1 \lea M_2$, and $|M_0| \subseteq |M_1|$, then $M_0 \lea M_1$.
    \item $L (K)$ is finitary.
    \item Tarski-Vaught axioms: If $\seq{M_i : i < \delta}$ is an increasing chain in $K$ and $\delta < \theta$, then $M_\delta := \bigcup_{i < \delta} M_i$ is such that:
      \begin{enumerate}
        \item $M_\delta \in K$.
        \item $M_0 \lea M_\delta$.
        \item\label{tv-last} If $M_i \lea N$ for all $i < \delta$, then $M_\delta \lea N$.
      \end{enumerate}
    \item Löwenheim-Skolem-Tarski axiom: There exists a cardinal $\mu \ge |L (K)| + \aleph_0$ such that for any $M \in K$ and any $A \subseteq |M|$, there exists $M_0 \lea M$ containing $A$ with $\|M_0\| \le |A| + \mu$. We write $\LS (K)$ (the \emph{Löwenheim-Skolem-Tarski number of $K$}) for the least such cardinal.
  \end{enumerate}

  When $\F = [0,\infty)$, we omit it. We say $K$ is an \emph{AEC in $\lambda$} if it is an AEC in $\{\lambda\}$. 
\end{defin}

Recall that an AEC in $\F$ can be made into an AEC: 

\begin{fact}[Lemma II.1.23 in \cite{shelahaecbook}]\label{kup-fact}
  If $K$ is an AEC in $\lambda := \LS (K)$, then there exists a unique AEC $K'$ such that $(K')_{\lambda} = K$ and $\LS (K') = \lambda$. The same holds if $K$ is an AEC in $\F$, $\F = [\lambda, \theta)$ (apply the previous sentence to $K_\lambda$).
\end{fact}
\begin{notation}\label{kup-def}\index{$\Kup$}
  Let $K$ be an AEC in $\F$ with $\F = [\lambda, \theta)$, $\lambda = \LS (K)$. Write $\Kup$ for the unique AEC $K'$ described by Fact \ref{kup-fact}.
\end{notation}

When studying independence, the following definition will be useful:

\begin{defin}\label{infty-aec-def}\index{coherent abstract class}\index{coherent abstract class in $\F$|see {coherent abstract class}}
  A \emph{coherent abstract class} in $\F$ is an abstract class in $\F$ satisfying the coherence property (see Definition \ref{mu-aec-def}).
\end{defin}

We also define the following weakening of the existence of a Löwenheim-Skolem-Tarski number:

\begin{defin}\label{lambda-closed-def}\index{closed}\index{$(<\lambda)$-closed|see {closed}}\index{$\lambda$-closed|see {closed}}\index{$(<\lambda)$-closed|see {closed}}
An abstract class $K$ is \emph{$(<\lambda)$-closed} if for any $M \in K$ and $A \subseteq |M|$ with $|A| < \lambda$, there exists $M_0 \lea M$ which contains $A$ and has size less than $\lambda$. $\lambda$-closed means $(<\lambda^+)$-closed.
\end{defin}
\begin{remark}
  An AEC $K$ is $(<\lambda)$-closed in every $\lambda > \LS (K)$.
\end{remark}

We will sometimes use the following consequence of Shelah's presentation theorem:

\begin{fact}[Conclusion I.1.11 in \cite{shelahaecbook}]\label{hanf-existence}
  Let $K$ be an AEC. If $K_{\ge \lambda} \neq \emptyset$ for every $\lambda < \hanf{\LS (K)}$, then $K$ has arbitrarily large models.
\end{fact}

As in the preliminaries of \cite{sv-infinitary-stability-v6-toappear}, we can define a notion of embedding for abstract classes and go on to define amalgamation\index{amalgamation}, joint embedding\index{joint embedding}, no maximal models\index{no maximal models}, Galois types\index{Galois type}, tameness, and type-shortness (that we will just call shortness). We give the definition of the last two (recall also Fact \ref{boney-lc} which says that under a large cardinal axiom any AEC is fully tame and short). 

\begin{defin}[Tameness and shortness]\index{tameness}\index{shortness}\index{tame|see {tameness}}\index{short|see {shortness}}\index{type-short|see {shortness}}\index{$(<\kappa)$-tame|see {tameness}}\index{fully $(<\kappa)$-tame and short|see {shortness}}
  Let $\K$ be an abstract class and let $\kappa$ be an infinite cardinal (most of the time $\kappa > \LS (\K)$).

  \begin{enumerate}
  \item \cite[Definition 3.2]{tamenessone} $\K$ is \emph{$(<\kappa)$-tame} if for any $M \in \K$ and any \emph{distinct} types $p, q \in \gS (M)$, there exists $A \subseteq |M|$ with $|A| < \kappa$ such that $p \rest A \neq q \rest A$.
  \item \cite[Definition 3.1]{tamelc-jsl} $\K$ is \emph{fully $(<\kappa)$-tame} if for any $M \in \K$, any ordinal $\alpha$, and any \emph{distinct} types $p, q \in \gS^{\alpha} (M)$ (so $p$ and $q$ can have \emph{any, possibly infinite, length}), there exists $A \subseteq |M|$ with $|A| < \kappa$ such that $p \rest A \neq q \rest A$.
  \item \cite[Definition 3.2]{tamelc-jsl} $\K$ is \emph{fully $(<\kappa)$-short} if for any $M \in \K$, any ordinal $\alpha$, and any \emph{distinct} types $p, q \in \gS^{\alpha} (M)$ , there exists $I \subseteq \alpha$ with $|I| < \kappa$ such that $p^I \neq q^I$.
  \end{enumerate}

  We say that $\K$ is \emph{fully $(<\kappa)$-tame and short} if it is fully $(<\kappa)$-tame and fully $(<\kappa)$-short. $\kappa$-tame means $(<\kappa^+)$-tame, and similarly for $\kappa$-short. We define local variations such as ``$(<\kappa)$-tame for types of length $\alpha$'' in a similar manner, see \cite[Definitions 3.1,3.2]{tamelc-jsl} or \cite[Definition 2.22]{sv-infinitary-stability-v6-toappear}. When we omit the parameter $\kappa$, we mean that there exists $\kappa$ such that the property holds.
\end{defin}

The following fact tells us that an AEC with amalgamation is a union of AECs with amalgamation and joint embedding. This a trivial observation from the definition of the \emph{diagram} of an AEC \cite[Definition I.2.2]{shelahaecbook}.

\begin{fact}[Lemma 16.14 in \cite{baldwinbook09}]\label{jep-partition}
    Let $K$ be an AEC with amalgamation. Then we can write $K = \bigcup_{i \in I} K^i$ where the $K^i$'s are disjoint AECs with $\LS (K^i) = \LS (K)$ and each $K^i$ has joint embedding and amalgamation.
\end{fact}

Using Galois types, a natural notion of saturation can be defined (see \cite[Definition 2.25]{sv-infinitary-stability-v6-toappear} for more explanation on the definition): 

\begin{defin}\label{sat-def}\index{saturated}\index{$\mu$-saturated|see {saturated}}\index{$\Ksatp{\mu}$|see {saturated}}
  Let $K$ be an abstract class and $\mu$ be an infinite cardinal. 
  \begin{enumerate}
    \item A model $M \in K$ is \emph{$\mu$-saturated} if for all $N \gea M$ and all $A_0 \subseteq |M|$ of size less than $\mu$, any $p \in \gS^{<\mu} (A_0; N)$ is realized inside $M$. When $\mu = \|M\|$, we omit it.
    \item We write $\Ksatp{\mu}$ for the class of $\mu$-saturated models of $K_{\ge \mu}$ (ordered by the strong substructure relation of $K$).
  \end{enumerate}
\end{defin}

\begin{remark}\label{sat-rmk}
By \cite[Lemma II.1.14]{shelahaecbook}, if $K$ is an AEC with amalgamation and $\mu > \LS (K)$, $M \in K$ is $\mu$-saturated if and only if for all $N \gea M$ and all $A_0 \subseteq |M|$ with $|A_0| < \mu$, any $p \in \gS (A_0; N)$ is realized in $M$. That is, it is enough to consider types of length 1 in the definition. We will use this fact freely.
\end{remark}

Finally, we recall there is a natural notion of stability in this context. This paper's definition follows \cite[Definition 2.23]{sv-infinitary-stability-v6-toappear} by defining what it means for a model to be stable and then specializing to the full class.

\begin{defin}[Stability]\label{stab-def}\index{stability} \index{stable|see {stability}} \index{stable in $\mu$|see {stability}} \index{$(<\alpha)$-stable in $\mu$|see {stability}} \index{$\alpha$-stable in $\mu$|see {stability}}
  Let $\alpha$ be a cardinal, $\mu$ be a cardinal. A model $N \in K$ is $(<\alpha)$-\emph{stable in $\mu$} if for all $A \subseteq |N|$ of size $\le \mu$, $|\gS^{<\alpha} (A; N)| \le \mu$. Here and below, $\alpha$-stable means $(< (\alpha^+))$-stable. We say ``stable'' instead of ``1-stable''.

  $K$ is \emph{$(<\alpha)$-stable in $\mu$} if every $N \in K$ is $(<\alpha)$-stable in $\mu$. $K$ is \emph{$(<\alpha)$-stable} if it is $(<\alpha)$-stable in unboundedly\footnote{Note (\cite[Corollary 6.4]{tamenessone}) that in a $\LS (K)$-tame AEC with amalgamation, this is equivalent to stability in \emph{some} cardinal.} many cardinals.
\end{defin}

A corresponding definition of the order property in AECs appears in \cite[Definition 4.3]{sh394}. For simplicity, we have removed one parameter from the definition.

\begin{defin}\label{def-op}\index{order property} \index{$\alpha$-order property|see {order property}} \index{$\alpha$-order property of length $\mu$|see {order property}} \index{order property of length $\mu$|see {order property}} \index{$(<\alpha)$-order property of length $\mu$|see {order property}}
  
  Let $\alpha$ and $\mu$ be cardinals and let $K$ be an abstract class. A model $M \in K$ has the \emph{$\alpha$-order property of length $\mu$} if there exists distinct $\seq{\ba_i : i < \mu}$ inside $M$ with $\ell (\ba_i) = \alpha$ for all $i < \mu$, such that for any $i_0 < j_0 < \mu$ and $i_1 < j_1 < \mu$, $\gtp (\ba_{i_0} \ba_{j_0} / \emptyset; M) \neq \gtp (\ba_{j_1} \ba_{i_1} / \emptyset; M)$.

  $M$ has the \emph{$(<\alpha)$-order property of length $\mu$} if it has the $\beta$-order property of length $\mu$ for some $\beta < \alpha$. $M$ has the \emph{order property of length $\mu$} if it has the $\alpha$-order property of length $\mu$ for some $\alpha$.

  \emph{$K$ has the $\alpha$-order property of length $\mu$} if some $M \in K$ has it, and similarly for other variations such as ``$K$ has the order property of length $\mu$''. \emph{$K$ has the  order property} if it has the order property for every length.
\end{defin}

When studying coheir, we will be interested in the $(<\kappa)$-order property of length $\kappa$, where $\kappa$ is a ``big'' cardinal (typically $\kappa = \beth_\kappa > \LS (\K)$). For completeness, we recall the definition of the following variation on the $(<\kappa)$-order property of length $\kappa$ that appears in \cite[Definition 4.2]{bg-v9} (but is adapted from a previous definition of Shelah, see there for more background):

\begin{defin}\label{def-weak-op}\index{weak order property}\index{weak $\kappa$-order property|see {weak order property}}
  Let $K$ be an AEC. For $\kappa > \LS (K)$, $K$ has the \emph{weak $\kappa$-order property} if there are $\alpha, \beta < \kappa$, $M \in K_{<\kappa}$, $N \gea M$, types $p \neq q \in \gS^{\alpha + \beta} (M)$, and sequences $\seq{\ba_i : i < \kappa}$, $\seq{\bb_i : i < \kappa}$ of distinct elements from $N$ so that for all $i, j < \kappa$:

  \begin{enumerate}
  \item $i \le j$ implies $\gtp (\ba_i \bb_j / M; N) = p$.
  \item $i > j$ implies $\gtp (\ba_i \bb_j / M; N) = q$.
  \end{enumerate}
\end{defin}

The following sums up all the results we will use about stability and the order property:

\begin{fact}\label{op-facts}
  Let $K$ be an AEC.

  \begin{enumerate}
    \item \cite[Lemma 4.8]{sv-infinitary-stability-v6-toappear} Let $\kappa = \beth_\kappa > \LS (K)$. The following are equivalent:
      \begin{enumerate}
        \item $K$ has the weak $\kappa$-order property.
        \item $K$ has the $(<\kappa)$-order property of length $\kappa$.
        \item $K$ has the $(<\kappa)$-order property.
      \end{enumerate}
    \item \cite[Theorem 4.13]{sv-infinitary-stability-v6-toappear} Assume $K$ is $(<\kappa)$-tame and has amalgamation. The following are equivalent:
      \begin{enumerate}
        \item $K$ is stable in some $\lambda \ge \kappa + \LS (K)$.
        \item There exists $\mu \le \lambda_0 < \hanf{\kappa + \LS (K)}$ such that $K$ is stable in any $\lambda \ge \lambda_0$ with $\lambda = \lambda^{<\mu}$.
        \item $K$ does not have the order property.
        \item $K$ does not have the $(<\kappa)$-order property.
      \end{enumerate}
    \item \cite[Theorem 4.5]{b-k-vd-spectrum} If $K$ is $\LS (K)$-tame, has amalgamation, and is stable in $\LS (K)$, then it is stable in $\LS (K)^+$.
  \end{enumerate}
\end{fact}

\subsection{Universal and limit extensions}

\begin{defin}\label{univ-def} \
  Let $K$ be an abstract class, $\lambda$ be a cardinal.
  \begin{enumerate}
    \item \index{universal over} \index{$\ltu$|see {universal over}}For $M, N \in K$, say $M \ltu N$ ($N$ is \emph{universal over $M$}) if and only if $M \lta N$ and whenever we have $M' \gea M$ such that $\|M'\| \le \|N\|$, then there exists $f: M' \xrightarrow[M]{} N$. Say $M \leu N$ if and only if $M = N$ or $M \ltu N$.
    \item \index{limit over} \index{$\ltl{\lambda}{\delta}$|see {limit over}} \index{$\lel{\lambda}{\delta}$|see {limit over}} \index{$(\lambda, \delta)$-limit|see {limit over}} \index{$\lambda$-limit|see {limit over}} For $M, N \in K$, $\lambda$ a cardinal and $\delta \le \lambda^+$, say $M \ltl{\lambda}{\delta} N$ ($N$ is \emph{($\lambda, \delta)$-limit over $M$}) if and only if $M \in K_\lambda$, $N \in K_{\lambda + |\delta|}$, $M \lta N$, and there exists $\seq{M_i : i \le \delta}$ increasing continuous such that $M_0 = M$, $M_i \ltu M_{i + 1}$ for all $i < \delta$, and $M_\delta = N$ if $\delta > 0$. Say $M \lel{\lambda}{\delta}$ if $M = N$ or $M \ltl{\lambda}{\delta} N$. We say $N \in K$ is a \emph{$(\lambda, \delta)$-limit model} if $M \ltl{\lambda}{\delta} N$ for some $M$. We say $N$ is \emph{$\lambda$-limit} if it is $(\lambda, \delta)$-limit for some limit $\delta < \lambda^+$. When $\lambda$ is clear from context, we omit it.
  \end{enumerate}
\end{defin}
\begin{remark}
  So for $M, N \in K_\lambda$, $M \ltl{\lambda}{0} N$ if and only if $M \lta N$, while $M \ltl{\lambda}{1}$ if and only if $M \ltu N$.
\end{remark}
\begin{remark}
  Variations on $\ltl{\lambda}{\delta}$ already appear as \cite[Definition 2.1]{sh394}. This paper's definition of being universal is different from the usual one (see e.g.\ \cite[Definition I.2.1.2]{vandierennomax}) because we ask only for $\|M'\| \le \|N\|$ rather than $\|M'\| = \|M\|$.
\end{remark}

The next fact is folklore.

\begin{fact}\label{ltl-basic-props}
 Let $K$ be an AC with amalgamation, $\lambda$ be an infinite cardinal, and $\delta \le \lambda^+$. Then:
  \begin{enumerate}
  \item \label{univ-trans} $M_0 \ltu M_1 \lea M_2$ and $\|M_1\| = \|M_2\|$ imply $M_0 \ltu M_2$.
  \item\label{univ-trans-2} $M_0 \lea M_1 \ltu M_2$ implies $M_0 \ltu M_2$.
  \item If $M_0 \in K_\lambda$, then $M_0 \lea M_1 \ltl{\lambda}{\delta} M_2$ implies $M_0 \ltl{\lambda}{\delta} M_2$.
  \item If $\delta < \lambda^+$, $K$ is an AEC in $\lambda = \LS (K)$ with no maximal models and stability in $\lambda$, then for any $M_0 \in K$ there exists $M_0'$ such that $M_0 \ltl{\lambda}{\delta} M_0'$.
  \end{enumerate}
\end{fact}
\begin{proof}
  All are straightforward, except perhaps the last which is due to Shelah. For proofs and references see \cite[Proposition 2.12]{ss-tame-jsl}.
\end{proof}

By a routine back and forth argument, we have:

\begin{fact}[Fact 1.3.6 in \cite{shvi635}]\label{lim-uq}
  Let $K$ be an AEC in $\lambda := \LS (K)$ with amalgamation. Let $\delta \le \lambda^+$ be a limit ordinal and for $\ell = 1,2$, let $\seq{M_i^\ell : i \le \delta}$ be increasing continuous with $M_0 := M_0^1 = M_0^2$ and $M_i^\ell \ltu M_{i + 1}^\ell$ for all $i < \delta$ (so they witness $M_0^\ell \ltl{\lambda}{\delta} M_\delta^\ell$).

  Then there exists $f: M_\delta^1 \cong_{M_0} M_\delta^2$ such that for all $i < \delta$, there exists $j < \delta$ such that $f[M_i^1] \lea M_j^2$ and $f^{-1}[M_i^2] \lea M_j^1$.
\end{fact}
\begin{remark}\label{lim-uq-rmk}
  Uniqueness of limit models that are \emph{not} of the same cofinality (i.e.\ the statement $M_0 \ltl{\lambda}{\delta} M_1$, $M_0 \ltl{\lambda}{\delta'} M_2$ implies $M_1 \cong_{M_0} M_2$ for any limit $\delta, \delta' < \lambda^+$) has been argued to be an important dividing line, akin to superstability in the first-order theory. See for example \cite{shvi635, vandierennomax, nomaxerrata, gvv-toappear-v1_2}. It is known to follow from the existence of a good $\lambda$-frame (see \cite[Lemma II.4.8]{shelahaecbook}, or \cite[Theorem 9.2]{ext-frame-jml} for a detailed proof).
\end{remark}

We could not find a proof of the next result in the literature, so we included one here.

\begin{lem}\label{lem-univ} 
  Let $K$ be an AEC with amalgamation. Let $\delta$ be a (not necessarily limit) ordinal and assume $(M_i)_{i \le \delta}$ is increasing continuous with $M_i \ltu M_{i + 1}$ for all $i < \delta$. Then $M_i \ltu M_\delta$ for all $i < \delta$.
\end{lem}
\begin{proof} \
By induction on $\delta$. If $\delta = 0$, there is nothing to do. If $\delta = \alpha + 1$ is a successor, let $i < \delta$. We know $M_i \lea M_\alpha$. By hypothesis, $M_\alpha \ltu M_\delta$. By Fact \ref{ltl-basic-props}.(\ref{univ-trans-2}), $M_i \ltu M_\delta$. Assume now $\delta$ is limit. In that case it is enough to show $M_0 \ltu M_\delta$. By the induction hypothesis, we can further assume that $\delta = \cf{\delta}$. Let $N \gea M_0$ be given such that $\mu := \|N\| \le \|M_\delta\|$, and $N$, $M_\delta$ are inside a common model $\bigN$. If $\mu < \|M_\delta\|$, then there exists $i < \delta$ such that $\mu \le \|M_i\|$, and we can use the induction hypothesis, so assume $\mu = \|M_\delta\|$. We can further assume $\mu > \|M_0\|$, for otherwise $N$ directly embeds into $M_1$ over $M_0$. The $M_i$s show that $\gamma := \cf{\mu} \le \delta$.   Let $\seq{N_i : i \le \gamma}$ be increasing continuous such that for all $i < \gamma$.

  \begin{enumerate}
    \item $N_0 = M_0$.
    \item $N_\gamma = N$.
    \item $\|N_i\| < \mu$.
  \end{enumerate}

  This exists since $\gamma = \cf{\mu}$.

  Build $\seq{f_i : i \le \gamma}$, increasing continuous such that for all $i < \gamma$, $f_i : N_i \xrightarrow[M_0]{} M_{k_i}$ for some $k_i < \delta$. This is enough, since then $f_\gamma$ will be the desired embedding. This is possible: For $i = 0$, take $f_0 := \id_{M_0}$. At limits, take unions: since $\delta$ is regular and $\gamma \le \delta$, $k_j < \delta$ for all $j < i < \gamma$ implies $k_i := \sup_{j < i} k_j < \delta$.

  Now given $i = j + 1$, first pick $k = k_j < \delta$ such that $f_j[N_j] \lea M_k$. Such a $k$ exists by the induction hypothesis. Find $k' > k$ such that $\|N_i\| \le \|M_{k'}\|$. This exists since $\|N_i\| < \mu = \|M_\delta\|$. Now by the induction hypothesis, $M_k \ltu M_{k'}$, so by Fact \ref{ltl-basic-props}.(\ref{univ-trans-2}), $f_j[N_j] \ltu M_{k'}$. Hence by some renaming, we can extend $f_i$ as desired.
\end{proof}

\begin{remark}
  $(K, \leu)$ is in general not an AEC as it may fail the Löwenheim-Skolem-Tarski axiom, the coherence axiom, and (\ref{tv-last}) in the Tarski-Vaught axioms of Definition \ref{mu-aec-def}.
\end{remark}

\section{Independence relations}\label{indep-rel-sec}

Since this section mostly lists definitions, the reader already familiar with independence (in the first-order context) may want to skip it and refer to it as needed. We would like a general framework in which to study independence in abstract elementary classes. One such framework is Shelah's good $\lambda$-frames \cite[Section II.6]{shelahaecbook}. Another is given by the definition of independence relation in \cite[Definition 3.1]{bgkv-v3-toappear} (itself adapted from \cite[Definition 3.3]{bg-v9} which can be traced back to the work of Makkai and Shelah \cite{makkaishelah}). Both definitions describe a relation ``$p$ does not fork over $M$'' for $p$ a Galois type over $N$ and $M \lea N$ and require it to satisfy some properties. 

In \cite{bgkv-v3-toappear}, it is also shown how to ``close'' such a relation to obtain a relation ``$p$ does not fork over $M$'' when $p$ is a type over an arbitrary set. We find that starting with such a relation makes the statement of symmetry transparent, and hence makes several proofs easier. Perhaps even more importantly, we can be very precise\footnote{Assume for example that $\s$ is a good-frame on a class of saturated models of an AEC $K$. Let $\seq{M_i : i < \delta}$ be an increasing chain of saturated models. Let $M_\delta := \bigcup_{i < \delta} M_i$ and let $p \in \gS (M_\delta)$. We would like to say that there is $i < \delta$ such that $p$ does not fork over $M_i$ but we may not know that $M_\delta$ is saturated, so maybe forking is not even defined for types over $M_\delta$. However if the forking relation were defined for types over sets, there would be no problem.}  when dealing with chain local character properties (see Definition \ref{loc-card-def}).

The definition in \cite{bgkv-v3-toappear} is not completely adequate for our purpose, however. There it is assumed that everything is contained inside a big homogeneous monster model. While we will always assume amalgamation, assuming the existence of a monster model is still problematic when for example we want to study independence over models of size $\lambda$ only (the motivation for good $\lambda$-frames, note that Shelah's definition does not assume the existence of a monster model). We also allow working inside more general classes than AECs: coherent abstract classes (recall Definition \ref{infty-aec-def}). This is convenient when working with classes of saturated models (see for example the study of weakly good independence relation in Section \ref{ns-canon}), but note that in general we may not be able to build a monster model there.

We also give a more general definition than \cite{bgkv-v3-toappear}, as we do not assume that everything happens in a big homogeneous monster model, and we allow working inside coherent abstract classes (recall Definition \ref{infty-aec-def}) rather than only abstract elementary classes. The later feature is convenient when working with classes of saturated models.

This means that we always have to carry over an ambient model $N$ that may shrink or be extended as needed. Although this makes the notation slightly heavier, it does not cause any serious technical difficulties. At first reading, the reader may simply want to ignore $N$ and assume everything takes place inside a monster model. 

Because we quote extensively from \cite{shelahaecbook}, which deals with frames, and also because it is sometimes convenient to ``forget'' the extension of the relation to arbitrary sets, we will still define frames and recall their relationship with independence relations over sets.

\subsection{Frames}
Shelah's definition of a pre-frame appears in \cite[Definition III.0.2.1]{shelahaecbook} and is meant to axiomatize the bare minimum of properties a relation must satisfy in order to be a meaningful independence notions. 

We make several changes: we do not mention basic types (we have no use for them), so in Shelah's terminology our pre-frames will be \emph{type-full}\index{type-full}. In fact, it is notationally convenient for us to define our frame on every type, not just the nonalgebraic ones. The \emph{disjointness} property (see Definition \ref{indep-props-def}) tells us that the frame behaves trivially on the algebraic types. We do not require it (as it is not required in \cite[Definition 3.1]{bgkv-v3-toappear}) but it will hold of all frames we consider. 

We require that the class on which the independence relation operates has amalgamation\footnote{This is required in Shelah's definition of good frames, but not in his definition of pre-frames.}, and we do not require that the base monotonicity property holds (this is to preserve the symmetry between right and left properties in the definition. All the frames we consider will have base monotonicity). Finally, we allow the size of the models to lie in an interval rather than just be restricted to a single cardinal as Shelah does. We also parametrize on the length of the types. This allows more flexibility and was already the approach favored in \cite{ss-tame-jsl, tame-frames-revisited-v5}.

\begin{defin}\label{pre-frame-def}\index{frame} \index{pre-frame|see {frame}} \index{pre-$(<\alpha, \F)$-frame|see {frame}} \index{$\s$|see {frame}} \index{invariance} \index{monotonicity} \index{ambient monotonicity} \index{normality}\index{type-full}
  Let $\F = [\lambda, \theta)$ be an interval of cardinals with $\aleph_0 \le \lambda < \theta$, $\alpha \le \theta$ be a cardinal or $\infty$.

  A \emph{type-full pre-$(<\alpha, \F)$-frame} is a pair $\s = (K, \nf)$, where:

  \begin{enumerate}
  \item $K$ is a coherent abstract class in $\F$ (see Definition \ref{infty-aec-def}) with amalgamation.
  \item $\nf$ is a relation on quadruples of the form $(M_0, A, M, N)$, where $M_0 \lea M \lea N$ are all in $K$, $A \subseteq |N|$ is such that $|A \backslash |M_0|| < \alpha$. We write $\nf(M_0, A, M, N)$ or $\nfs{M_0}{A}{M}{N}$ instead of $(M_0, A, M, N) \in \nf$.

  \item The following properties hold:
    \begin{enumerate}
    \item \underline{Invariance}: If $f: N \cong N'$ and $\nfs{M_0}{A}{M}{N}$, then $\nfs{f[M_0]}{f[A]}{f[M]}{N'}$.
    \item \underline{Monotonicity}: Assume $\nfs{M_0}{A}{M}{N}$. Then:
      \begin{enumerate}
      \item Ambient monotonicity: If $N' \gea N$, then $\nfs{M_0}{A}{M}{N'}$. If $M \lea N_0 \lea N$ and $A \subseteq |N_0|$, then $\nfs{M_0}{A}{M}{N_0}$.
      \item Left and right monotonicity: If $A_0 \subseteq A$, $M_0 \lea M' \lea M$, then $\nfs{M_0}{A_0}{M'}{N}$.
      \end{enumerate}
    \item \underline{Left normality}: If $\nfs{M_0}{A}{M}{N}$, then\footnote{For sets $A$ and $B$, we sometimes write $AB$ instead of $A \cup B$.} $\nfs{M_0}{AM_0}{M}{N}$.
    \end{enumerate}
  \end{enumerate}
  
  When $\alpha$ or $\F$ are clear from context or irrelevant, we omit them and just say that $\s$ is a pre-frame (or just a frame). We may omit the ``type-full''. A $(\le \alpha)$-frame is just a $(<\alpha^+)$-frame. We might omit $\alpha$ when $\alpha = 2$ (i.e.\ $\s$ is a $(\le 1)$-frame) and we might talk of a $\lambda$-frame or a $(\ge \lambda)$-frame instead of a $\{\lambda\}$-frame or a $[\lambda, \infty)$-frame.
\end{defin}
\begin{notation}\label{frame-objects-notation}\index{$K_{\s}$}\index{$\nf_{\s}$}\index{$\nf$} \index{$\alpha_{\s}$} \index{$\F_{\s}$} \index{$\lambda_{\s}$} \index{$\theta_{\s}$}
  For $\s = (K, \nf)$ a pre-$(<\alpha, \F)$-frame with $\F = [\lambda, \theta)$, write $K_{\s} := K$, $\nf_{\s} := \nf$, $\alpha_{\s} := \alpha$, $\F_{\s} = \F$, $\lambda_{\s} := \lambda$, $\theta_{\s} := \theta$. Note that pedantically, $\alpha$, $\F$, and $\theta$ should be part of the data of the frame in order for this notation to be well-defined but we ignore this detail.
\end{notation}

\begin{notation}\label{nf-notation}\index{$\nf (M_0, \ba, M, N)$} \index{$\nfs{M_0}{\ba}{M}{N}$}\index{$\s$-fork}\index{does not fork}\index{forking}
  For $\s = (K, \nf)$ a pre-frame, we write $\nf (M_0, \ba, M, N)$ or $\nfs{M_0}{\ba}{M}{N}$ for $\nfs{M_0}{\text{ran} (\ba)}{M}{N}$ (similarly when other parameters are sequences). When $p \in \gS^{<\infty} (M)$, we say $p$ \emph{does not $\s$-fork over $M_0$} (or just \emph{does not fork over $M_0$} if $\s$ is clear from context) if $\nfs{M_0}{\ba}{M}{N}$ whenever $p = \gtp (\ba / M; N)$ (using monotonicity and invariance, it is easy to check that this does not depend on the choice of representatives).
\end{notation}
\begin{remark}
  In the definition of a pre-frame given in \cite[Definition 3.1]{tame-frames-revisited-v5}, the left hand side of the relation $\nf$ is a sequence, not just a set. Here, we simply assume outright that the relation is defined so that order does not matter.
\end{remark}
\begin{remark}
  We can go back and forth from this paper's definition of pre-frame to Shelah's. We sketch how. From a pre-frame $\s$ in our sense (with $K_{\s}$ an AEC), we can let $\Sbs (M)$ be the set of nonalgebraic $p \in \gS (M)$ that do not $\s$-fork over $M$. Then restricting $\nf_{\s}$ to the basic types we obtain (assuming that $\s$ has base monotonicity, see Definition \ref{indep-props-def}) a pre-frame in Shelah's sense. From a pre-frame $(K, \nf, \Sbs)$ in Shelah's sense (where $K$ has amalgamation), we can extend $\nf$ by specifying that algebraic and basic types do not fork over their domains. We then get a pre-frame $\s$ in our sense with base monotonicity and disjointness.
\end{remark}

\subsection{Independence relations}

We now give a definition for an independence notion that also takes sets on the right hand side.

\begin{defin}[Independence relation]\label{indep-rel-def}\index{independence relation}\index{$(<\alpha, \F, <\beta)$-independence relation|see {independence relation}} \index{$(<\alpha, \F)$-independence relation|see {independence relation}} \index{$\is$|see {independence relation}} \index{invariance} \index{monotonicity} \index{ambient monotonicity} \index{normality}
  Let $\F = [\lambda, \theta)$ be an interval of cardinals with $\aleph_0 \le \lambda < \theta$, $\alpha, \beta \le \theta$ be cardinals or $\infty$.
    A \emph{$(<\alpha, \F, <\beta)$-independence relation} is a pair $\is = (K, \nf)$, where:

    \begin{enumerate}
      \item $K$ is a coherent abstract class in $\F$ with amalgamation.
      \item $\nf$ is a relation on quadruples of the form $(M, A, B, N)$, where $M \lea N$ are all in $K$, $A \subseteq |N|$ is such that $|A \backslash |M|| < \alpha$ and $B \subseteq |N|$ is such that $|B \backslash |M|| < \beta$. We write $\nf(M, A, B, N)$ or $\nfs{M}{A}{B}{N}$ instead of $(M, A, B, N) \in \nf$.
      \item The following properties hold:

        \begin{enumerate}
        \item \underline{Invariance}: If $f: N \cong N'$ and $\nfs{M}{A}{B}{N}$, then $\nfs{f[M]}{f[A]}{f[B]}{N'}$.
        \item \underline{Monotonicity}: Assume $\nfs{M}{A}{B}{N}$. Then:
          \begin{enumerate}
          \item Ambient monotonicity: If $N' \gea N$, then $\nfs{M}{A}{B}{N'}$. If $M \lea N_0 \lea N$ and $A \cup B \subseteq |N_0|$, then $\nfs{M}{A}{B}{N_0}$.
          \item Left and right monotonicity: If $A_0 \subseteq A$, $B_0 \subseteq B$, then $\nfs{M}{A_0}{B_0}{N}$.
          \end{enumerate}
        \item \underline{Left and right normality}: If $\nfs{M}{A}{B}{N}$, then $\nfs{M}{AM}{BM}{N}$.
        \end{enumerate}
    \end{enumerate}

    We adopt the conventions described at the end of Definition \ref{pre-frame-def}. For example, a $(\le \alpha, \F, <\beta)$-independence relation is just a $(<\alpha^+, \F, <\beta)$-independence relation.

    When $\beta = \theta$, we omit it. More generally, when $\alpha$, $\beta$ are clear from context or irrelevant, we omit them and just say that $\is$ is an independence relation.
\end{defin}

\begin{notation}\index{$\K_{\is}$}\index{$\nf_{\is}$}\index{$\alpha_{\is}$}\index{$\beta_{\is}$}\index{$\F_{\is}$}\index{$\lambda_{\is}$}\index{$\theta_{\is}$}\index{$\is$-fork}
  We adopt the same notational conventions as for pre-frames: $\K_{\is}$, $\nf_{\is}$, $\alpha_{\is}$, $\beta_{\is}$, $\F_{\is}$, $\lambda_{\is}$, $\theta_{\is}$ are defined as in Notation \ref{frame-objects-notation} and $p$ does not $\is$-fork over $M_0$ is defined as in \ref{nf-notation}.
\end{notation}

\begin{remark}
  It seems that in every case of interest $\beta = \theta$ (this will always be the case in the next sections of this paper). We did not make it part of the definition to avoid breaking the symmetry between $\alpha$ and $\beta$ (and hence make it possible to define the dual independence relation and the left version of a property, see Definitions \ref{dual-def} and \ref{left-p-def}). Note also that the case $\alpha = \theta = \infty$ is of particular interest in Section \ref{long-transfer-sec}.
\end{remark}

Before listing the properties independence relations and frames could have, we discuss how to go from one to the other. The $\cl$ operation is called the \emph{minimal closure} in \cite[Definition 3.4]{bgkv-v3-toappear}.

\begin{defin}\label{cl-def}\index{minimal closure}\index{$\cl (\s)$|see {minimal closure}} \index{$\pre (\is)$} \
  \begin{enumerate}
    \item Given a pre-frame $\s := (K, \nf)$, let $\cl (\s) := (K, \nf^{\cl})$, where $\nf^{\cl} (M, A, B, N)$ if and only if $M \lea N$, $|B| < \theta_{\s}$, and there exists $N' \gea N$, $M' \gea M$ containing $B$ such that $\nf (M, A, M', N')$.
    \item Given a $(<\alpha, \F)$-independence relation $\is = (K, \nf)$ let $\pre (\is) := (K, \nf^{\pre})$, where $\nf^{\pre} (M, A, M', N)$ if and only if $M \lea M' \lea N$ and $\nf (M, A, M', N)$.
  \end{enumerate}
\end{defin}
\begin{remark} \
  \begin{enumerate}
    \item If $\is$ is a $(<\alpha, \F)$-independence relation, then $\pre (\is)$ is a pre-$(<\alpha, \F)$-frame.
    \item If $\s$ is a pre-$(<\alpha, \F)$-frame, then $\cl (\s)$ is a $(<\alpha, \F)$-independence relation and $\pre (\cl (\s)) = \s$.
  \end{enumerate}
\end{remark}

Other properties of $\cl$ and $\pre$ are given by Proposition \ref{cl-basics}.

\begin{remark}
  The reader may wonder why we do not assume that every independence relation is the closure of a pre-frame, i.e.\ why we do not assume that for any independence relation $\is = (\K, \nf)$, if $\nfs{M}{A}{B}{N}$, there exists $N' \gea N$ and $M' \lea N'$ with $M \lea M'$ such that $B \subseteq |M'|$ and $\nfs{M}{A}{M'}{N'}$ (this can be written abstractly as $\is = \cl (\pre (\is)$)? This would allow us to avoid the redundancies between the definition of an independence relation and that of a pre-frame. However, several interesting independence notions do not satisfy that property (see \cite[Section 3.2]{bgkv-v3-toappear}). Further, it is not clear that the property $\is = \cl (\pre (\is))$ transfers upward (see Definition \ref{is-up-def}). Therefore we prefer to be agnostic and not require it.
\end{remark}

Next, we give a long list of properties that an independence relation may or may not have. Most are classical and already appear for example in \cite{bgkv-v3-toappear}. We give them here again both for the convenience of the reader and because their definition is sometimes slightly modified compared to \cite{bgkv-v3-toappear} (for example, symmetry there is called right full symmetry here, and some properties like uniqueness and extensions are complicated by the fact we do not work in a monster model). They will be used throughout this paper (for example, Section \ref{indep-calc-sec} discusses implications between the properties).

\begin{defin}[Properties of independence relations]\label{indep-props-def}
  Let $\is := (K, \nf)$ be a $(<\alpha, \F, <\beta)$-independence relation.
  \begin{enumerate}
    \item\index{disjointness} $\is$ has \emph{disjointness} if $\nfs{M}{A}{B}{N}$ implies $A \cap B \subseteq |M|$.
    \item\index{symmetry} $\is$ has \emph{symmetry} if $\nfs{M}{A}{B}{N}$ implies that for all\footnote{Why not just take $B_0 = B$? Because the definition of $\nf$ requires that the left hand side has size less than $\alpha$. Similarly for right full symmetry.} $B_0 \subseteq B$ of size less than $\alpha$ and all $A_0 \subseteq A$ of size less than $\beta$, $\nfs{M}{B_0}{A_0}{N}$.
    \item\index{full symmetry} $\is$ has \emph{right full symmetry} if $\nfs{M}{A}{B}{N}$ implies that for all $B_0 \subseteq B$ of size less than $\alpha$ and all $A_0 \subseteq A$ of size less than $\beta$, there exists $N' \gea N$, $M' \gea M$ containing $A_0$ such that $\nfs{M}{B_0}{M'}{N'}$.
    \item\index{monotonicity}\index{base monotonicity}\label{base-monot-def} $\is$ has \emph{right base monotonicity} if $\nfs{M}{A}{B}{N}$ and $M \lea M' \lea N$, $|M'| \subseteq B \cup |M|$ implies $\nfs{M'}{A}{B}{N}$.
    \item\index{existence} $\is$ has \emph{right existence} if $\nfs{M}{A}{M}{N}$ for any $A \subseteq |N|$ with $|A| < \alpha$.
    \item\index{uniqueness} $\is$ has \emph{right uniqueness} if whenever $M_0 \lea M \lea N_\ell$, $\ell = 1,2$, $|M_0| \subseteq B  \subseteq |M|$, $q_\ell \in \gS^{<\alpha} (B; N_\ell)$, $q_1 \rest M_0 = q_2 \rest M_0$, and $q_\ell$ does not fork over $M_0$, then $q_1 = q_2$.
    \item\index{extension} $\is$ has \emph{right extension} if whenever $p \in \gS^{<\alpha} (M B; N)$ does not fork over $M$ and $B \subseteq C \subseteq |N|$ with $|C| < \beta$, there exists $N' \gea N$ and $q \in \gS^{<\alpha} (M C; N')$ extending $p$ such that $q$ does not fork over $M$.
    \item\index{independent amalgamation} $\is$ has \emph{right independent amalgamation} if $\alpha > \lambda$, $\beta = \theta$, and\footnote{Note that even though the next condition is symmetric, the condition on $\alpha$ and $\beta$ make the left version of the property different from the right.} whenever $M_0 \lea M_\ell$ are in $K$, $\ell = 1,2$, there exists $N \in K$ and $f_\ell : M_\ell \xrightarrow[M_0]{} N$ such that $\nfs{M_0}{f_1[M_1]}{f_2[M_2]}{N}$.
    \item\index{witness property}\index{model witness property|see {witness property}}\index{$(<\kappa)$-witness property|see {witness property}}\index{$\lambda$-witness property|see {witness property}}  $\is$ has the \emph{right $(<\kappa)$-model-witness property} if whenever $M \lea M' \lea N$, $||M'| \backslash |M|| < \beta$, $A \subseteq |N|$, and $\nfs{M}{A}{B_0}{N}$ for all $B_0 \subseteq |M'|$ of size less than $\kappa$, then $\nfs{M}{A}{M'}{N}$. $\is$ has the \emph{right $(<\kappa)$-witness property} if this is true when $M'$ is allowed to be an arbitrary set. The \emph{$\lambda$-[model-]witness property} is the $(<\lambda^+)$-[model-]witness property.
    \item\index{transitivity}\index{strong transitivity|see {transitivity}} $\is$ has \emph{right transitivity} if whenever $M_0 \lea M_1 \lea N$, $\nfs{M_0}{A}{M_1}{N}$ and $\nfs{M_1}{A}{B}{N}$ implies $\nfs{M_0}{A}{B}{N}$. \emph{Strong right transitivity} is the same property when we do not require $M_0 \lea M_1$.
    \item\index{full model continuity}\index{model continuity|see {full model continuity}} $\is$ has \emph{right full model-continuity} if $K$ is an AEC in $\F$, $\alpha > \lambda$, $\beta = \theta$, and whenever $\seq{M_i^\ell : i \le \delta}$ is increasing continuous with $\delta$ limit, $\ell \le 3$, for all $i < \delta$, $M_i^0 \lea M_i^\ell \lea M_i^3$, $\ell = 1,2$, $\|M_\delta^1\| < \alpha$, and $\nfs{M_i^0}{M_i^1}{M_i^2}{M_i^3}$ for all $i < \delta$, then $\nfs{M_\delta^0}{M_\delta^1}{M_\delta^2}{M_\delta^3}$.
    \item\index{weak chain local character} \emph{Weak chain local character} is a technical property used to generate weakly good independence relations, see Definition \ref{weak-chain-lc-def}.
  \end{enumerate}

  Whenever this makes sense, we similarly define the same properties for pre-frames.
\end{defin}

Note that we have defined the right version of the asymmetric properties. One can define a left version by looking at the \emph{dual independence relation.}

\begin{defin}\label{dual-def}\index{dual independence relation}\index{$\is^d$|see {dual independence relation}}
  Let $\is := (K, \nf)$ be a $(<\alpha, \F, <\beta)$-independence relation. Define the \emph{dual independence relation} $\is^d := (K, \nf^d)$ by $\nf^d (M, A, B, N)$ if and only if $\nf (M, B, A, N)$.
\end{defin}

\begin{remark} \
  \begin{enumerate}
    \item If $\is$ is a $(<\alpha, \F, <\beta)$-independence relation, then $\is^d$ is a $(<\beta, \F, <\alpha)$-independence relation and $\left(\is^d\right)^d = \is$.
    \item Let $\is$ be a $(<\alpha, \F, <\alpha)$-independence relation. Then $\is$ has symmetry if and only if $\is = \is^d$.
  \end{enumerate}
\end{remark}

\begin{defin}\label{left-p-def}\index{left $P$ ($P$ a property of an independence relation)}\index{right $P$ ($P$ a property of an independence relation)}
  For $P$ a property, we will say \emph{$\is$ has left $P$} if $\is^d$ has right $P$. When we omit left or right, we mean the right version of the property.
\end{defin}

\begin{defin}[Locality cardinals]\label{loc-card-def}\index{locality cardinals}\index{chain local character|see {locality cardinals}}
  Let $\is = (K, \nf)$ be a $(<\alpha, \F)$-independence relation, $\F = [\lambda, \theta)$. Let $\alpha_0 < \alpha$ be such that $|\alpha_0|^+ < \theta$.

  \begin{enumerate}
    \item \index{$\slc{<\alpha} (\is)$|see {locality cardinals}} \index{$\slc{\alpha} (\is)$|see {locality cardinals}} Let $\slc{\alpha_0} (\is)$ be the minimal cardinal $\mu \ge |\alpha_0|^+ + \lambda^+$ such that for any $M \lea N$ in $K$, any $A \subseteq |N|$ with $|A| \le \alpha_0$, there exists $M_0 \lea M$ in $K_{<\mu}$ with $\nfs{M_0}{A}{M}{N}$. When $\mu$ does not exist, we set $\slc{\alpha_0} (\is) = \infty$.
      \item \index{$\clc{<\alpha} (\is, R)$|see {locality cardinals}} \index{$\clc{<\alpha} (\is)$|see {locality cardinals}} \index{$\clc{\alpha} (\is, R)$|see {locality cardinals}} \index{$\clc{\alpha} (\is)$|see {locality cardinals}} For $R$ a binary relation on $K$, Let $\clc{\alpha_0} (\is, R)$ be the minimal cardinal $\mu \ge |\alpha_0|^+ + \aleph_0$ such that for any regular $\delta \ge \mu$, any $R$-increasing (recall Definition \ref{r-increasing-def}) $\seq{M_i : i < \delta}$ in $K$, any $N \in \K$ extending all the $M_i$'s, and any $A \subseteq |N|$ of size $\le \alpha_0$, there exists $i < \delta$ such that $\nfs{M_i}{A}{M_\delta}{N}$. Here, we have set\footnote{Recall that $K$ is only a coherent abstract class, so may not be closed under unions of chains of length $\delta$. Thus we think of $M_\delta$ as a set.} $M_\delta := \bigcup_{i < \delta} M_i$. When $R = \lea$, we omit it. When $\mu$ does not exist or $\plus{\mu} \ge \theta$, we set $\clc{\alpha_0} (\is) = \infty$.
  \end{enumerate}
  \index{$\slc{\alpha} (\nf)$|see {locality cardinals}}
  
  When $K$ is clear from context, we may write $\slc{\alpha_0} (\nf)$. For $\alpha_0 \le \alpha$, we also let $\slc{<\alpha_0} (\is) := \sup_{\alpha_0' < \alpha_0} \slc{\alpha_0'} (\is)$. Similarly define $\clc{<\alpha_0}$. 

  \index{$\slc{\alpha} (\s)$|see {locality cardinals}} \index{$\clc{\alpha} (\s)$|see {locality cardinals}}
  
  We similarly define $\slc{\alpha_0} (\s)$ and $\clc{\alpha_0} (\s)$ for $\s$ a pre-frame (in the definition of $\clc{\alpha_0} (\s)$, we require in addition that $M_\delta$ be a member of $K$).
\end{defin}

We will use the following notation to restrict independence relations to smaller domains:

\begin{notation}\label{restr-not}\index{restriction (of an independence relation)} \index{$\is^{<\alpha}$|see {restriction (of an independence relation)}} \index{$\is \rest K'$|see {restriction (of an independence relation)}}\index{$\s^{<\alpha}$|see {restriction (of an independence relation)}}\index{$\s \rest K'$|see {restriction (of an independence relation)}}
  Let $\is$ be a $(<\alpha, \F, <\beta)$-independence relation. 
  \begin{enumerate}
    \item For $\alpha_0 \le \alpha$, $\beta_0 \le \beta$, let $\is^{<\alpha_0, <\beta_0}$ denotes the $(<\alpha_0, \F, <\beta_0)$-independence relation obtained by restricting the types to have length less than $\alpha_0$ and the right hand side to have size less than $\beta_0$ (in the natural way). When $\beta_0 = \beta$, we omit it.
    \item For $K'$ a coherent sub-AC of $K_{\is}$, let $\is \rest K'$ be the $(<\alpha, \F, <\beta)$-independence relation obtained by restricting the underlying class to $K'$. When $\is$ is a $(<\alpha, \F)$-independence relation and $\F_0 \subseteq \F$ is an interval of cardinals, $\F_0 = [\lambda_{\is}, \theta_0)$, we let $\is_{\F_0} := \is^{<\min (\alpha, \theta_0)} \rest (K_{\is})_{\F_0}$ be the restriction of $\is$ to models of size in $\F_0$ and types of appropriate length.
  \end{enumerate}
\end{notation}

We end this section with two examples of independence relations. The first is coheir. In first-order logic, coheir was first defined in \cite{lascar-poizat}\footnote{The equivalence of nonforking with coheir (for stable theories) was already established by Shelah in the early seventies and appears in Section III.4 of \cite{shelahfobook78}, see also \cite[Corollary III.4.10]{shelahfobook}.}. A definition of coheir for classes of models of an $\Ll_{\kappa, \omega}$ sentence appears in \cite{makkaishelah} and was later adapted to general AECs in \cite{bg-v9}. In \cite{sv-infinitary-stability-v6-toappear}, we gave a more conceptual (but equivalent) definition and improved some of the results of Boney and Grossberg. Here, we use Boney and Grossberg's definition but rely on \cite{sv-infinitary-stability-v6-toappear}.

\begin{defin}[Coheir] \label{coheir-def}\index{coheir}\index{$\isch{\kappa} (K)$|see {coheir}}
  Let $K$ be an AEC with amalgamation and let $\kappa > \LS (K)$.

  Define $\isch{\kappa} (K) := (\Ksatp{\kappa}, \nf)$ by $\nf (M, A, B, N)$ if and only if $M \lea N$ are in $\Ksatp{\kappa}$, $A \cup B \subseteq |N|$, and for any $\ba \in \fct{<\kappa}{A}$ and $B_0 \subseteq |M| \cup B$ of size less than $\kappa$, there exists $\ba' \in \fct{<\kappa}{|M|}$ such that $\gtp (\ba / B_0; N) = \gtp (\ba' / B_0; M)$.
\end{defin}

\begin{fact}[Theorem 5.15 in \cite{sv-infinitary-stability-v6-toappear}]\label{coheir-syn} 
  Let $K$ be an AEC with amalgamation and let $\kappa > \LS (K)$. Let $\is := \isch{\kappa} (K)$.
  \begin{enumerate}
    \item\label{coheir-1} $\is$ is a $(<\infty, [\kappa, \infty))$-independence relation with disjointness, base monotonicity, left and right existence, left and right $(<\kappa)$-witness property, and strong left transitivity.
    \item\label{coheir-2} If $K$ does not have the $(<\kappa)$-order property of length $\kappa$, then:
      \begin{enumerate}
        \item $\is$ has symmetry and strong right transitivity.
        \item For all $\alpha$, $\slc{\alpha} (\is) \le \left(\left(\alpha + 2\right)^{<\kappap}\right)^+$.
        \item\label{coheir-uq} If $M_0 \lea M \lea N_\ell$ for $\ell = 1,2$, $|M_0| \subseteq B \subseteq |M|$. $q_\ell \in \gS^{<\infty} (B; N_\ell)$, $q_1 \rest M_0 = q_2 \rest M_0$, $q_\ell$ does not $\is$-fork over $M_0$ for $\ell = 1,2$, and $K$ is $(<\kappa)$-tame and short for $\{q_1, q_2\}$, then $q_1 = q_2$.
        \item If $K$ is $(<\kappa)$-tame and short for types of length less than $\alpha$, then $\pre (\is^{<\alpha})$ has uniqueness. Moreover\footnote{Of course, this is only interesting if $\alpha \le \kappa$.} $\is_{[\kappa, \alpha)}^{<\alpha}$ has uniqueness.

    \end{enumerate}
  \end{enumerate}
\end{fact}
\begin{remark}
  The extension property\footnote{A word of caution: In \cite[Section 4]{hyttinen-lessmann}, the authors give Shelah's example of an $\omega$-stable class that does not have extension. However, the extension property they consider is \emph{over all sets}, not only over models.} seems to be more problematic. In \cite{bg-v9}, Boney and Grossberg simply assumed it (they also showed that it followed from $\kappa$ being strongly compact \cite[Theorem 8.2.(1)]{bg-v9}). From superstability-like hypotheses, we will obtain more results on it (see Theorem \ref{categ-frame}, Theorem \ref{main-thm}, and Theorem \ref{good-frame-succ}).
\end{remark}

We now consider another independence notion: splitting. This was first defined for AECs in \cite[Definition 3.2]{sh394}. Here we define the negative property (nonsplitting), as it is the one we use the most.

\begin{defin}[$\lambda$-nonsplitting]\index{splitting}\index{nonsplitting|see {splitting}}\index{$\lambda$-splitting|see {splitting}}\index{$\lambda$-nonsplitting|see {splitting}}\index{$\snsp{\lambda} (K)$|see {splitting}}\index{$\sns (K)$|see {splitting}}\index{$\insp{\lambda} (K)$|see {splitting}} \index{$\ins (K)$|see {splitting}}
  Let $K$ be a coherent abstract class with amalgamation.
  
  \begin{enumerate}
  \item For $\lambda$ an infinite cardinal, define $\snsp{\lambda} (K) := (K, \nf)$ by $\nfs{M_0}{\ba}{M}{N}$ if and only if $M_0 \lea M \lea N$, $A \subseteq |N|$, and whenever $M_0 \lea N_\ell \lea M$, $N_\ell \in K_{\le \lambda}$, $\ell = 1,2$, and $f: N_1 \cong_{M_0} N_2$, then $f (\gtp (\ba / N_1; N)) = \gtp (\ba / N_2; N)$.
    \item Define $\sns (K)$ to have underlying AEC $K$ and forking relation defined such that $p \in \gS^{<\infty} (M)$ does not $\sns (K)$-fork over $M_0 \lea M$ if and only if $p$ does not $\snsp{\lambda} (K)$-fork over $M_0$ for all infinite  $\lambda$.
    \item Let $\insp{\lambda} (K) := \cl (\snsp{\lambda} (K))$, $\ins (K) := \cl (\sns (K))$.
  \end{enumerate}
\end{defin}

\begin{fact}\label{splitting-basics}
  Assume $K$ is a coherent AC in $\F = [\lambda, \theta)$ with amalgamation. Let $\s := \sns (K)$, $\s' := \snsp{\lambda} (K)$.
  \begin{enumerate}
    \item $\s$ and $\s'$ are pre-$(<\infty, \F)$-frame with base monotonicity, left and right existence. If $K$ is $\lambda$-closed, $\s'$ has the right $\lambda$-model-witness property.
    \item\label{splitting-basics-25} If $K$ is an AEC in $\F$ and is stable in $\lambda$, then $\slc{<\omega} (\s') = \lambda^+$.
    \item\label{splitting-basics-3} If $\ts$ is a pre-$(<\infty, \F)$-frame with uniqueness and $K_{\ts} = K$, then $\nf_{\ts} \subseteq \nf_{\s}$.
    \item\label{splitting-basics-tameness} Always, $\nf_{\s} \subseteq \nf_{\s'}$. Moreover if $K$ is $\lambda$-tame for types of length less than $\alpha$, then $\s^{<\alpha} = (\s')^{<\alpha}$.
    \item\label{splitting-weak-props} Let $M_0 \ltu M \lea N$ with $\|M\| = \|N\|$. 
      \begin{enumerate}
        \item Weak uniqueness: If $p_\ell \in \gS^{\alpha} (N)$, $\ell = 1,2$, do not $\s$-fork over $M_0$ and $p_1 \rest M = p_2 \rest M$, then $p_1 = p_2$.
        \item Weak extension: If $p \in \gS^{<\infty} (M)$ does not $\s$-fork over $M_0$ and $f: N \xrightarrow[M_0]{} M$, then $q := f^{-1} (p)$ is an extension of $p$ to $\gS^{<\infty}(N)$ that does not $\s$-fork over $M_0$. Moreover $q$ is algebraic if and only if $p$ is algebraic.
      \end{enumerate}
  \end{enumerate}
\end{fact}
\begin{proof} \
  \begin{enumerate}
    \item Easy.
    \item By \cite[Claim 3.3.1]{sh394} (see also \cite[Fact 4.6]{tamenessone}).
    \item By \cite[Lemma 4.2]{bgkv-v3-toappear}.
    \item By \cite[Proposition 3.12]{bgkv-v3-toappear}).
    \item By \cite[Theorem I.4.10, Theorem I.4.12]{vandierennomax} (the moreover part is easy to see from the definition of $q$).
  \end{enumerate}
\end{proof}

\begin{remark}
Fact \ref{splitting-basics}.(\ref{splitting-basics-3}) tells us that any reasonable independence relation will be extended by nonsplitting. In this sense, nonsplitting is a \emph{maximal} candidate for an independence relation\footnote{Moreover, $(<\kappa)$-coheir is a \emph{minimal} candidate in the following sense: Let us say an independence relation $\is = (K, \nf)$ has the \index{strong witness property}\index{strong $(<\kappa)$-witness property|see {strong witness property}} \emph{strong $(<\kappa)$-witness property} if whenever $\nnfs{M}{A}{B}{N}$, there exists $\ba_0 \in \fct{<\kappa}{A}$ and $B_0 \subseteq |M| \cup B$ of size less than $\kappa$ such that $\gtp (\ba_0' / B_0; N) = \gtp (\ba_0 / B_0; N)$ implies $\nnfs{M}{\ba_0}{B}{N}$. Intuitively, this says that forking is witnessed by a formula (and this could be made precise using the notion of Galois Morleyization, see \cite{sv-infinitary-stability-v6-toappear}). It is easy to check that $(<\kappa)$-coheir has this property, and any independence relation with strong $(<\kappa)$-witness and left existence must extend $(<\kappa)$-coheir.}. 
\end{remark}

\section{Some independence calculus}\label{indep-calc-sec}

We investigate relationships between properties and how to go from a frame to an independence relation. Most of it appears already in \cite{bgkv-v3-toappear} and has a much longer history, described there. The following are new: Lemma \ref{cont-lem} gives a way to get the witness properties from tameness, partially answering \cite[Question 5.5]{bgkv-v3-toappear}. Lemmas \ref{lc-cont} and \ref{lc-monot} are technical results used in the last sections.

The following proposition investigates what properties are preserved by the operations $\cl$ and $\pre$ (recall Definition \ref{cl-def}). This was done already in \cite[Section 5.1]{bgkv-v3-toappear}, so we cite from there.

\begin{prop}\label{cl-basics} \
  Let $\s$ be a pre-$(<\alpha, \F)$-frame and let $\is$ be a $(<\alpha, \F)$-independence relation.

  \begin{enumerate}
  \item For $P$ in the list of properties of Definition \ref{indep-props-def}, if $\is$ has $P$, then $\pre (\is)$ has $P$.
    \item For $P$ a property in the following list, $\is$ has $P$ if (and only if) $\pre (\is)$ has $P$: existence, independent amalgamation, full model-continuity.
    \item For $P$ a property in the following list, $\cl (\s)$ has $P$ if (and only if) $\s$ has $P$: disjointness, full symmetry, base monotonicity, extension, transitivity.
    \item\label{cl-basics-1} If $\pre (\is)$ has extension, then $\cl (\pre (\is)) = \is$ if and only if $\is$ has extension.
    \item The following are equivalent:
      \begin{enumerate}
        \item $\s$ has full symmetry.
        \item $\cl (\s)$ has symmetry.
        \item $\cl (\s)$ has full symmetry.
      \end{enumerate}
    \item If $\pre (\is)$ has uniqueness and $\is$ has extension, then $\is$ has uniqueness.
    \item\label{cl-basics-ext} If $\pre (\is)$ has extension and $\is$ has uniqueness, then $\is$ has extension.
    \item \footnote{Note that maybe $\alpha = \infty$. However we can always apply the proposition to $\s^{<\alpha_0}$ for an appropriate $\alpha_0 \le \alpha$.} $\slc{<\alpha} (\is) = \slc{<\alpha} (\pre (\is))$.
    \item $\clc{<\alpha} (\pre (\is)) \le \clc{<\alpha} (\is)$. If $K_{\is}$ is an AEC, then this is an equality.
  \end{enumerate}
\end{prop}
\begin{proof} 
  All are straightforward. See \cite[Lemmas 5.3, 5.4]{bgkv-v3-toappear}.
\end{proof}

To what extent is an independence relation determined by its corresponding frame? There is an easy answer:

\begin{lem}\label{pre-to-cl}
  Let $\is$ and $\is'$ be independence relations with $\pre (\is) = \pre (\is')$. If $\is$ and $\is'$ both have extension, then $\is = \is'$.
\end{lem}
\begin{proof}
  By Proposition \ref{cl-basics}.(\ref{cl-basics-1}), $\is = \cl (\pre (\is))$ and $\is' = \cl (\pre (\is')) = \cl (\pre (\is)) = \is$.
\end{proof}

The next proposition gives relationships between the properties. We state most results for frames, but they usually have an analog for independence relations that can be obtained using Proposition \ref{cl-basics}.

\begin{prop}\label{indep-props}
  Let $\is$ be a $(<\alpha, \F)$-independence relation with base monotonicity. Let $\s$ be a pre-$(<\alpha, \F)$-frame with base monotonicity.
  \begin{enumerate} 
    \item If $\is$ has full symmetry, then it has symmetry. If $\is$ has the $(<\kappa)$-witness property, then it has the $(<\kappa)$-model-witness property. If $\is$ [$\s$] has strong transitivity, then it has transitivity.
    \item\label{indep-props-4} If $\s$ has uniqueness and extension, then it has transitivity.
    \item\label{indep-props-7} For $\alpha > \lambda$, if $\s$ has extension and existence, then $\s$ has independent amalgamation. Conversely, if $\s$ has transitivity and independent amalgamation, then $\s$ has extension and existence. Moreover if $\s$ has uniqueness and independent amalgamation, then it has transitivity.
    \item\label{indep-props-exi} If $\min(\clc{<\alpha} (\s), \slc{<\alpha} (\s)) < \infty$, then $\s$ has existence.
    \item\label{indep-props-lc} $\clc{<\alpha} (\s) \le \slc{<\alpha} (\s)$.
    \item\label{indep-props-wit} If $K_{\s}$ is $\lambda_{\s}$-closed, $\slc{<\alpha} (\s) = \lambda_\s^+$ and $\s$ has transitivity, then $\s$ has the right $\lambda$-model-witness property.
    \item\label{indep-props-sym} If $K_{\s}$ does not have the order property (Definition \ref{def-op}), any chain in $K_{\s}$ has an upper bound, $\theta = \infty$, and $\s$ has uniqueness, existence, and extension, then $\s$ has full symmetry.
  \end{enumerate}
\end{prop}
\begin{proof} \
  \begin{enumerate}
    \item Easy.
    \item As in the proof of \cite[Claim II.2.18]{shelahaecbook}.
    \item The first sentence is easy, since independent amalgamation is a particular case of extension and existence. Moreover to show existence it is enough by monotonicity to show it for types of models. The proof of transitivity from uniqueness and independent amalgamation is as in (\ref{indep-props-4}).
    \item By definition of the local character cardinals.
    \item Let $\delta = \cf{\delta} \ge \slc{<\alpha} (\s)$ and $\seq{M_i : i < \delta}$ be increasing in $K$, $N \gea M_i$ for all $i < \delta$ and $A \subseteq |N|$ with $|A| < \alpha$. Assume $M_\delta := \bigcup_{i < \delta} M_i$ is in $K$. By definition of $\slc{<\alpha}$ there exists $N \lea M_\delta$ of size less than $\slc{\alpha_0} (\is)$ such that $p$ does not fork over $N$. Now use regularity of $\delta$ to find $i < \delta$ with $N \lea M_i$.
    \item Let $\lambda := \lambda_{\s}$, say $\s = (K, \nf)$. Let $M_0 \lea M \lea N$ and assume $\nfs{M_0}{A}{B}{N}$ for all $B \subseteq |M|$ with $|B| \le \lambda$. By definition of $\slc{<\alpha} (\s)$, there exists $M_0' \lea M$ of size $\lambda$ such that $\nfs{M_0'}{A}{M}{N}$. By $\lambda$-closure and base monotonicity, we can assume without loss of generality that $M_0 \lea M_0'$. By assumption, $\nfs{M_0}{A}{M_0'}{N}$, so by transitivity, $\nfs{M_0}{A}{M}{N}$.
    \item As in \cite[Corollary 5.16]{bgkv-v3-toappear}.
  \end{enumerate}
\end{proof}
\begin{remark}\label{indep-props-sym-rmk}
  The precise statement of \cite[Corollary 5.16]{bgkv-v3-toappear} shows that Proposition \ref{indep-props}.(\ref{indep-props-sym}) is local in the sense that to prove symmetry over the base model $M$, it is enough to require uniqueness and extension over this base model (i.e.\ any two types that do not fork over $M$, have the same domain, and are equal over $M$ are equal over their domain, and any type over $M$ can be extended to an arbitrary domain so that it does not fork over $M$).
\end{remark}

\begin{lem}\label{cont-lem}
  Let $\is = (K, \nf)$ be a $(<\alpha, \F)$-independence relation. If $\is$ has extension and uniqueness, then:
  
      \begin{enumerate}
        \item If $K$ is $(<\kappa)$-tame for types of length less than $\alpha$, then $K$ has the right $(<\kappa)$-model-witness property.
        \item If $K$ is $(<\kappa)$-tame and short for types of length less than $\theta_{\is}$, then $K$ has the right $(<\kappa)$-witness property.
        \item If $K$ is $(<\kappa)$-tame and short for types of length less than $\kappa + \alpha$ and $\is$ has symmetry, then $K$ has the left $(<\kappa)$-witness property.
      \end{enumerate}
\end{lem}
\begin{proof} \
\begin{enumerate}
      \item Let $M \lea M' \lea N$ be in $K$, $A \subseteq |N|$ have size less than $\alpha$. Assume $\nfs{M}{A}{B_0}{N}$ for all $B_0 \subseteq |M'|$ of size less than $\kappa$. We want to show that $\nfs{M}{A}{M'}{N}$. Let $\ba$ be an enumeration of $A$, $p := \gtp (\ba / M; N)$. Note that (taking $B_0 = \emptyset$ above) normality implies $p$ does not fork over $M$. By extension, let $q \in \gS^{<\alpha} (M')$ be an extension of $p$ that does not fork over $M$. Using amalgamation and some renaming, we can assume without loss of generality that $q$ is realized in $N$. Let $p' := \gtp (\ba / M'; N)$. We claim that $p' = q$, which is enough by invariance. By the tameness assumption, it is enough to check that $p' \rest B_0 = q \rest B_0$ for all $B_0 \subseteq |M'|$ of size less than $\kappa$. Fix such a $B_0$. By assumption, $p' \rest B_0$ does not fork over $M$. By monotonicity, $q \rest B_0$ does not fork over $M$. By uniqueness, $p' \rest B_0 = q \rest B_0$, as desired.
      \item Similar to before, noting that for $M \lea N$, $\gtp (\ba / M\bb; N) = \gtp (\ba' / M\bb; N)$ if and only if $\gtp (\ba \bb / M; N) = \gtp (\ba' \bb / M; N)$.
      \item Observe that in the proof of the previous part, if the set on the right hand side has size less than $\kappa$, it is enough to require $(<\kappa)$-tameness and shortness for types of length less than $(\alpha + \kappa)$. Now use symmetry.
\end{enumerate}
\end{proof}

Having a nice independence relation makes the class nice. The results below are folklore:

\begin{prop}\label{class-props}
  Let $\is = (K, \nf)$ be a $(<\alpha, \F)$-independence relation with base monotonicity. Assume $K$ is an AEC in $\F$ with $\LS (K) = \lambda_{\is}$.
  \begin{enumerate}
    \item If $\is$ has uniqueness, and $\kappa := \slc{<\alpha} (\is) < \infty$, then $K$ is $(<\kappa)$-tame for types of length less than $\alpha$.
    \item If $\is$ has uniqueness and $\kappa := \slc{<\alpha} (\is) < \infty$, then $K$ is $(<\alpha)$-stable in any infinite $\mu$ such that $\mu = \mu^{<\kappa}$.
    \item If $\is$ has uniqueness, $\mu > \LS (K)$, $K$ is $(<\alpha)$-stable in unboundedly many $\mu_0 < \mu$, and $\cf{\mu} \ge \clc{<\alpha} (\is)$, then $K$ is $(<\alpha)$-stable in $\mu$.
  \end{enumerate}
\end{prop}
\begin{proof} \
  \begin{enumerate}
    \item See \cite[p.~15]{superior-aec}, or the proof of \cite[Theorem 3.2]{ext-frame-jml}.
    \item Let $\mu = \mu^{<\kappa}$ be infinite. Let $M \in K_{\le \mu}$, $\seq{p_i : i < \mu^+}$ be elements in $\gS^{<\alpha} (M)$. It is enough to show that for some $i < j$, $p_i = p_j$. For each $i < \lambda^+$, there exists $M_i \lea M$ in $K_{<\kappa}$ such that $p_i$ does not fork over $M_i$. Since $\mu = \mu^{<\kappa}$, we can assume without loss of generality that $M_i = M_0$ for all $i < \mu^+$. Also, $|\gS^{<\alpha} (M_0)| \le 2^{<\kappa} \le \mu^{<\kappa} = \mu$ so there exists $i < j < \lambda^+$ such that $p_i \rest M_0 = p_j \rest M_0$. By uniqueness, $p_i = p_j$, as needed.
    \item As in the proof of \cite[Lemma 5.5]{ss-tame-jsl}.
  \end{enumerate}
\end{proof}

The following technical result is also used in the last sections. Roughly, it gives conditions under which we can take the base model given by local character to be contained in both the left and right hand side.

\begin{lem}\label{lc-monot}
  Let $\is = (K, \nf)$ be a $(<\alpha, \F)$-independence relation, $\F = [\lambda, \theta)$, with $\alpha > \lambda$. Assume:
    \begin{enumerate}
      \item $K$ is an AEC with $\LS (K) = \lambda$.
      \item $\is$ has base monotonicity and transitivity.
      \item $\mu$ is a cardinal, $\lambda \le \mu < \theta$.
      \item $\is$ has the left $(<\kappa)$-model-witness property for some regular $\kappa \le \mu$.
      \item $\slc{\mu} (\is) = \mu^+$. 
    \end{enumerate}

    Let $M^0 \lea M^\ell \lea N$ be in $K$, $\ell = 1,2$ and assume $\nfs{M^0}{M^1}{M^2}{N}$. Let $A \subseteq |M^1|$, be such that $|A| \le \mu$. Then there exists $N^1 \lea M^1$ and $N^0 \lea M^0$ such that:

  \begin{enumerate}
    \item $A \subseteq |N^1|$, $A \cap |M^0| \subseteq |N^0|$.
    \item $N^0 \lea N^1$ are in $K_{\le \mu}$.
    \item $\nfs{N^0}{N^1}{M^2}{N}$.
  \end{enumerate}
\end{lem}
\begin{proof}
  For $\ell = 0,1$, we build $\seq{N_i^\ell : i \le \kappa}$ increasing continuous in $K_{\le \mu}$ such that for all $i < \kappa$ and $\ell = 0, 1$:

  \begin{enumerate}
    \item $A \subseteq |N_0^1|$, $A \cap |M^0| \subseteq |N_0^0|$.
    \item $N_i^\ell \lea M^\ell$.
    \item $N_i^0 \lea N_i^1$.
    \item $\nfs{N_{i + 1}^0}{N_i^1}{M^2}{N}$.
  \end{enumerate}

  \paragraph{\underline{This is possible}} Pick any $N_0^0 \lea M^0$ in $K_{\le \mu}$ containing $A \cap |M^0|$. Now fix $i < \kappa$ and assume inductively that $\seq{N_j^0 : j \le i}$, $\seq{N_j^1 : j < i}$ have been built. If $i$ is a limit, we take unions. Otherwise, pick any $N_i^1 \lea M^1$ in $K_{\le \mu}$ that contains $A$, $N_j^1$ for all $j < i$ and $N_i^0$. Now use right transitivity and $\slc{\mu} (\is) = \mu^+$ to find $N_{i + 1}^0 \lea M^0$ such that $\nfs{N_{i + 1}^0}{N_i^1}{M^2}{N}$. By base monotonicity, we can assume without loss of generality that $N_i^0 \lea N_{i + 1}^0$.

  \paragraph{\underline{This is enough}} We claim that $N^\ell := N_\kappa^\ell$ are as required. By coherence, $N^0 \lea N^1$ and since $\kappa \le \mu$ they are in $K_{\le \mu}$. Since $A \subseteq |N_0^1|$, $A \subseteq |N^1|$. It remains to see $\nfs{N^0}{N^1}{M^2}{N}$. By the left witness property\footnote{Note that we do \emph{not} need to use full model continuity, as we only care about chains of cofinality $\ge \kappa$.}, it is enough to check it for every $B \subseteq |N^1|$ of size less than $\kappa$. Fix such a $B$. Since $\kappa$ is regular, there exists $i < \kappa$ such that $B \subseteq |N_i^1|$. By assumption and monotonicity, $\nfs{N_{i + 1}^0}{B}{M^2}{N}$. By base monotonicity, $\nfs{N_{\kappa}^0}{B}{M^2}{N}$, as needed.
\end{proof}

With a similar proof, we can clarify the relationship between full model continuity and local character. Essentially, the next lemma says that local character for types up to a certain length plus full model-continuity implies local character for all lengths. It will be used in Section \ref{long-transfer-sec}.

\begin{lem}\label{lc-cont}
  Let $\is = (K, \nf)$ be a $(<\theta_{\is}, \F)$-independence relation, $\F = [\lambda, \theta)$. Assume:

    \begin{enumerate}
      \item $K$ is an AEC with $\LS (K) = \lambda$.
      \item $\is$ has base monotonicity, transitivity, and full model continuity.
      \item\label{lc-cont-3} $\is$ has the left $(<\kappa)$-model-witness property for some regular $\kappa \le \lambda$.
      \item\label{lc-cont-4} For all cardinals $\mu \le \lambda$, $\slc{\mu} (\is) = \lambda^+$.
    \end{enumerate}

    Then for all cardinals $\mu < \theta$, $\slc{\mu} (\is) = \lambda^+ + \mu^+$.
\end{lem}
\begin{proof}
  By induction on $\mu$. If $\mu \le \lambda$, this holds by hypothesis, so assume $\mu > \lambda$. Let $\delta := \cf{\mu}$.

  Let $M^0 \lea M^1$ be in $K$ and let $A \subseteq |M^1|$ have size $\mu$. We want to find $M \lea M^0$ such that $\nfs{M}{A}{M^0}{N}$ and $\|M\| \le \mu$. Let $\seq{A_i : i \le \delta}$ be increasing continuous such that $A = A_\delta$ and $|A_i| < \mu$ for all $i < \delta$. 

  For $\ell = 0,1$, we build $\seq{N_i^\ell : i \le \delta}$ increasing continuous such that for all $i < \delta$ and $\ell = 0, 1$:

  \begin{enumerate}
    \item $N_i \in K_{<\mu}$.
    \item $A_i \subseteq |N_i^1|$, $A_i \cap |M^0| \subseteq |N_i^0|$.
    \item $N_i^\ell \lea M^\ell$.
    \item $N_i^0 \lea N_i^1$.
    \item $\nfs{N_{i + 1}^0}{N_i^1}{M^0}{M^1}$.
  \end{enumerate}

  \paragraph{\underline{This is possible}} 
By (\ref{lc-cont-3}) and (\ref{lc-cont-4}), we have $\nfs{M^0}{M^1}{M^0}{M^1}$. Now proceed as in the proof of Lemma \ref{lc-monot}.

\paragraph{\underline{This is enough}}

As in the proof of Lemma \ref{lc-monot}, for any $i < \delta$ of cofinality at least $\kappa$ we have $\nfs{N_i^0}{N_i^1}{M^0}{M^1}$. Thus by full model continuity (applied to the sequences $\seq{N_i^\ell : i < \delta, \cf{i} \ge \kappa}$), $\nfs{N_\delta^0}{N_\delta^1}{M^0}{M^1}$. Since $A = A_\delta \subseteq |N_\delta^1|$, $M := N_\delta^0$ is as needed.
\end{proof}

\section{Skeletons}\label{dense-subac-sec}

We define what it means for an abstract class $K'$ to be a \emph{skeleton} of an abstract class $K$. The main examples are classes of saturated models with the usual ordering (or even universal or limit extension). Except perhaps for Lemma \ref{proper-ext-skel}, the results of this section are either easy or well known, we simply put them in the general language of this paper.

We will use skeletons to generalize various statements of chain local character (for example in \cite{gvv-toappear-v1_2, ss-tame-jsl}) that only ask that if $\seq{M_i : i < \delta}$ is an increasing chain with respect to \emph{some restriction of the} ordering of $K$ (usually being universal over) and the $M_i$s are inside some subclass of $K$ (usually some class of saturated models), then any $p \in \gS (\bigcup_{i < \delta} M_i)$ does not fork over some $M_i$. Lemma \ref{up-indep-prop}, is they key upward transfer of that property. Note that Lemma \ref{univ-lc} shows that one can actually assume that skeletons have a particular form. However the generality is still useful when one wants to prove the local character statement.

\begin{defin}\label{sub-ac}\index{sub-AC}\index{sub-abstract class| see {sub-AC}} \index{$\leg$}
  For $(K, \lea)$ an abstract class, we say $(K', \leg)$ is a \emph{sub-AC} of $K$ if $K' \subseteq K$, $(K', \leg)$ is an AC, and $M \leg N$ implies $M \lea N$. We similarly define sub-AEC, etc. When $\leg = \lea \rest K'$, we omit it (or may abuse notation and write $(K', \lea)$).
\end{defin}

\begin{defin}\label{dense-ac}\index{dense (subset of an abstract class)}
  For $(K, \lea)$ an abstract class, we say a set $S \subseteq K$ is \emph{dense} in $(K, \lea)$ if for any $M \in K$ there exists $M' \in S$ with $M \lea M'$.
\end{defin}

\begin{defin}\label{sub-skeletal-ac}\index{skeleton}
  An abstract class $(K', \leg)$ is a \emph{skeleton} of $(K, \lea)$ if:

  \begin{enumerate}
    \item\label{skel-1} $(K', \leg)$ is a sub-AC of $(K, \lea)$.
    \item\label{skel-2} $K'$ is dense in $(K, \lea)$.
    \item\label{skel-3} If $\seq{M_i : i < \alpha}$ is a $\lea$-increasing chain in $K'$ ($\alpha$ not necessarily limit) and there exists $N \in K'$ such that $M_i \lta N$ for all $i < \alpha$, then we can choose such an $N$ with $M_i \ltg N$ for all $i < \alpha$.
  \end{enumerate}
\end{defin}
\begin{remark}\index{skeletal}
  The term ``skeleton'' is inspired from the term ``skeletal'' in \cite{ss-tame-jsl}, although there ``skeletal'' is applied to frames. The intended philosophical meaning is the same: $K'$ has enough information about $K$ so that for several purposes we can work with $K'$ rather than $K$.
\end{remark}
\begin{remark}\label{abstract-univ-rmk}
  Let $(K, \lea)$ be an abstract class. Assume $(K', \leg)$ is a dense sub-AC of $(K, \lea)$ with no maximal models satisfying in addition: If $M_0 \lea M_1 \ltg M_2$ are in $K'$, then $M_0 \ltg M_2$. Then $(K', \leg)$ is a skeleton of $(K, \lea)$. This property of the ordering already appears in the definition of an abstract universal ordering in \cite[Definition 2.13]{ss-tame-jsl}. In the terminology there, if $(K, \lea)$ is an AEC and $\ltg$ is an (invariant) universal ordering on $K_\lambda$, then $(K_\lambda, \leg)$ is a skeleton of $(K_\lambda, \lea)$.
\end{remark}

\begin{example}\label{skel-examples}
  Let $K$ be an AEC. Let $\lambda \ge \LS (K)$. Assume that $K_\lambda$ has amalgamation, no maximal models and is stable in $\lambda$. Let $K'$ be dense in $K_\lambda$ and let $\delta < \lambda^+$. Then $(K', \lel{\lambda}{\delta})$ (recall Definition \ref{univ-def}) is a skeleton of $(K_\lambda, \lea)$ (use Fact \ref{ltl-basic-props} and Remark \ref{abstract-univ-rmk}).
\end{example}

The next lemma is a useful tool to find extensions in the skeleton of an AEC with amalgamation:

\begin{lem}\label{proper-ext-skel}
  Let $(K', \leg)$ be a skeleton of $(K, \lea)$. Assume $K$ is an AEC in $\F := [\lambda, \theta)$ with amalgamation. If $M \lea N$ are in $K'$, then there exists $N' \in K'$ such that $M \leg N'$ and $N \leg N'$.
\end{lem}
\begin{proof}
  If $N$ is not maximal (with respect to either of the orderings, it does not matter by definition of a skeleton), then using the definition of a skeleton with $\alpha = 2$ and the chain $\seq{M, N}$, we can find $N' \in K'$ such that $N \ltg N'$ and $M \ltg N'$, as needed.
  
  Now assume $N$ is maximal. We claim that $M \leg N$, so $N' := N$ is as desired. Suppose not. Let $\mu := \|N\|$. 

  We build $\seq{M_i : i < \mu^+}$ and $\seq{f_i : M_i \xrightarrow[M]{} N : i < \mu^+}$ such that:

  \begin{enumerate}
    \item $\seq{M_i : i < \mu^+}$ is a strictly increasing chain in $(K', \leg)$ with $M_0 = M$.
    \item $\seq{f_i : i < \mu^+}$ is a strictly increasing chain of $K$-embeddings.
  \end{enumerate}

  \paragraph{\underline{This is enough}} Let $B_{\mu^+} := \bigcup_{i < \mu^+} |M_i|$ and $f_{\mu^+} := \bigcup_{i < \mu^+} f_i$ (Note that it could be that $\mu^+ = \theta$, so $B_{\mu^+}$ is just a set and we do not claim that $f_{\mu^+}$ is a $K$-embedding). Then $f_{\mu^+}$ is an injection from $B_{\mu^+}$ into $|N|$. This is impossible because $|B_{\mu^+}| \ge \mu^+ > \mu = \|N\|$.

  \paragraph{\underline{This is possible}} Set $M_0 := M$, $f_0 := \id_{M}$. If $i < \mu^+$ is limit, let $M_i' := \bigcup_{j < i} M_j \in K$. By density, find $M_i'' \in K'$ such that $M_i' \lea M_i''$. We have that $M_j \lta M_i''$ for all $j < i$. By definition of a skeleton, this means we can find $M_i \in K'$ with $M_j \ltg M_i$ for all $j < i$. Let $f_i' := \bigcup_{j < i} f_j$. Using amalgamation and the fact that $N$ is maximal, we can extend it to $f_i : M_i \xrightarrow[M]{} N$. If $i = j + 1$ is successor, we consider two cases:

  \begin{itemize}
    \item If $M_j$ is not maximal, let $M_i \in K'$ be a $\ltg$-extension of $M_j$. Using amalgamation and the fact $N$ is maximal, pick $f_i: M_i \xrightarrow[M]{} N$ an extension of $f_j$.
    \item If $M_j$ is maximal, then by amalgamation and the fact both $N$ and $M_j$ are maximal, we must have $N \cong_{M} M_j$. However by assumption $M_0 \leg M_j$ so $M = M_0 \leg N$, a contradiction.
  \end{itemize}
\end{proof}

Thus we get that several properties of a class transfer to its skeletons.

\begin{prop}\label{skel-transfer}
  Let $(K, \lea)$ be an AEC in $\F$ and let $(K', \leg)$ be a skeleton of $K$.

  \begin{enumerate}
    \item $(K, \lea)$ has no maximal models if and only if $(K', \leg)$ has no maximal models.
    \item If $(K, \lea)$ has amalgamation, then:
      \begin{enumerate}
      \item $(K', \leg)$ has amalgamation.
      \item $(K, \lea)$ has joint embedding if and only if $(K', \leg)$ has joint embedding.
      \item Galois types are the same in $(K, \lea)$ and $(K', \leg)$: For any $N \in K'$, $A \subseteq |N|$, $\bb, \bc \in \fct{\alpha}{|N|}$, $\gtp_K (\bb / A; N) = \gtp_K (\bc / A; N)$ if and only if $\gtp_{K'} (\bb / A; N) = \gtp_{K'} (\bc / A; N)$. Here, by $\gtp_K$ we denote the Galois type computed in $(K, \lea)$ and by $\gtp_{K'}$ the Galois type computed in $(K', \leg)$.
      \item $(K, \lea)$ is $\alpha$-stable in $\lambda$ if and only if $(K', \leg)$ is $\alpha$-stable in $\lambda$.
      \end{enumerate}
  \end{enumerate}
\end{prop}
\begin{proof} \
  \begin{enumerate}
    \item Directly from the definition.
    \item \begin{enumerate}
    \item Let $M_0 \leg M_\ell$ be in $K'$, $\ell = 1, 2$. By density, find $N \in K'$ and $f_\ell : M_\ell \xrightarrow[M_0]{} N$ $K$-embeddings. By Lemma \ref{proper-ext-skel}, there exists $N_1 \in K'$ such that $N \leg N_1$, $f_1[M_1] \leg N_1$. By Lemma \ref{proper-ext-skel} again, there exists $N_2 \in K'$ such that $N_1 \leg N_2$, $f_2[M_2] \leg N_2$. Thus we also have $f_1[M_1] \leg N_2$. It follows that $f_\ell: M_\ell \xrightarrow[M]{} N_2$ is a $\leg$-embedding.
    \item If $(K', \leg)$ has joint embedding, then by density $(K, \lea)$ has joint embedding. The converse is similar to the proof of amalgamation above.
    \item Note that by density any Galois type (in $K$) is realized in an element of $K'$. Since $(K', \leg)$ is a sub-AC of $(K, \lea)$, equality of the types in $K'$ implies equality in $K$ (this doesn't use amalgamation). Conversely, assume $\gtp_K (\bb / A; N) = \gtp_K (\bc / A; N)$. Fix $N' \gea N$ in $K$ and a $K$-embedding $f: N \xrightarrow[A]{} N'$ such that $f (\bb) = \bc$. By density, we can assume without loss of generality that $N' \in K'$. By Lemma \ref{proper-ext-skel}, find $N'' \in K'$ such that $N \leg N''$, $N' \leg N''$. By Lemma \ref{proper-ext-skel} again, find $N''' \in K'$ such that $f[N] \leg N'''$, $N'' \leg N'''$. By transitivity, $N \leg N'''$ and $f: N \xrightarrow[A]{} N'''$ witnesses equality of the Galois types in $(K', \leg)$.
    \item Because Galois types are the same in $K$ and $K'$.
    \end{enumerate}
  \end{enumerate}
\end{proof}

We end with an observation concerning universal extensions that will be used in the proof of Lemma \ref{univ-lc}.

\begin{lem}\label{univ-skel-lem}
  Let $K$ be an AEC in $\lambda := \LS (K)$. Assume $K$ has amalgamation, no maximal models, and is stable in $\lambda$. Let $(K', \leg)$ be a skeleton of $K$. For any $M \in K'$, there exists $N \in K'$ such that both $M \ltg N$ and $M \ltu N$. Thus $(K', \leg \cap \leu)$ is a skeleton of $K$.
\end{lem}
\begin{proof}
  For the last sentence, let $\leg' := \leg \cap \leu$. Note that if $\seq{M_i : i < \alpha}$ is a $\leg'$-increasing chain in $K'$ and $M \in K'$ is such that $M_i \lta M$ for all $i < \alpha$, then by definition of a skeleton we can take $M$ so that $M_i \ltg M$ for all $i < \alpha$. If we know that there exists $N \in K'$ with $M \ltg N$ and $M \ltu N$, then for all $i < \alpha$, $M_i \ltg N$ by transitivity, and $M_i \ltu N$ by Lemma \ref{lem-univ}.

  Now let $M \in K'$. By Fact \ref{ltl-basic-props}, there exists $N \in K$ with $M \ltu N$. By density (note that if $N' \gea N$ is in $K$, then $M \ltu N'$) we can take $N \in K'$. By Lemma \ref{proper-ext-skel}, there exists $N' \in K'$ such that $M \leg N'$ and $N \leg N'$. Thus (using Fact \ref{ltl-basic-props} again) $M \ltu N'$, as desired.
\end{proof}

\section{Generating an independence relation}\label{sec-transfer-up}

In \cite[Section II.2]{shelahaecbook}, Shelah showed how to extend a good $\lambda$-frame to a pre-$(\ge \lambda)$-frame. Later, \cite{ext-frame-jml} (with improvements in \cite{tame-frames-revisited-v5}) gave conditions under which all the properties transferred. Similar ideas are used in \cite{ss-tame-jsl} to directly build a good frame. In this section we adapt Shelah's definition to this paper's more general setup. It is useful to think of the initial $\lambda$-frame as a \emph{generator}\footnote{In \cite{ss-tame-jsl}, we called a generator a skeletal frame (and in earlier version a poor man's frame) but never defined it precisely.} for a $(\ge \lambda)$-frame, since in case the frame is not good we usually can only get a nice independence relation on $\lambda^+$-saturated models (and thus cannot really \emph{extend} the good $\lambda$-frame to a good $(\ge \lambda)$-frame). Moreover, it is often useful to work with the independence relation being only defined on a dense sub-AC of the original AEC.

\begin{defin}\label{generator-def}\index{generator}\index{$\lambda$-generator (for an independence relation) |see {generator}}
    $(K, \is)$ is a \emph{$\lambda$-generator for a $(<\alpha)$-independence relation} if:

    \begin{enumerate}
      \item $\alpha$ is a cardinal with $2 \le \alpha \le \lambda^+$. $\lambda$ is an infinite cardinal.
      \item $K$ is an AEC in $\lambda = \LS (K)$
      \item $\is$ is a $(<\alpha, \lambda)$-independence relation.
      \item $(K_{\is}, \lea)$ is a dense sub-AC (recall Definitions \ref{sub-ac}, \ref{dense-ac}) of\footnote{Why not be more general and require only $(K_{\is}, \leg)$ to be a skeleton of $K$ (for some ordering $\leg$)? Because some examples of skeletons do not satisfy the coherence axiom which is required by the definition of an independence relation.} $(K, \lea)$.

      \item $\Kup$ (recall Definition \ref{kup-def}) has amalgamation.
    \end{enumerate}  
\end{defin}
\begin{remark}
  We could similarly define a $\lambda$-generator for a $(<\alpha)$-independence relation below $\theta$, where we require $\theta \ge \lambda^{++}$ and only $\Kup_{\F}$ has amalgamation (so when $\theta = \infty$ we recover the above definition). We will not adopt this approach as we have no use for the extra generality and do not want to complicate the notation further. We could also have required less than ``$K$ is an AEC in $\lambda$'' but again we have no use for it.
\end{remark}

\begin{defin}\label{is-up-def}\index{$\Kupp{(K, \is)}$} \index{$\Kupp{\nf}$}\index{$\Kupp{\is}$}
  Let $(K, \is)$ be a $\lambda$-generator for a $(<\alpha)$-independence relation. Define $\Kupp{(K, \is)} := (\Kup, \Kupp{\nf})$ by $\Kupp{\nf} (M, A, B, N)$ if and only if $M \lea N$ are in $\Kup$ and there exists $M_0 \lea M$ in $K_{\is}$ such that for all $B_0 \subseteq B$ with $|B_0| \le \lambda$ and all $N_0 \lea N$ in $K_{\is}$ with $A \cup B_0 \subseteq |N_0|$, $M_0 \lea N_0$, we have $\nf_{\is} (M_0, A, B_0, N_0)$.

  When $K = K_{\is}$, we write $\Kupp{\is}$ for $\Kupp{(K, \is)}$.
\end{defin}
\begin{remark}
  In general, we do not claim that $\Kupp{(K, \is)}$ is even an independence relation (the problem is that given $A \subseteq |N|$ with $N \in \Kup$ and $|A| \lea \lambda$, there might not be any $M \in K_{\is}$ with $M \lea N$ and $A \subseteq |M|$ so the monotonicity properties can fail). Nevertheless, we will abuse notation and use the restriction operations on it.
\end{remark}

\begin{lem}\label{up-indep-rel}
  Let $(K, \is)$ be a $\lambda$-generator for a $(<\alpha)$-independence relation. Then:

  \begin{enumerate}
    \item\label{up-indep-rel-1} If $K = K_{\is}$, then $\Kupp{\is} := \Kupp{(K, \is)}$ is an independence relation.
    \item\label{up-indep-rel-2} $\Kupp{(K, \is)} \rest \Ksatpp{(\Kup)}{\lambda^+}$ is an independence relation.
  \end{enumerate}

\end{lem}
\begin{proof}
  As in \cite[Claim II.2.11]{shelahaecbook}, using density and homogeneity in the second case.
\end{proof}

The case (\ref{up-indep-rel-1}) of Lemma \ref{up-indep-rel} has been well studied (at least when $\alpha = 2$): see \cite[Section II.2]{shelahaecbook} and \cite{ext-frame-jml, tame-frames-revisited-v5}. We will further look at it in the last sections. We will focus on case (\ref{up-indep-rel-2}) for now. It has been studied (implicitly) in \cite{ss-tame-jsl} when $\is$ is nonsplitting and satisfies some superstability-like assumptions. We will use the same arguments as there to obtain more general results. The generality will be used, since for example we also care about what happens when $\is$ is coheir. 

The following property of a generator will be very useful in the next section. The point is that $\bigcup_{i < \lambda^+} M_i$ below is usually not a member of $K_{\is}$ so forking is not defined on it.

\begin{defin}\label{weak-chain-lc-def}\index{weak chain local character}
  Let $(K, \is)$ be a $\lambda$-generator for a $(<\alpha)$-independence relation.

  $(K, \is)$ has \emph{weak chain local character} if there exists $\leg$ such that $(K_{\is}, \leg)$ is a skeleton of $K$ and whenever $\seq{M_i : i < \lambda^+}$ is $\leg$-increasing in $K_\is$ and $p \in \gS^{<\alpha} (\bigcup_{i < \lambda^+} M_i)$, there exists $i < \lambda^+$ such that $p \rest M_{i + 1}$ does not fork over $M_i$.
\end{defin}

The following technical lemma shows that local character in a skeleton implies local character in a bigger class with the universal ordering: 

\begin{lem}\label{univ-lc}
  Let $(K, \is)$ be a $\lambda$-generator for a $(<\alpha)$-independence relation. 

  Assume that $K$ has amalgamation, no maximal models, and is stable in $\lambda$. Assume $\is$ has base monotonicity. Let $(K', \leg)$ be a skeleton of $(K_{\is}, \lea)$ and let $\is' := \is \rest (K', \lea)$. Then:

  \begin{enumerate}
    \item $\clc{<\alpha} (\is, \leu) \le \clc{<\alpha} (\is', \leg)$.
    \item If $(K, \is')$ has weak chain local character, then $(K, \is)$ has it and it is witnessed by $\ltu$.
  \end{enumerate}
\end{lem}
\begin{proof} \
  \begin{enumerate}
    \item By Lemma \ref{univ-skel-lem}, we can (replacing $\leg$ by $\leg \cap \leu$) assume without loss of generality that $\leg$ is extended by $\leu$. Let $\seq{M_i : i < \delta}$ be $\leu$-increasing in $K_{\is}$, $\delta = \cf{\delta} \ge \clc{<\alpha} (\is', \leg)$, $\delta < \lambda^+$. Without loss of generality, $\seq{M_i : i < \delta}$ is $\ltu$-increasing. Let $M_\delta := \bigcup_{i < \delta} M_i$ and let $p \in \gS^{<\alpha} (M_\delta)$. 

      By density, pick $M_0' \in K'$ such that $M_0 \ltu M_0'$. Now build $\seq{M_i' : i < \delta}$ $\ltg$-increasing in $K'$. Let $M_\delta' := \bigcup_{i < \delta} M_i'$. By Fact \ref{lim-uq}, there exists $f: M_\delta' \cong_{M_0} M_\delta$ such that for every $i < \delta$ there exists $j < \delta$ with $f[M_i'] \lea M_j$, $f^{-1}[M_i] \lea M_j'$. By definition of $\clc{<\alpha} (\is', \leg)$, there exists $i < \delta$ such that $f^{-1}(p)$ does not $\is'$-fork over $M_i'$. Let $j < \delta$ be such that $f[M_i'] \lea M_j$. By invariance, $p$ does not $\is'$-fork over $f[M_i]$, so does not $\is$-fork over $f[M_i]$. By base monotonicity, $p$ does not $\is$-fork over $M_j$, as desired.
    \item Similar.
  \end{enumerate}
\end{proof}

The last lemma of this section investigates what properties directly transfer up. 

\begin{lem}\label{up-indep-prop}
  Let $(K, \is)$ be a $\lambda$-generator for a $(<\alpha)$-independence relation. Let $\is' := \Kupp{(K, \is)} \rest \Ksatpp{(\Kup)}{\lambda^+}$.

  \begin{enumerate}
    \item If $\is$ has base monotonicity, then $\is'$ has base monotonicity.
    \item Assume $\is$ has base monotonicity and $(K, \is)$ has weak chain local character. Then:

      \begin{enumerate}
      \item\label{77-a} $\slc{<\alpha} (\is') = \lambda^{++}$.
      \item\label{77-b} If $\leg$ is an ordering such that $(K_{\is}, \leg)$ is a skeleton of $K$, then for any $\alpha_0 < \alpha$, $\clc{\alpha_0} (\is') \le \clc{\alpha_0} (\is, \leg)$.
      \end{enumerate}
  \end{enumerate}
\end{lem}
\begin{proof} \
  \begin{enumerate}
    \item As in \cite[Claim II.2.11.(3)]{shelahaecbook}
    \item This is a generalization of the proof of \cite[Lemma 4.11]{ss-tame-jsl} (itself a variation on \cite[Claim II.2.11.(5)]{shelahaecbook}) but we have to say slightly more so we give the details. Let $\leg^0$ be an ordering witnessing weak chain local character. We first prove (\ref{77-b}). Fix $\alpha_0 < \alpha$, and assume $\clc{\alpha_0} (\is, \leg) < \infty$. Then by definition $\clc{\alpha_0} (\is, \leg) \le \lambda$. Let $\delta = \cf{\delta} \ge \clc{\alpha_0} (\is, \leg)$. 

Let $\seq{M_i : i < \delta}$ be increasing in $\Ksatp{\lambda^+}$ and write $M_\delta := \bigcup_{i < \delta} M_i$ (note that we do not claim $M_\delta \in \Ksatp{\lambda^+}$. However, $M_\delta \in K_{\ge \lambda}$). Let $p \in \gS^{\alpha_0} (M_\delta)$. We want to find $i < \delta$ such that $p$ does not fork over $M_i$. There are two cases:

      \begin{itemize}
        \item \underline{Case 1: $\delta < \lambda^+$}: 

          We imitate the proof of \cite[Claim II.2.11.(5)]{shelahaecbook}. Assume the conclusion fails. Build $\seq{N_i :i < \delta}$ $\leg$-increasing in $K_{\is}$, $\seq{N_i' : i < \delta}$ $\lea$-increasing in $K_{\is}$ such that for all $i < \delta$:
          \begin{enumerate}
          \item $N_i \lea M_i$.
          \item \label{lcproof-ni-mdelta} $N_i \lea N_i' \lea M_\delta$.
          \item \label{lcproof-forkcond} $p \rest N_i'$ $\is$-forks over $N_i$.
          \item \label{lcproof-ni-contain} $\bigcup_{j < i} (|N_j'| \cap |M_j|) \subseteq |N_i|$.
          \end{enumerate}

          \paragraph{\underline{This is possible}}
          Assume $N_j$ and $N_j'$ have been constructed for $j < i$. Choose $N_i \lea M_i$ satisfying (\ref{lcproof-ni-contain}) so that $N_j \leg N_i$ for all $j < i$ (This is possible: use that $M_i$ is $\lambda^+$-saturated and that in skeletons of AECs, chains have upper bounds). By assumption, $p$ $\is'$-forks over $M_i$, and so by definition of forking there exists $N_i' \lea M_\delta$ in $K_{\is}$ such that $p \rest N_i'$ forks over $N_i$. By monotonicity, we can of course assume $N_i' \gea N_i$, $N_i' \gea N_j'$ for all $j < i$.

          \paragraph{\underline{This is enough}}
          Let $N_\delta := \bigcup_{i < \delta} N_i$, $N_\delta' := \bigcup_{i < \delta} N_i'$. By local character for $\is$, there is $i < \delta$ such that $p \rest N_\delta$ does not fork over $N_i$. By (\ref{lcproof-ni-mdelta}) and (\ref{lcproof-ni-contain}), $N_\delta' \lea N_\delta$. Thus by monotonicity $p \rest N_i'$ does not $\is$-fork over $N_i$, contradicting (\ref{lcproof-forkcond}).
        \item \underline{Case 2: $\delta \ge \lambda^+$}: Assume the conclusion fails. As in the previous case (in fact it is easier), we can build $\seq{N_i : i < \lambda^+}$ $\leg^0$-increasing in $K_{\is}$ such that $N_i \lea M_\delta$ and $p \rest N_{i + 1}$ $\is$-forks over $N_i$. Since $\is$ has weak chain local character, there exists $i < \lambda^+$ such that $p \rest N_{i + 1}$ does not $\is$-fork over $N_i$, contradiction.
      \end{itemize}

      For (\ref{77-a}), assume not: then there exists $M \in \Ksatp{\lambda^+}$ and $p \in \gS^{<\alpha} (M)$ such that for all $M_0 \lea M$ in $\Ksatp{\lambda^+}_{\lambda^+}$, $p$ $\is'$-forks over $M_0$. By stability, for any $A \subseteq |M|$ with $|A| \le \lambda$, there exists $M_0 \lea M$ containing $A$ which is $\lambda^+$-saturated of size $\lambda^+$. As in case 2 above, we build $\seq{N_i : i < \lambda^+}$ $\leg^0$-increasing in $K_{\is}$ such that $N_i \lea M$ and $p \rest N_{i + 1}$ $\is$-fork over $N_i$. This is possible (for the successor step, given $N_i$, take any $M_0 \lea M$ saturated of size $\lambda^+$ containing $N_i$. By definition of $\is'$ and the fact $p$ $\is'$-forks over $M_0$, there exists $N_{i + 1}' \lea M$ in $K_{\is}$ witnessing the forking. This can further extended to $N_{i + 1}$ which is as desired). This is enough: we get a contradiction to weak chain local character.
  \end{enumerate}
\end{proof}

\section{Weakly good independence relations}\label{ns-canon}

Interestingly, nonsplitting and $(<\kappa)$-coheir (for a suitable choice of $\kappa$) are already well-behaved if the AEC is stable. This raises the question of whether there is an object playing the role of a good frame (see the next section) in AECs that are stable but not superstable (whatever the exact meaning of superstability should be in this context, see Section \ref{sec-ss}). Note that \cite{bgkv-v3-toappear} proves the canonicity of independence relations that satisfy much less than the full properties of good frames, so it is reasonable to expect existence of such an object. The next definition comes from extracting all the properties we are able to prove from the construction of a good frame in \cite{ss-tame-jsl} assuming only stability.

\begin{defin}\label{weakly-good-def}\index{weakly good (independence relation or frame)}\index{pre-weakly good|see {weakly good (independence relation or frame)}}
Let $\is = (K, \nf)$ be a $(<\alpha, \F)$-independence relation, $\F = [\lambda, \theta)$. $\is$ is \emph{weakly good}\footnote{The name ``weakly good'' is admittedly not very inspired. A better choice may be to rename good independence relations to superstable independence relations and weakly independence relations to stable independence relations. We did not want to change Shelah's terminology here and wanted to make the relationship between ``weakly good'' and ``good'' clear.} if:
  
  \begin{enumerate}
    \item $K$ is nonempty, is $\lambda$-closed (Recall Definition \ref{lambda-closed-def}), and every chain in $K$ of ordinal length less than $\theta$ has an upper bound.
    \item $K$ is stable in $\lambda$.
    \item $\is$ has base monotonicity, disjointness, existence, and transitivity.
    \item $\pre (\is)$ has uniqueness.
    \item $\is$ has the left $\lambda$-witness property and the right $\lambda$-model-witness property.
    \item Local character: For all $\alpha_0 < \min(\lambda^+, \alpha)$, $\slc{\alpha_0} (\is) = \lambda^+$.
    \item Local extension and uniqueness: $\is_{\lambda}^{<\lambda^+}$ has extension and uniqueness.
  \end{enumerate}

  We say a pre-$(<\alpha, \F)$-frame $\s$ is \emph{weakly good} if $\cl (\s)$ is weakly good. $\is$ is \emph{pre-weakly good} if $\pre (\is)$ is weakly good.
\end{defin}
\begin{remark}\label{weakly-good-rmk}
  By Propositions \ref{indep-props}.(\ref{indep-props-exi}), \ref{indep-props}.(\ref{indep-props-wit}), existence and the right $\lambda$-witness property follow from the others.
\end{remark}

Our main tool to build weakly good independence relations will be to start from a $\lambda$-generator (see Definition \ref{generator-def}) which satisfies some additional properties:

\begin{defin}\index{generator for a weakly good independence relation}\index{$\lambda$-generator for a weakly good $(<\alpha)$-independence relation|see {generator for a weakly good independence relation}}
    $(K, \is)$ is a \emph{$\lambda$-generator for a weakly good $(<\alpha)$-independence relation} if:

    \begin{enumerate}
      \item $(K, \is)$ is a $\lambda$-generator for a $(<\alpha)$-independence relation.
      \item $K$ is nonempty, has no maximal models, and is stable in $\lambda$.
      \item $\Ksatpp{(\Kup)}{\lambda^+}$ is $\lambda$-tame for types of length less than $\alpha$.
      \item $\is$ has base monotonicity, existence, and is extended by $\lambda$-nonsplitting: whenever $p \in \gS^{<\alpha} (M)$ does not $\is$-fork over $M_0 \lea M$, then $p$ does not $\snsp{\lambda} (K_{\is})$-fork over $M_0$.
      \item $(K, \is)$ has weak chain local character.
    \end{enumerate}
\end{defin}

Both coheir and $\lambda$-nonsplitting induce a generator for a weakly good independence relation:

\begin{prop}\label{weakly-good-examples} 
  Let $K$ be an AEC with amalgamation and let $\lambda \ge \LS (K)$ be such that $K_\lambda$ is nonempty, has no maximal models, and $K$ is stable in $\lambda$. Let $2 \le \alpha \le \lambda^+$.
  \begin{enumerate}
    \item Let $\LS (K) < \kappa \le \lambda$. Assume that $K$ is $(<\kappa)$-tame and short for types of length less than $\alpha$. Let $\is := \left(\isch{\kappa} (K)\right)^{<\alpha}$.
      \begin{enumerate}
      \item If $K$ does not have the $(<\kappa)$-order property of length $\kappa$, $\clc{<\alpha} (\is) \le \lambda^+$, and $\Ksatp{\kappa}_{\lambda}$ is dense in $K_\lambda$, then $(K_\lambda, \is_{\lambda})$ is a $\lambda$-generator for a weakly good $(<\alpha)$-independence relation.
      \item If $\kappa = \beth_\kappa$, $(\alpha_0 + 2)^{<\kappap} \le \lambda$ for all $\alpha_0 < \alpha$, then $(K_\lambda, \is_{\lambda})$ is a $\lambda$-generator for a weakly good $(<\alpha)$-independence relation.
      \end{enumerate}
    \item Assume $\alpha \le \omega$ and $\Ksatp{\lambda^+}$ is $\lambda$-tame for types of length less than $\alpha$. Then $\left(K_\lambda, \left(\insp{\lambda} (K_\lambda)\right)^{<\alpha}\right)$ is a $\lambda$-generator for a weakly good $(<\alpha)$-independence relation.
    \item Let $K'$ be a dense sub-AC of $K$ such that $\Ksatp{\lambda^+} \subseteq K'$ and let $\is$ be a $(<\alpha, \ge \lambda)$-independence relation with $K_{\is} = K'$, such that $\pre (\is)$ has uniqueness, $\is$ has base monotonicity, and $\slc{<\alpha} (\is) = \lambda^+$. If $K_\lambda'$ is dense in $K_\lambda$, then $(K_\lambda, \is_\lambda)$ is a $\lambda$-generator for a weakly good $(<\alpha)$-independence relation.
  \end{enumerate}
\end{prop}
\begin{proof} \
  \begin{enumerate}
    \item \begin{enumerate}
        \item By Fact \ref{coheir-syn}, $\is$ has base monotonicity, existence, and uniqueness. By Fact \ref{splitting-basics}.(\ref{splitting-basics-3}), coheir is extended by $\lambda$-nonsplitting. The other properties are easy. For example, weak chain local character follows from $\clc{<\alpha} (\is) \le \lambda^+$ and monotonicity.
        \item We check that $K$ and $\is$ satisfy all the conditions of the previous part. By Fact \ref{op-facts}, $K$ does not have the $(<\kappa)$-order property of length $\kappa$. By (the proof of) Proposition \ref{indep-props}.(\ref{indep-props-lc}) and Fact \ref{coheir-syn}: 

          $$
          \clc{<\alpha} (\is) \le \slc{<\alpha} (\is) \le \sup_{\alpha_0 < \alpha} \left((\alpha_0 + 2)^{<\kappap}\right)^+ \le \lambda^+
          $$
          
          Since $K$ is stable in $\lambda$, if $\kappa < \lambda$ then $\Ksatp{\kappa}_\lambda$ is dense in $K_\lambda$. If $\kappa = \lambda$, then $\kappa = 2^{<\kappap}$ so is regular, hence strongly inaccessible, so $\kappa = \kappa^{<\kappa}$ so again it is easy to check that $\Ksatp{\kappa}_\lambda$ is dense in $K_\lambda$.
    \end{enumerate}
    \item Let $\is := (\snsp{\lambda} (K))^{<\alpha}$. By Fact \ref{splitting-basics}.(\ref{splitting-basics-25}) and Proposition \ref{indep-props}.(\ref{indep-props-lc}), $\clc{<\alpha} (\is) = \lambda^+$. By monotonicity, weak chain local character follows. The other properties are easy to check. 
    \item By Fact \ref{splitting-basics}.(\ref{splitting-basics-3}), $\is$ is extended by $\lambda$-nonsplitting. Weak chain character follows from $\clc{<\alpha} (\is) = \lambda^+$. By (the proof of) Proposition \ref{class-props}, $\Ksatp{\lambda^+}$ is $\lambda$-tame for types of length less than $\alpha$. The other properties are easy to check.
  \end{enumerate}
\end{proof}

The next result is that a generator for a weakly good independence relation indeed induces a weakly good independence relation.

\begin{thm}\label{weakly-good-up}
  Let $(K, \is)$ be a $\lambda$-generator for a weakly good $(<\alpha)$-independence relation. Then $\Kupp{(K, \is)} \rest \Ksatpp{(\Kup)}{\lambda^+}$ is a pre-weakly good $(<\alpha, \ge \lambda^+)$-independence relation.
\end{thm}
\begin{proof}
  This follows from the arguments of \cite{ss-tame-jsl}, but we give some details. Let $\is' := \Kupp{(K, \is)} \rest \Ksatpp{(\Kup)}{\lambda^+}$. Let $\nf := \nf_{\is'}$, $\s' := \pre (\is')$. We check the conditions in the definition of a weakly good independence relation. Note that by Remark \ref{weakly-good-rmk} we do not need to check existence or the right $\lambda^+$-witness property.

  \begin{itemize}
    \item $\is'$ is a $(<\alpha, \ge \lambda^+)$-independence relation: By Lemma \ref{up-indep-rel}.
    \item $K_{\is'}$ is stable in $\lambda^+$: By Fact \ref{op-facts}, $\Kup$ is stable in $\lambda^+$. By stability, $K_{\is'}$ is dense in $K$ so by Proposition \ref{skel-transfer}, $K_{\is'}$ is stable in $\lambda^+$.
    \item $K_{\is'} \neq \emptyset$ since it is dense in $\Kup_{\lambda^+}$ and $\Kup_\lambda = K$ is nonempty and has no maximal models. Every chain $\seq{M_i : i < \delta}$ in $K_{\is'}$ has an upper bound: we have $M_\delta := \bigcup_{i < \delta} M_i \in K$, and by density there exists $M \gea M_\delta$ in $K_{\is'}$. $K_{\is'}$ is $\lambda^+$-closed by an easy increasing chain argument, using stability in $\lambda^+$.
    \item Local character: $\slc{<\alpha} (\is') = \lambda^{++}$ by Lemma \ref{up-indep-prop}.
    \item $\s'$ has:
      \begin{itemize}
        \item Base monotonicity: By Lemma \ref{up-indep-prop}.
        \item Uniqueness: First observe that using local character, base monotonicity, $\lambda^+$-closure, and the fact that $K_{\is'}$ is $\lambda^+$-tame for types of length less than $\alpha$, it is enough to show uniqueness for $(\s')_{\lambda^+}$. For this imitate the proof of \cite[Lemma 5.3]{ss-tame-jsl} (the key is weak uniqueness: Fact \ref{splitting-basics}.(\ref{splitting-weak-props})).
        \item Local extension: Let $p \in \gS^{<\alpha} (M)$, $M_0 \lea M \lea N$ be in $(K_{\is'})_{\lambda^+}$ such that $p$ does not fork over $M_0$. Let $M_0' \lea M_0$ be in $K_{\is}$ and witness it. By homogeneity, $M_0' \ltu M$ so there exists $f: N \xrightarrow[M_0']{} M$. Let $q := f^{-1} (p) \rest N$. By invariance, $q$ does not fork over $M_0$ (as witnessed by $M_0'$). Since $\lambda$-nonsplitting extends nonforking, $p \rest M'$ does not $\snsp{\lambda} (K_{\is})$-fork over $M_0'$ whenever $M_0' \lea M' \lea M$ is such that $M' \in K_{\is}$. Let $K' := K_{\is} \cup \Ksatp{\lambda^+}$. By (the proof of) Fact \ref{splitting-basics}.(\ref{splitting-basics-tameness}), $p$ does not $\sns (K')$-fork over $M_0'$. By weak extension (Fact \ref{splitting-basics}.(\ref{splitting-weak-props}), $q$ extends $p$ and is algebraic if and only if $q$ is.
        \item Transitivity: Imitate the proof of \cite[Lemma 4.10]{ss-tame-jsl}. 
        \item Disjointness: It is enough to prove it for types of length 1 so assume $\alpha = 2$. Assume $\nfs{M_0}{a}{M}{N}$ (with $M_0 \lea M \lea N$ in $\Kmhp{\lambda^+}$) and $a \in M$. We show $a \in M_0$. Using local character, we can assume without loss of generality that $\|M_0\| = \lambda^+$ and (by taking a submodel of $M$ containing $a$ of size $\lambda^+$) that also $\|M\| = \lambda^+$. Find $M_0' \lea M_0$ in $K_{\is}$ witnessing the nonforking. By the proof of local extension, we can find $p \in \gS (M)$ extending $p_0 := \gtp (a / M_0; N)$ such that $p_0$ is algebraic if and only if $p$ is. Since $a \in N$, we must have by uniqueness that $p$ is algebraic so $p_0$ is algebraic, i.e.\ $a \in M_0$.
      \end{itemize}
      
      Now by Proposition \ref{cl-basics}, $\cl (\s')$ has the above properties.
    \item $\cl (\s')$ has the left $\lambda$-witness property: Because $\alpha \le \lambda^+$.
  \end{itemize}
\end{proof}

Interestingly, the generator can always be taken to have a particular form:

\begin{lem}\label{gen-canon}
  Let $(K, \is)$ be a $\lambda$-generator for a weakly good $(<\alpha)$-independence relation. Let $\is' := \insp{\lambda} (K)^{<\alpha}$. Then:

  \begin{enumerate}
    \item $(K, \is')$ is a $\lambda$-generator for a weakly good $(<\alpha)$-independence relation and $\ltu$ is the ordering witnessing weak chain local character.
    \item $\pre (\Kupp{(K, \is)}) \rest \Ksatpp{(\Kup)}{\lambda^+} = \pre (\Kupp{(K, \is')}) \rest \Ksatpp{(\Kup)}{\lambda^+}$.
  \end{enumerate}
\end{lem}
\begin{proof} \
  \begin{enumerate}
    \item By Lemma \ref{univ-lc} (with $K$, $\is'$, $K_{\is}$ here standing for $K, \is$, $K'$ there), $(K, \is')$ has weak chain local character (witnessed by $\ltu$) and the other properties are easy to check.
    \item Let $\s := \pre (\Kupp{(K, \is)}) \rest \Ksatpp{(\Kup)}{\lambda^+}$, $\s' := \pre (\Kupp{(K, \is')}) \rest \Ksatpp{(\Kup)}{\lambda^+}$. We want to see that $\nf_{\s} = \nf_{\s'}$. Since $\pre(\is)$ is extended by $\lambda$-nonsplitting, it is easy to check that $\nf_{\s} \subseteq \nf_{\s'}$. By the proof of \cite[Lemma 4.1]{bgkv-v3-toappear}, $\nf_{\s_{\lambda^+}} = \nf_{\s_{\lambda^+}'}$. By the right $\lambda$-model-witness property, $\nf_{\s} = \nf_{\s'}$.
  \end{enumerate}
\end{proof}

In Theorem \ref{weakly-good-up}, $\is' := \Kupp{(K, \is)} \rest \Ksatpp{(\Kup)}{\lambda^+}$ is only pre-weakly good, not necessarily weakly good: in general, only $\is'' := \cl (\pre (\is'))$ will be weakly good. The following technical lemma shows that $\is'$ and $\is''$ agree on slightly more than $\pre (\is')$.

\begin{lem}\label{pre-cl-generator}
  Let $(K, \is)$ be a $\lambda$-generator for a weakly good $(<\alpha)$-independence relation. Let $\is' := \Kupp{(K, \is)} \rest \Ksatpp{(\Kup)}{\lambda^+}$ and let $\is'' := \cl (\pre (\is'))$. Let $M \lea N$ be in $\Kup_{\ge \lambda^+}$ with $M \in \Ksatp{\lambda^+}$ (but maybe $N \notin \Ksatp{\lambda^+}$). Assume $\Kup$ is $\lambda$-tame\footnote{Note that the definition of a generator for a weakly good independence relation only requires that $\Ksatpp{(\Kup)}{\lambda^+}$ be $\lambda$-tame for types of length less than $\alpha$.} for types of length less than $\alpha$. Let $p \in \gS^{<\alpha} (N)$. 

  If $\|N\| = \lambda^+$ or $\is''$ has extension, then $p$ does not $\is'$-fork over $M$ if and only if $p$ does not $\is''$-fork over $M$.
\end{lem}
\begin{proof}
  Assume $p$ does not $\is''$-fork over $M$. Then by definition there exists an extension of $p$ to a model in $\Ksatp{\lambda^+}$ that does not $\is'$-fork over $M$ so by monotonicity $p$ does not $\is'$-fork over $M$. Assume now that $p$ does not $\is'$-fork over $M$. Note that the proof of Theorem \ref{weakly-good-up} (more precisely \cite[Lemma 5.3]{ss-tame-jsl}) implies that $p$ is the unique type over $N$ that does not $\is'$-fork over $M$. 

  Pick $N' \gea N$ in $\Ksatp{\lambda^+}$ with $\|N'\| = \|N\|$. We imitate the proof of \cite[Lemma 4.1]{bgkv-v3-toappear}. By extension (or local extension if $\|N\| = \lambda^+$, recall that $\is''$ is weakly good, see Theorem \ref{weakly-good-up}), there exists $q \in \gS^{<\alpha} (N')$ that does not $\is''$-fork over $M$ and extends $p \rest M$. By the above, $q$ does not $\is'$-fork over $M$. By uniqueness, $q$ extends $p$, so $q \rest N = p$ does not $\is''$-fork over $M$.
\end{proof}

Note that if the independence relation of the generator is coheir, then the weakly good independence relation obtained from it is also coheir. We first prove a slightly more abstract lemma:

\begin{lem}\label{weakly-good-coheir-0}
  Let $K$ be an AEC, $\lambda \ge \LS (K)$. Let $K'$ be a dense sub-AC of $K$ such that $\Ksatp{\lambda^+} \subseteq K'$ and $K_\lambda'$ is dense in $K_\lambda$. Let $\is$ be a $(<\alpha, \ge \lambda)$-independence relation with base monotonicity and $K_{\is} = K'$, $2 \le \alpha \le \lambda^+$. Assume that $\is$ has base monotonicity and the right $\lambda$-model-witness property.

  Assume $\slc{<\alpha} (\is) = \lambda^+$ and $(K_\lambda, \is_\lambda)$ is a $\lambda$-generator for a weakly good $(<\alpha)$-independence relation. Let $\is' := \Kupp{(K_\lambda, \is_\lambda)} \rest \Ksatp{\lambda^+}$. Then $\pre (\is') = \pre (\is) \rest \Ksatp{\lambda^+}$.

  Moreover if $\is$ has the right $\lambda$-witness property, then $\is' = \is \rest \Ksatp{\lambda^+}$.
\end{lem}
\begin{proof} 
  We prove the moreover part and it will be clear how to change the proof to prove the weaker statement (just replace the use of the witness property by the model-witness property).
  
  Let $M \lea N$ be in $\Ksatp{\lambda^+}$, $p \in \gS^{<\alpha} (B; N)$. We want to show that $p$ does not $\is$-fork over $M$ if and only if there exists $M_0 \lea M$ in $K_\lambda'$ so that for all $B_0 \subseteq B$ of size $\le \lambda$, $p \rest B_0$ does not $\is$-fork over $M_0$. 
      Assume first that $p$ does not $\is$-fork over $M$. Since $\slc{<\alpha} (\is) = \lambda^+$, there exists $M_0 \lea M$ in $K_\lambda$ such that $p$ does not $\is$-fork over $M_0$. By base monotonicity and homogeneity, we can assume that $M_0 \in K_\lambda'$. I particular $p \rest B_0$ does not $\is$-fork over $B_0$ for all $B \subseteq B$ of size $\le \lambda$.

      Conversely, assume $p$ does not $\is'$-fork over $M$, and let $M_0 \lea M$ in $K_\lambda'$ witness it. Then by the right $\lambda$-witness property, $p$ does not $\is$-fork over $M_0$, so over $M$, as desired.

\end{proof}

\begin{lem}\label{weakly-good-coheir}
  Let $K$ be an AEC, $\LS (K) < \kappa \le \kappa' \le \lambda$. Let $2 \le \alpha \le \lambda^+$. Let $\is := \left(\isch{\kappa} (K)\right)_{\ge \lambda}^{<\alpha} \rest \Ksatp{\kappa'}$. 
  
  Assume $\slc{<\alpha} (\is) = \lambda^+$ and $(K_\lambda, \is_\lambda)$ is a $\lambda$-generator for a weakly good $(<\alpha)$-independence relation. Let $\is' := \Kupp{(K_\lambda, \is_\lambda)} \rest \Ksatp{\lambda^+}$. Then $\is' = \is \rest \Ksatp{\lambda^+}$.
\end{lem}
\begin{proof}
By Lemma \ref{weakly-good-coheir-0} applied with $K' = \Ksatp{\kappa'}$.
\end{proof}

We end this section by showing how to build a weakly good independence relation in any stable fully tame and short AEC (with amalgamation and no maximal models).

\begin{thm}
  Let $K$ be a $\LS (K)$-tame AEC with amalgamation and no maximal models. Let $\kappa = \beth_\kappa > \LS (K)$. Assume $K$ is stable and $(<\kappa)$-tame and short for types of length less than $\alpha$, $\alpha \ge 2$.

  If $K_\kappa \neq \emptyset$, then $\isch{\kappa} (K)^{<\alpha} \rest \Ksatp{(2^{\kappa})^+}$ is a pre-weakly good $(<\alpha, \ge (2^{\kappa})^+)$-independence relation. Moreover if $\alpha = \infty$, then it is weakly good.
\end{thm}
\begin{proof}
  Let $\lambda := 2^{\kappa}$. By Fact \ref{coheir-syn}, $\isch{\kappa} (K)^{<\alpha} \rest \Ksatp{\lambda^+}$ already has several of the properties of a weakly good independence relation, and in particular has the left $\lambda$-witness property so it is enough to check that $\is := \isch{\kappa} (K)^{<(\min(\alpha, \lambda^+))} \rest \Ksatp{\lambda^+}$ is weakly good, so assume now without loss of generality that $\alpha \le \lambda^+$. Note that by Fact \ref{coheir-syn}, $\slc{<\alpha} (\is) \le (\lambda^{\kappa})^+ = \lambda^+$. By Lemma \ref{weakly-good-coheir} it is enough to check that $(K_\lambda, \is_\lambda)$ is a $\lambda$-generator for a weakly good $(< \alpha)$-independence relation. From Fact \ref{op-facts}, we get that $K$ is stable in $\lambda$. Finally, note that $K_\lambda \neq \emptyset$. Now apply Proposition \ref{weakly-good-examples}.

  If $\alpha = \infty$, then by Fact \ref{coheir-syn}, $\is$ has uniqueness. Since $\is$ is pre-weakly good, $\pre (\is_\lambda)$ has extension, so by Proposition \ref{cl-basics}.(\ref{cl-basics-ext}), $\is_\lambda$ also has extension. The other properties of a weakly good independence relation follow from Fact \ref{coheir-syn}.
\end{proof}

\section{Good independence relations}\label{good-sec}

Good frames were introduced by Shelah \cite[Definition II.2.1]{shelahaecbook} as a ``bare bone'' definition of superstability in AECs. Here we adapt Shelah's definition to independence relations. We also define a variation, being \emph{fully} good. This is only relevant when the types are allowed to have length at least $\lambda$, and asks for more continuity (like in \cite{tame-frames-revisited-v5}, but the continuity property asked for is different). This is used to enlarge a good frame in the last sections.

\begin{defin}\label{goodness-def}\index{good (independence relation or frame)} \index{pre-good|see {good (independence relation or frame)}}\index{fully good|see {good (independence relation or frame)}} \index{fully pre-good|see {good (independence relation or frame)}} \
  \begin{enumerate}
    \item A \emph{good $(<\alpha, \F)$-independence relation} $\is = (K, \nf)$ is a $(<\alpha, \F)$-independence relation satisfying:
      \begin{enumerate}
        \item $K$ is a nonempty AEC in $\F$, $\LS (K) = \lambda_{\is}$, $K$ has no maximal models and joint embedding, $K$ is stable in all cardinals in $\F$.
        \item $\is$ has base monotonicity, disjointness, symmetry, uniqueness, existence, extension, the left $\lambda_{\is}$-witness property, and for all $\alpha_0 < \alpha$ with $|\alpha_0|^+ < \theta_{\is}$, $\clc{\alpha_0} (\is) = |\alpha_0|^+ + \aleph_0$ and $\slc{\alpha_0} (\is) = |\alpha_0|^+ + \lambda_{\is}^+$.
      \end{enumerate}
    \item A \emph{type-full good $(<\alpha, \F)$-frame} $\s$ is  a pre-$(<\alpha, \F)$-frame so that $\cl (\s)$ is good.
    \item $\is$ is \emph{pre-good} if $\pre (\is)$ is good. 
  \end{enumerate}

  When we add ``fully'', we require in addition that the frame/independence relation satisfies full model-continuity.
\end{defin}
\begin{remark}
  This paper's definition is equivalent to that of Shelah \cite[Definition II.2.1]{shelahaecbook} if we remove the requirement there on the existence of a superlimit (as was done in almost all subsequent papers, for example in \cite{jrsh875}) and assume the frame is type-full\index{type-full} (i.e.\ the basic types are all the nonalgebraic types). For example, the continuity property that Shelah requires follows from $\clc{1} (\s) = \aleph_0$ (\cite[Claim II.2.17.(3)]{shelahaecbook}).
\end{remark}
\begin{remark}
  If $\is$ is a good $(<\alpha, \F)$-independence relation (except perhaps for the symmetry axiom) then $\is$ is weakly good.
\end{remark}

\begin{defin}\label{goodness-def-2}\index{good (AEC)}\index{$(<\alpha, \F)$-good|see {good (AEC)}} \index{fully good (AEC)|see {good (AEC)}} \index{fully $(<\alpha, \F)$-good|see {good (AEC)}}
  An AEC $K$ is \emph{[fully] $(<\alpha, \F)$-good} if there exists a [fully] $(<\alpha, \F)$-good independence relation $\is$ with $K_{\is} = K$. When $\alpha = \infty$ and $\F = [\LS (K), \infty)$, we omit them.
\end{defin}

As in the previous section, we give conditions for a generator to induce a good independence relation:

\begin{defin}\index{generator for a good independence relation}\index{$\lambda$-generator for a good $(<\alpha)$-independence relation|see {generator for a good independence relation}}
    $(K, \is)$ is a \emph{$\lambda$-generator for a good $(<\alpha)$-independence relation} if:

    \begin{enumerate}
      \item $(K, \is)$ is a $\lambda$-generator for a weakly good $(<\alpha)$-independence relation.
      \item $\Kup$ is $\lambda$-tame.
      \item There exists $\mu \ge \lambda$ such that $\Kup_{\mu}$ has joint embedding.
      \item Local character: For all $\alpha_0 < \min(\alpha, \lambda)$, there exists an ordering $\leg$ such that $(K_{\is}, \leg)$ is a skeleton of $K$ and $\clc{\alpha_0} (\is, \leg) = |\alpha_0|^+ + \aleph_0$.
    \end{enumerate}
\end{defin}
\begin{remark}
  If $(K, \is)$ is a $\lambda$-generator for a good $(<\alpha)$-independence relation, then it is a $\lambda$-generator for a weakly good $(<\alpha)$-independence relation. Moreover if $\alpha < \lambda^+$, the weak chain local character axiom follows from the local character axiom. 
\end{remark}

As before, the generator can always be taken to be of a particular form: 

\begin{lem}\label{gen-canon-good}
  Let $(K, \is)$ be a $\lambda$-generator for a good $(<\alpha)$-independence relation. Let $\is' := \insp{\lambda} (K)^{<\alpha}$. Then:

  \begin{enumerate}
    \item $(K, \is')$ is a $\lambda$-generator for a good $(<\alpha)$-independence relation and $\ltu$ is the ordering witnessing local character.
    \item $\pre (\Kupp{(K, \is)}) \rest \Ksatpp{(\Kup)}{\lambda^+} = \pre (\Kupp{(K, \is')}) \rest \Ksatpp{(\Kup)}{\lambda^+}$.
  \end{enumerate}
\end{lem}
\begin{proof} \
  \begin{enumerate}
    \item By Lemma \ref{univ-lc} (with $K$, $\is'$, $K_{\is}$ here standing for $K, \is$, $K''$ there), $(K, \is')$ has the local character properties, witnessed by $\ltu$, and the other properties are easy to check.
    \item By Lemma \ref{gen-canon}.
  \end{enumerate}
\end{proof}

Unfortunately it is not strictly true that a generator for a good $(<\alpha)$-independence relation induces a good independence relation. For one thing, the extension property is problematic when $\alpha > \omega$ and this in turn creates trouble in the proof of symmetry. Also, we are unable to prove $\Ksatp{\lambda^+}$ is an AEC (although we suspect it should be true, see also Fact \ref{bv-sat-fact}). For the purpose of stating a clean result, we introduce the following definition: 

\begin{defin}\label{almost-good-def}\index{almost pre-good (for an independence relation)}
  $\is$ is an \emph{almost pre-good $(<\alpha, \F)$-independence relation} if:
  \begin{enumerate}
    \item It is a pre-weakly good $(<\alpha, \F)$-independence relation. 
    \item It satisfies all the conditions in the definition of a \emph{pre-good} independence relation except that: 
  \begin{enumerate}
    \item $K_{\is}$ is not required to be an AEC.
    \item $\cl (\pre(\is))$ is not required to have extension or uniqueness, but we still ask that $\pre (\is^{<\omega})$ has extension.
    \item $\cl (\pre (\is))$ is not required to have symmetry, but we still require that $\pre (\is^{<\omega})$ has full symmetry.
    \item We replace the condition on $\clc{\alpha_0} (\cl (\pre (\is)))$ by:
      \begin{enumerate}
      \item\label{almost-good-cl-pre} $\clc{<(\min (\alpha, \omega))} (\cl (\pre (\is))) = \aleph_0$.
      \item For all $\alpha_0 < \alpha$, $\clc{\alpha_0} (\is) = |\alpha_0|^+ + \aleph_0$.
      \end{enumerate}
  \end{enumerate}
  \end{enumerate}
\end{defin}

\begin{thm}\label{good-up}
  Let $(K, \is)$ be a $\lambda$-generator for a good $(<\alpha)$-independence relation. Then:

  \begin{enumerate}
    \item\label{good-up-1} $\Kup$ has joint embedding and no maximal models.
    \item\label{good-up-2} $\Kup$ is stable in every $\mu \ge \lambda$.
    \item\label{good-up-3} $\is' := \Kupp{(K, \is)} \rest \Ksatpp{(\Kup)}{\lambda^+}$ is an almost pre-good $(<\alpha, \ge \lambda^+)$-independence relation.
    \item\label{good-up-4} If $\alpha \le \omega$ and $\mu \ge \lambda^+$ is such that $\Ksatpp{(\Kup)}{\mu}$ is an AEC with Löwenheim-Skolem-Tarski number $\mu$, then $\left(\is'\right)^{<\alpha} \rest \Ksatpp{(\Kup)}{\mu}$ is a pre-good $(<\alpha, \ge \mu)$-independence relation.
  \end{enumerate}
\end{thm}
\begin{proof}
  Again, this follows from the arguments in \cite{ss-tame-jsl}, but we give some details. We show by induction on $\theta \ge \lambda^+$ that $\s' := \pre (\is')_{[\lambda^+, \theta]} $ is a good frame, except perhaps for symmetry and the conditions in Definition \ref{almost-good-def}. This gives (\ref{good-up-3}) (use Proposition \ref{indep-props}.(\ref{indep-props-sym}) to get symmetry, the proof of \cite[Lemma 5.9]{ss-tame-jsl} to get extension for types of finite length, and Lemma \ref{pre-cl-generator} to get (\ref{almost-good-cl-pre}) in Definition \ref{almost-good-def}), and (\ref{good-up-4}) together with (\ref{good-up-1}),(\ref{good-up-2}) (use Proposition \ref{skel-transfer}) follow.

  \begin{itemize}
    \item $\s'$ is a weakly good $(<\alpha, [\lambda^+, \theta])$-frame: By Theorem \ref{weakly-good-up}.

    \item Let $\mu \ge \lambda$ be such that $\Kup_{\mu}$ has joint embedding. By amalgamation, $\Kup_{\ge \mu}$ has joint embedding. Once it is shown that $\Kup$ has no maximal models, it will follow that $\Kup$ has joint embedding (every model of size $\ge \lambda$ extends to one of size $\mu$). Note that joint embedding is never used in any of the proofs below.
    \item To prove that $\Kup_{[\lambda, \theta]}$ has no maximal models, we can assume without loss of generality that $\alpha = 2$ and (by Lemma \ref{gen-canon-good}) that $\is = \insp{\lambda} (K)$, with $\clc{1} (\is, \ltu) = \aleph_0$. By the induction hypothesis (and the assumption that $K$ has no maximal models), $\Kup_{[\lambda, \theta)}$ has no maximal models. It remains to see that $\Kup_\theta$ has no maximal models. Assume for a contradiction that $M \in \Kup_\theta$ is maximal. Then it is easy to check that $M \in \Ksatpp{(\Kup)}{\theta}_\theta$. Build $\seq{M_i : i < \theta}$ increasing continuous and $a \in |M|$ such that for all $i < \theta$:
          \begin{enumerate}
          \item $M_i \lea M$.
          \item $M_i \ltu M_{i + 1}$.
          \item $M_i \in \Kup_{<\theta}$.
          \item $a \notin |M_i|$.
          \end{enumerate}
          
          \paragraph{\underline{This is enough}} Let $M_\theta := \bigcup_{i < \theta} M_i$. Note that $\|M_\theta\| = \theta$ and $a \in |M| \backslash |M_\theta|$, so $M_\theta \lta M$. By Lemma \ref{lem-univ}, $M_0 \ltu M_\theta$. Thus there exists $f: M \xrightarrow[M_0]{} M_\theta$ and since $M$ is maximal $f$ is an isomorphism. However $M$ is maximal whereas $M$ witnesses that $M_\theta$ is not maximal, so $M$ cannot be isomorphic to $M_\theta$, a contradiction.
          \paragraph{\underline{This is possible}} Imitate the proof of \cite[Lemma 5.12]{ss-tame-jsl} (this is where it is useful that the generator is nonsplitting and the local character is witnessed by $\ltu$).
    \item $\Kup$ is stable in all $\mu \in [\lambda^+, \theta]$: Exactly as in the proof of \cite[Theorem 5.6]{ss-tame-jsl}.
    \item $\s'$ has base monotonicity, disjointness, and uniqueness because it is weakly good. For all $\alpha_0 < \alpha$, $\clc{\alpha_0} (\is') = |\alpha_0|^+ + \aleph_0$, $\slc{\alpha_0} (\s') = |\alpha_0|^+ + \lambda^{++} = \lambda^{++}$ by Lemma \ref{up-indep-prop}.
  \end{itemize}
\end{proof}
\begin{remark}
  Our proof of no maximal models above improves on \cite[Conclusion 4.13.(3)]{shelahaecbook}, as it does not use the symmetry property. 
\end{remark}

\section{Canonicity}\label{canon-sec}

In \cite{bgkv-v3-toappear}, we gave conditions under which two independence relations are the same. There we strongly relied on the extension property, but coheir and weakly good frames only have a weak version of it. In this section, we show that if we just want to show two independence relations are the same over sufficiently saturated models, then the proofs become easier and the extension property is not needed. In addition, we obtain an explicit description of the forking relation. We conclude that coheir, weakly good frames, and good frames are (in a sense made precise below) canonical. This gives further evidence that these objects are not ad-hoc and answers several questions in \cite{bgkv-v3-toappear}. The results of this section are also used in Section \ref{sec-ss} to show the equivalence between superstability and strong superstability.

\begin{lem}[The canonicity lemma]\label{canon-lem}
    Let $K$ be an AEC with amalgamation and let $\lambda \ge \LS (K)$ be such that $K$ is stable in $\lambda$. Let $K'$ be a dense sub-AC of $K$ such that $\Ksatp{\lambda^+} \subseteq K'$ and $K_\lambda'$ is dense in $K_\lambda$. Let $\is$, $\is'$ be $(<\alpha, \ge \lambda)$-independence relation with $K_{\is} = K_{\is'} = K'$. Let $\alpha_0 := \min (\alpha, \lambda^+)$.

    If:

    \begin{enumerate}
      \item $\pre (\is)$ and $\pre (\is')$ have uniqueness.
      \item $\is$ and $\is'$ have base monotonicity, the left $\lambda$-witness property, and the right $\lambda$-model-witness property.
      \item $\slc{<\alpha_0} (\is) = \slc{<\alpha_0} (\is') = \lambda^+$.
    \end{enumerate}

    Then $\pre (\is) \rest \Ksatp{\lambda^+} = \pre (\is') \rest \Ksatp{\lambda^+}$, and if in addition both $\is$ and $\is'$ have the right $\lambda$-witness property, then $\is \rest \Ksatp{\lambda^+} = \is' \rest \Ksatp{\lambda^+}$.

    Moreover for $M \lea N$ in $\Ksatp{\lambda^+}$, $p \in \gS^{<\alpha} (N)$ does not $\is$-fork over $M$ if and only if for all $I \subseteq \ell (p)$ with $|I| \le \lambda$, there exists $M_0 \lea M$ in $K_\lambda'$ such that $p^I$ does not $\snsp{\lambda} (K')$-fork over $M_0$.
\end{lem}
\begin{proof}
  By Fact \ref{jep-partition}, we can assume without loss of generality that $K$ has joint embedding. If $K_{\lambda^+} = \emptyset$, there is nothing to prove so assume $K_{\lambda^+} \neq \emptyset$. Using joint embedding, it is easy to see that $K_\lambda$ is nonempty and has no maximal models. By the left $\lambda$-witness property, we can assume without loss of generality that $\alpha \le \lambda^+$, i.e.\ $\alpha = \alpha_0$. By Proposition \ref{weakly-good-examples}, $(K, \is)$ and $(K, \is')$ are $\lambda$-generators for a weakly good $(<\alpha)$-independence relation.   By Lemma \ref{gen-canon}, $\pre (\Kupp{(K, \is)}) \rest \Ksatp{\lambda^+} = \pre (\Kupp{(K, \is')}) \rest \Ksatp{\lambda^+}$.

By Lemma \ref{weakly-good-coheir-0}, for $x \in \{\is, \is'\}$, $\pre (\Kupp{(K, x)}) \rest \Ksatp{\lambda^+} = \pre (x) \rest \Ksatp{\lambda^+}$, so the result follows (the definition of $(K, x)_{\ge \lambda}$ and Lemma \ref{gen-canon} also give the moreover part). The moreover part of lemma \ref{weakly-good-coheir-0} says that if $x \in \{\is, \is'\}$ has the right $\lambda$-witness property, then $\Kupp{(K, x)} \rest \Ksatp{\lambda^+} = x \rest \Ksatp{\lambda^+}$, so in case both $\is$ and $\is'$ have the right $\lambda$-witness property, we must have $\is \rest \Ksatp{\lambda^+} = \is' \rest \Ksatp{\lambda^+}$.
\end{proof}
\begin{remark}
  If $K$ is an AEC with amalgamation, $K'$ is a dense sub-AC of $K$ such that $\Ksatp{\lambda^+} \subseteq K'$ and $K_\lambda'$ is dense in $K_\lambda$, and $\is$ is a $(\le 1, \ge \lambda)$-independence relation with $K_{\is} = K'$ and base monotonicity, uniqueness, $\slc{1} (\is) = \lambda^+$, then by the proof of Proposition \ref{class-props} and Lemma \ref{skel-transfer} $K$ is stable in any $\mu \ge \LS (K)$ with $\mu = \mu^\lambda$.
\end{remark}

\begin{thm}[Canonicity of coheir]\label{canon-coheir}
  Let $K$ be an AEC with amalgamation. Let $\kappa = \beth_\kappa > \LS (K)$. Assume $K$ is $(<\kappa)$-tame and short for types of length less than $\alpha$, $\alpha \ge 2$. 

  Let $\lambda \ge \kappa$ be such that $K$ is stable in $\lambda$ and $(\alpha_0 + 2)^{<\kappap} \le \lambda$ for all $\alpha_0 < \min (\lambda^+, \alpha)$. Let $\is$ be a $(<\alpha, \ge \lambda)$-independence relation so that:

  \begin{enumerate}
    \item $K_{\is}$ is a dense sub-AC of $K$ so that $\Ksatp{\lambda^+} \subseteq K_{\is}$ and $\left(K_{\is}\right)_\lambda$ is dense in $K_\lambda$.
    \item $\pre (\is)$ has uniqueness.
    \item $\is$ has base monotonicity, the left $\lambda$-witness property, and the right $\lambda$-model-witness property.
    \item $\slc{<\min (\lambda^+, \alpha)} (\is) = \lambda^+$.
  \end{enumerate}

  Then $\pre (\is) \rest \Ksatp{\lambda^+} = \pre (\isch{\kappa} (K)^{<\alpha}) \rest \Ksatp{\lambda^+}$. If in addition $\is$ has the right $\lambda$-witness property, then $\is \rest \Ksatp{\lambda^+} = \isch{\kappa} (K)^{<\alpha} \rest \Ksatp{\lambda^+}$.
\end{thm}
\begin{proof}
  By Fact \ref{jep-partition}, we can assume without loss of generality that $K$ has joint embedding. If $K_{\lambda^+} = \emptyset$, there is nothing to prove so assume $K_{\lambda^+} \neq \emptyset$. By Fact \ref{hanf-existence}, $K$ has arbitrarily large models so no maximal models. Let $\is' := \isch{\kappa} (K)^{<\alpha}$. By the proof of Proposition \ref{weakly-good-examples}, $\is' \rest K_{\is}$ satisfies the hypotheses of Lemma \ref{canon-lem}. Moreover, it has the right $(<\kappa)$-witness property so the result follows.
\end{proof}

\begin{thm}[Canonicity of weakly good independence relations]\label{canon-weakly-good}
    Let $K$ be an AEC with amalgamation and let $\lambda \ge \LS (K)$. Let $K'$ be a dense sub-AC of $K$ such that $\Ksatp{\lambda^+} \subseteq K'$ and $K_\lambda'$ is dense in $K_\lambda$. Let $\is$, $\is'$ be weakly good $(<\alpha, \ge \lambda)$-independence relations with $K_{\is} = K_{\is'} = K'$.

    Then $\pre (\is) \rest \Ksatp{\lambda^+} = \pre (\is') \rest \Ksatp{\lambda^+}$. If in addition both $\is$ and $\is'$ have the right $\lambda$-witness property, then $\is \rest \Ksatp{\lambda^+} = \is' \rest \Ksatp{\lambda^+}$.  
\end{thm}
\begin{proof}
  By definition of a weakly good independence relation, $K_\lambda'$ is stable in $\lambda$. Therefore by Lemma \ref{skel-transfer} $K_\lambda$, and hence $K$, is stable in $\lambda$. Now apply Lemma \ref{canon-lem}.
\end{proof}
\begin{thm}[Canonicity of good independence relations]\label{good-canon}
  If $\is$ and $\is'$ are good $(<\alpha, \ge \lambda)$-independence relations with the same underlying AEC $K$, then $\is \rest \Ksatp{\lambda^+} = \is' \rest \Ksatp{\lambda^+}$.
\end{thm}
\begin{proof}
  By Theorem \ref{canon-weakly-good} (with $K' := K$), $\pre (\is) \rest \Ksatp{\lambda^+} = \pre (\is') \rest \Ksatp{\lambda^+}$. Since good independence relations have extension, Lemma \ref{pre-to-cl} implies $\is \rest \Ksatp{\lambda^+} = \is' \rest \Ksatp{\lambda^+}$.
\end{proof}

Recall that \cite[Question 6.14]{bgkv-v3-toappear} asked if two good $\lambda$-frames with the same underlying AEC should be the same. We can make progress toward this question by slightly refining our arguments. Note that the results below can be further adapted to work for not necessarily type-full frames (that is for two good frames, in Shelah's sense, with the same basic types and the same underlying AEC).

\begin{lem}\label{good-frame-limit}
  Let $\s$ and $\s'$ be good $(<\alpha, \lambda)$-frames with the same underlying AEC $K$ and $\alpha \le \lambda$. Let $K'$ be the class of $\lambda$-limit models of $K$ (recall Definition \ref{univ-def}). Then $\s \rest K' = \s' \rest K'$.
\end{lem}
\begin{proof}[Proof sketch]
  By Remark \ref{lim-uq-rmk}, $I(K') = 1$. Now refine the proof of Theorem \ref{good-canon} by replacing $\lambda^+$-saturated models by $(\lambda, |\beta|^+ + \aleph_0)$-limit models for each $\beta < \alpha$. Everything still works since one can use the weak uniqueness and extension properties of nonsplitting (Fact \ref{splitting-basics}.(\ref{splitting-weak-props})).
\end{proof}

\begin{thm}[Canonicity of categorical good $\lambda$-frames]\label{good-frame-cor}
Let $\s$ and $\s'$ be good $(<\alpha, \lambda)$-frames with the same underlying AEC $K$ and $\alpha \le \lambda$. If $K$ is categorical in $\lambda$, then $\s = \s'$.
\end{thm}
\begin{proof}
  By Fact \ref{ltl-basic-props}, $K$ has a limit model, and so by categoricity any model of $K$ is limit. Now apply Lemma \ref{good-frame-limit}.
\end{proof}
\begin{remark}
  The proof also gives an explicit description of forking: For $M_0 \lea M$ with $M_0$ a limit model, $p \in \gS (M)$ does not $\s$-fork over $M_0$ if and only if there exists $M_0' \ltu M_0$ such that $p$ does not $\snsp{\lambda}$-fork over $M_0'$. Note that this is the definition of forking in \cite{ss-tame-jsl}.
\end{remark}

Note that Shelah's construction of a good $\lambda$-frame in \cite[Theorem II.3.7]{shelahaecbook} relies on categoricity in $\lambda$, so Theorem \ref{good-frame-cor} establishes that the frame there is canonical. We are still unable to show that the frame built in Theorem \ref{categ-frame} is canonical in general, although it will be if $\lambda$ is the categoricity cardinal or if it is weakly successful (by \cite[Theorem 6.13]{bgkv-v3-toappear}).

\section{Superstability}\label{sec-ss}

Shelah has pointed out \cite[p.~19]{shelahaecbook} that superstability in abstract elementary classes suffers from schizophrenia, i.e.\ there are several different possible definitions that are equivalent in elementary classes but not necessarily in AECs. The existence of a good $(\ge \lambda)$-frame is a possible candidate but it is very hard to check. Instead, one would like a simple definition that implies existence of a good frame.

Shelah claims in chapter IV of his book that solvability\footnote{One can ask whether there are any implications between this paper's definition of superstability and Shelah's. We leave this to future work.} (\cite[Definition IV.1.4]{shelahaecbook}) is such a notion, but his justification is yet to appear (in \cite{sh842}). Essentially, solvability says that certain EM models are superlimits. On the other hand previous work (for example \cite{sh394, shvi635, vandierennomax, nomaxerrata, gvv-toappear-v1_2}) all rely on a local character property for nonsplitting. This is even made into a definition of superstability in \cite[Definition 7.12]{grossberg2002}. In \cite{ss-tame-jsl} we gave a similar condition and used it with tameness to build a good frame. Shelah has shown \cite[Lemma I.6.3]{sh394} that categoricity in a cardinal of high-enough cofinality implies the superstability condition.

We now aim to show the same conclusion under categoricity in a high-enough cardinal of arbitrary cofinality. The following definition of superstability is implicit in \cite{shvi635} and stated explicitly in \cite[Definition 7.12]{grossberg2002}.

\begin{defin}[Superstability]\label{ss-def}\index{superstability}\index{superstable|see {superstability}}\index{$\mu$-superstable|see {superstability}}\index{$\ssp$|see {superstability}}\index{$\mu$-$\ssp$|see {superstability}}
  An AEC $K$ is \emph{$\mu$-superstable} if:

  \begin{enumerate}
  \item $\LS (K) \le \mu$.
  \item\label{cond-two-mods} There exists $M \in K_\mu$ such that for any $M' \in K_\mu$ there is $f: M' \rightarrow M$ with $f[M'] \ltu M$.
  \item $\clc{1} (\snsp{\mu} (K_\mu), \leu) = \aleph_0$.
  \end{enumerate}

  We say $K$ is \emph{$\mu$-$\ssp$} if $K_{\ge \mu}$ is $\mu$-superstable, has amalgamation, and is $\mu$-tame. We may omit $\mu$, in which case we mean there exists a value such that the definition holds, e.g.\ $K$ is \emph{superstable} if it is $\mu$-superstable for some $\mu$.
\end{defin}
\begin{remark}\label{equiv-two-mods}
  Using Fact \ref{ltl-basic-props}, it is easy to check that Condition (\ref{cond-two-mods}) above is equivalent to ``$K_\mu$ is nonempty, has amalgamation, joint embedding, no maximal models, and is stable in $\mu$''.
\end{remark}
\begin{remark}
  While Definition \ref{ss-def} makes sense in any AEC, here we focus on tame AECs with amalgamation, and will not study what happens to Definition \ref{ss-def} without these assumptions (although this can be done, see \cite{gvv-toappear-v1_2}). In other words, we will study ``$\ssp$'' rather than just ``superstable''.
\end{remark}

For technical reasons, we will also use the following version that uses coheir rather than nonsplitting.

\begin{defin}\label{ss-strong-def}\index{strong superstability}\index{strongly superstable|see {strong superstability}}\index{$\kappa$-strongly $\mu$-superstable|see {strong superstability}}\index{$\kappa$-strongly $\mu$-$\ssp$|see {strong superstability}}
  An AEC $K$ is \emph{$\kappa$-strongly $\mu$-superstable} if:
  \begin{enumerate}
    \item\label{ss-strong-def-1} $\LS (K) < \kappa \le \mu$.
    \item (\ref{cond-two-mods}) in Definition \ref{ss-def} holds.
    \item\label{ss-strong-def-5} $K$ does not have the $(<\kappa)$-order property of length $\kappa$.
    \item $\Ksatp{\kappa}_\mu$ is dense in $K_\mu$.
    \item $\clc{1} (\isch{\kappa} (K)_\mu, \leu) = \aleph_0$.
  \end{enumerate}
  
  As before, we may omit some parameters and say $K$ is \emph{$\kappa$-strongly $\mu$-$\ssp$} if there exists $\kappa_0 < \kappa$ such that $K_{\ge \kappa_0}$ is $\kappa$-strongly $\mu$-superstable, has amalgamation, and is $(<\kappa)$-tame.
\end{defin}

It is not too hard to see that a $\mu$-$\ssp$ AEC induces a generator for a good independence relation, but what if we have a generator of some other form (assume for example that $\ltu$ is replaced by $\ltl{\mu}{\delta}$ in the definition)? This is the purpose of the next definition. 

\begin{defin}
  Let $K$ be an AEC.
  \begin{enumerate}
    \item\index{$(\mu, \is)$-$\ssp$} $K$ is $(\mu, \is)$-$\ssp$ if $\LS (K) \le \mu$ and $(K_\mu, \is)$ is a $\mu$-generator for a good $(\le 1)$-independence relation.
    \item\index{$\kappa$-strongly $(\mu, \is)$-$\ssp$} $K$ is $\kappa$-strongly $(\mu, \is)$-$\ssp$ if:
      \begin{enumerate}
      \item $\LS (K) < \kappa \le \mu$.
      \item There exists $\kappa_0 < \kappa$ such that $K_{\ge \kappa_0}$ has amalgamation.
      \item $K$ is $(<\kappa)$-tame.
      \item $K$ does not have the $(<\kappa)$-order property of length $\kappa$.
      \item $K$ is $(\mu, \is)$-$\ssp$.
      \item $K_{\is} \subseteq \Ksatp{\kappa}_\mu$ and $\is = \isch{\kappa} (K)^{\le 1} \rest K_{\is}$.
      \end{enumerate}
  \end{enumerate}
\end{defin}

The terminology is justified by the next proposition which tells us that the existence of \emph{any} generator is equivalent to superstability. It makes checking that superstability holds easier and we will use it freely.

\begin{prop}\label{equiv-ss}
  Let $K$ be an AEC.

  \begin{enumerate}
    \item $K$ is $\mu$-$\ssp$ if and only if there exists $\is$ such that $K$ is $(\mu, \is)$-$\ssp$.
    \item $K$ is $\kappa$-strongly $\mu$-$\ssp$ if and only if there exists $\is$ such that $K$ is $\kappa$-strongly $(\mu, \is)$-$\ssp$.
  \end{enumerate}
\end{prop}
\begin{proof} \
  \begin{enumerate}
    \item Assume first that $K$ is $\mu$-$\ssp$. Then one can readily check (using Proposition \ref{weakly-good-examples} and Remark \ref{equiv-two-mods}) that $(K_\mu, \insp{\mu} (K)^{\le 1})$ is a generator for a good independence relation, where the local character axiom is witnessed by $\leu$. Conversely, assume that $K$ is $(\mu, \is)$-$\ssp$. By definition, $\LS (K) \le \mu$ and by definition of a generator $K_{\ge \mu}$ has amalgamation and is $\mu$-tame. By Lemma \ref{gen-canon-good}, $(K_\mu, \insp{\mu} (K)^{\le 1})$ is a $\mu$-generator for a good $(\le 1)$-independence relation, and $\leu$ is the ordering witnessing local character. Thus $K$ is $\mu$-$\ssp$.
    \item Assume first that $K$ is $\kappa$-strongly $\mu$-$\ssp$. Let $\kappa_0 < \kappa$ be such that $K_{\ge \kappa_0}$ has amalgamation. Assume without loss of generality that $\kappa_0 = \LS (K)$ and that $K_{\ge \kappa_0} = K$. By (the proof of) Proposition \ref{weakly-good-examples}, $(K_\mu, \isch{\kappa} (K)_\mu^{\le 1})$ is a $\mu$-generator for a weakly good $(\le 1)$-independence relation. By the other conditions, it is actually a $\mu$-generator for a good $(\le 1)$-independence relation. Conversely, assume that $K$ is $\kappa$-strongly $(\mu, \is)$-$\ssp$. We check the last two conditions in the definition of strong superstability, the others are straightforward. We know that $(K_\mu, \is)$ is a generator and $\is = \isch{\kappa} (K)^{\le 1} \rest K_{\is}$. Thus $K_{\is} \subseteq \Ksatp{\kappa}_\mu$ is dense in $K_\mu$, so $\Ksatp{\kappa}_\mu$ is dense in $K_\mu$. By Lemma \ref{univ-lc}, $\clc{1} (\isch{\kappa} (K)_\mu, \leu) \le \clc{1} (\is, \leg)$ for any $\leg$ such that $(K_{\is}, \leg)$ is a skeleton of $K_\mu$ (and hence of $\Ksatp{\kappa}_\mu$). By assumption one can find such a $\leg$ with $\clc{1} (\is, \leg) = \aleph_0$. Thus 
      $$
      \clc{1} (\isch{\kappa} (K)_\mu, \leu) = \aleph_0
      $$
      
  \end{enumerate}
\end{proof}
\begin{remark}\label{equiv-ss-rmk}
  Thus in Definitions \ref{ss-def} and \ref{ss-strong-def}, one can replace $\leu$ by $\lel{\mu}{\delta}$ for $1 \le \delta < \mu^+$.
\end{remark}

The next result gives evidence that Definition \ref{ss-def} is a reasonable definition of superstability, at least in tame AECs with amalgamation. Note that most of it already appears implicitly in \cite{ss-tame-jsl} and essentially restates Theorem \ref{good-up}.

\begin{thm}\label{superstab-frame}
  Assume $K$ is a $(\mu, \is)$-$\ssp$ AEC. Then:

  \begin{enumerate}
    \item\label{class-cond} $K_{\ge \mu}$ has joint embedding, no maximal models, and is stable in all $\lambda \ge \mu$.
    \item\label{frame-cond} Let $\lambda \ge \mu^+$ and let $\is' := \Kupp{(K_\mu, \is)} \rest \Ksatp{\mu^+}_{\ge \lambda}$.   \begin{enumerate}
    \item\label{frame-cond-a} $\is'$ is an almost pre-good $(\le 1, \ge \lambda)$-independence relation (recall Definition \ref{almost-good-def}).
    \item\label{frame-cond-b} If in addition $K$ is \emph{$\kappa$-strongly} $(\mu, \is)$-$\ssp$, then $\pre (\is') = \pre(\isch{\kappa} (K))^{\le 1} \rest \Ksatp{\mu^+}_{\ge \lambda}$. That is, the frame is $(<\kappa)$-coheir.
    \item\label{frame-cond-c} If $\theta \ge \mu^+$ is such that $K' := \Ksatp{\theta}_{\ge \lambda}$ is an AEC with $\LS (K') = \lambda$, then $\is' \rest K'$ is a pre-good $(\le 1, \ge \lambda)$-independence relation that will be $(<\kappa)$-coheir if $K$ is $\kappa$-strongly $(\mu, \is)$-$\ssp$.
    \end{enumerate}
  \end{enumerate}
\end{thm}
\begin{proof}
  Theorem \ref{good-up} gives (\ref{class-cond}) and (\ref{frame-cond-a}), while (\ref{frame-cond-c}) follows from (\ref{frame-cond-a}) and (\ref{frame-cond-b}). It remains to prove (\ref{frame-cond-b}). Let $\is'' := \isch{\kappa} (K))^{\le 1} \rest \Ksatp{\mu^+}_{\ge \lambda}$. By the proof of Lemma \ref{weakly-good-coheir-0}, $\nf_{\is'} \subseteq \nf_{\is''}$. Now by (\ref{frame-cond-a}), $\pre (\is')$ has existence and extension and by Fact \ref{coheir-syn}, $\is''$ has uniqueness. By \cite[Lemma 4.1]{bgkv-v3-toappear}, $\pre (\is') = \pre(\is'')$, as desired.
\end{proof}
\begin{remark}
  Let $T$ be a complete first-order theory and let $K := (\text{Mod} (T), \preceq)$. Then this paper's definitions of superstability and strong superstability coincide with the classical definition. More precisely for all $\mu \ge |T|$, $K$ is (strongly) $\mu$-superstable if and only if $T$ is stable in all $\lambda \ge \mu$.
\end{remark}

Note also that [strong] $\mu$-$\text{superstability}^+$ is monotonic in $\mu$:

\begin{prop}\label{ss-monot}
  If $K$ is [$\kappa$-strongly] $\mu$-$\ssp$ and $\mu' \ge \mu$, then $K$ is [$\kappa$-strongly] $\mu'$-$\ssp$.
\end{prop}
\begin{proof}
  Say $K$ is $(\mu, \is)$-$\ssp$. It is clearly enough to check that $K$ is $\mu'$-superstable. Let $\is' := \Kupp{(K, \is)} \rest \Ksatp{\mu^+}_{\ge \mu'}$. By Theorem \ref{superstab-frame} and Proposition \ref{weakly-good-examples}, $(K_{\mu'}, \is')$ is a generator for a good $\mu'$-independence relation, so $K$ is $(\mu', \is')$-superstable. Similarly, if $K$ is $\kappa$-strongly $(\mu, \is)$-$\ssp$ then $K$ will be $\kappa$-strongly $(\mu', \is')$-superstable.
\end{proof}

Theorem \ref{superstab-frame}.(\ref{frame-cond-b}) is the reason we introduced strong superstability. While it may seem like a detail, we are interested in extending our good frame to a frame for types longer than one element and using coheir to do so seems reasonable. Using the canonicity of coheir, we can show that superstability and strong superstability are equivalent if we do not care about the parameter $\mu$:

\begin{thm}\label{ss-strong-equiv}
  If $K$ is $\mu$-$\ssp$ and $\kappa = \beth_\kappa > \mu$, then $K$ is $\kappa$-strongly $(2^{<\kappap})^+$-$\ssp$.

  In particular a tame AEC with amalgamation is strongly superstable if and only if it is superstable.
\end{thm}
\begin{proof} 
  Let $\mu' := (2^{<\kappap})^+$. We show that $K$ is $\kappa$-strongly $\mu'$-$\ssp$. By Theorem \ref{superstab-frame}, $K_{\ge \mu}$ has joint embedding, no maximal models and is stable in all cardinals. By definition, $K_{\ge \mu}$ also has amalgamation. Also, $K$ is $\mu$-tame, hence $(<\kappa)$-tame. By Fact \ref{op-facts}, $K$ does not have the $(<\kappa)$-order property of length $\kappa$. Moreover we have already observed that $K_{\mu'}$ is stable in $\mu'$ and has joint embedding and no maximal models. Also, $\Ksatp{\kappa}_{\mu'}$ is dense in $K_{\mu'}$ by stability and the fact $\mu' > \kappa$. It remains to check that $\clc{1} (\isch{\kappa} (K)_{\mu'}, \leu) = \aleph_0$.

By Theorem \ref{superstab-frame}, there is a $(\le 1, \ge \mu^+)$-independence relation $\is'$ such that $K_{\is'} = \Ksatp{\mu^+}$ and $\is'$ is good, except that $K_{\is'}$ may not be an AEC. By Theorem \ref{canon-coheir} (with $\lambda$ there standing for $2^{<\kappap}$ here), $\pre (\is') \rest \Ksatp{\mu'} = \pre (\isch{\kappa} (K)^{\le 1}) \rest \Ksatp{\mu'}$. By the proof of Lemma \ref{cont-lem}, $\is'$ has the right $(<\kappa)$-witness property for members of $K_{\ge \mu^+}$: If $M \in K_{\ge \mu^+}$, $M_0 \lea M$ is in $\Ksatp{\mu^+}$, and $p \in \gS (M)$, then $p$ does not $\is$-fork over $M_0$ if and only if $p \rest B$ does not $\is$-fork over $M_0$ for all $B \subseteq |M|$ with $|B| < \kappa$. Therefore by the proof of Theorem \ref{canon-coheir}, we actually have that for any $M \in K_{\ge \mu'}$ and $M_0 \lea M$ in $\Ksatp{\mu'}$, $p \in \gS (M)$ does not $\is'$-fork over $M_0$ if and only if $p$ does not $\isch{\kappa} (K)$-fork over $M_0$. In particular:

$$
\clc{1} (\isch{\kappa} (K)_{\mu'}) = \clc{1} (\is_\mu') = \aleph_0
$$

Therefore $\clc{1} (\isch{\kappa} (K)_{\mu'}, \leu) = \aleph_0$, as needed.

\end{proof}

We now arrive to the main result of this section: categoricity implies strong superstability. We first recall several known consequences of categoricity.

\begin{fact}\label{categ-facts}
  Let $K$ be an AEC with no maximal models, joint embedding, and amalgamation. Assume $K$ is categorical in a $\lambda > \LS (K)$. Then:
  \begin{enumerate}
    \item\label{categ-facts-1} \cite[Claim I.1.7]{sh394} $K$ is stable in all $\mu \in [\LS (K), \lambda)$.
    \item\label{ss-high-cof} \cite[Lemma 6.3]{sh394} For $\LS (K) \le \mu < \cf{\lambda}$, $\clc{1} (\snsp{\mu} (K_\mu), \lel{\mu}{\omega}) = \aleph_0$.
    \item\label{categ-facts-bg} \cite[Theorem 6.6]{bg-v9} Assume $K$ does not have the weak $\kappa$-order property (see Definition \ref{def-weak-op}) and $\LS (K) < \kappa \le \mu < \lambda$. Then:
      $$
      \clc{1}(\isch{\kappa} (K)_\mu, \leu) = \aleph_0
      $$
    \item \label{sat-transfer} \cite[Lemma II.1.5]{sh394} If the model of size $\lambda$ is $\mu$-saturated for $\mu > \LS (K)$, then every member of $K_{\ge \chi}$ is $\mu$-saturated, where $\chi := \min (\lambda, \sup_{\mu_0 < \mu} \hanf{\mu_0})$.
  \end{enumerate}
\end{fact}

The next proposition is folklore: it derives joint embedding and no maximal models from amalgamation and categoricity. We could not find a proof in the literature, so we include one here.

\begin{prop}\label{jep-from-categ}
  Let $K$ be an AEC with amalgamation. If there exists $\lambda \ge \LS (K)$ such that $K_\lambda$ has joint embedding, then there exists $\chi < \hanf{\LS (K)}$ such that $K_{\ge \chi}$ has joint embedding and no maximal models.
\end{prop}
\begin{proof}
  Write $\mu := \hanf{\LS (K)}$. If $K_\mu = \emptyset$, then by Fact \ref{hanf-existence} there exists $\chi < \mu$ such that $K_{\ge \chi} = \emptyset$, so it has has joint embedding and no maximal models. Now assume $K_\mu \neq \emptyset$. In particular, $K$ has arbitrarily large models. By amalgamation, $K_{\ge \lambda}$ has joint embedding, and so no maximal models. If $\lambda < \mu$ we are done so assume $\lambda \ge \mu$. It is enough to show that there exists $\chi < \mu$ such that $K_{\ge \chi}$ has no maximal model since then any model of $K_{\ge \chi}$ embeds inside a model in $K_{\ge \lambda}$ and hence $K_{\ge \chi}$ has joint embedding.

  By Fact \ref{jep-partition}, we can write $K = \bigcup_{i \in I} K^i$ where the $K^i$'s are disjoint AECs with $\LS (K^i) = \LS (K)$ and each $K^i$ has joint embedding and amalgamation. Note that $|I| \le I (K, \LS (K)) \le 2^{\LS (K)}$. For $i \in I$, let $\chi_i$ be the least $\chi < \mu$ such that $K_{\ge \chi}^i = \emptyset$, or $\LS (K)$ if $K_\mu^i \neq \emptyset$. Let $\chi := \sup_{i \in I} \chi_i$. Note that $\cf{\mu} = \left(2^{\LS (K)}\right)^+ > 2^{\LS (K)} \ge |I|$, so $\chi < \mu$. 

  Now let $M \in K_{\ge \chi}$. Let $i \in I$ be such that $M \in K^i$. $M$ witnesses that $K_\chi^i \neq \emptyset$ so by definition of $\chi$, $K^i$ has arbitrarily large models. Since $K^i$ has joint embedding, this implies that $K^i$ has no maximal models. Therefore there exists $N \in K^i \subseteq K$ with $M \lta N$, as desired.
\end{proof}

The next two results are simple consequences of Fact \ref{categ-facts}.(\ref{ss-high-cof}).

\begin{prop}\label{thm-ss-high-cof}
  Let $K$ be an $\LS (K)$-tame AEC with amalgamation and no maximal models. If $K$ is categorical in a $\lambda$ with $\cf{\lambda} > \LS (K)$, then $K$ is $\LS (K)$-$\ssp$.
\end{prop}
\begin{proof}
By amalgamation, categoricity, and no maximal models, $K$ has joint embedding. By Fact \ref{categ-facts}.(\ref{categ-facts-1}), $K$ is stable in $\LS (K)$. Now apply Fact \ref{categ-facts}.(\ref{ss-high-cof}) and Proposition \ref{equiv-ss} (with Remark \ref{equiv-ss-rmk}).
\end{proof}
\begin{prop}
  Let $K$ be an $\LS (K)$-tame AEC with amalgamation. If $K$ is categorical in a $\lambda$ with $\cf{\lambda} \ge \hanf{\LS (K)}$, then there exists $\mu < \hanf{\LS (K)}$ such that $K$ is $\mu$-$\ssp$.
\end{prop}
\begin{proof}
  By Proposition \ref{jep-from-categ}, there exists $\mu < \hanf{\LS (K)}$ such that $K_{\ge \mu}$ has joint embedding and no maximal models. Now apply Proposition \ref{thm-ss-high-cof} to $K_{\ge \mu}$.
\end{proof}

We now remove the restriction on the cofinality and get strong superstability. The downside is that $\hanf{\LS (K)}$ is replaced by a fixed point of the beth function above $\LS (K)$.

\begin{thm}\label{categ-frame}
  Let $K$ be an AEC with amalgamation. Let $\kappa = \beth_{\kappa} > \LS (K)$ and assume $K$ is $(<\kappa)$-tame. If $K$ is categorical in a $\lambda > \kappa$, then:

  \begin{enumerate}
    \item \label{categ-frame-1} $K$ is $\kappa$-strongly $\kappa$-$\ssp$.
    \item \label{categ-frame-2} $K$ is stable in all cardinals above or equal to $\hanf{\LS (K)}$.
    \item \label{categ-frame-3} The model of size $\lambda$ is saturated.
    \item \label{categ-frame-5} $K$ is categorical in $\kappa$.
    \item \label{categ-frame-4} For $\chi := \min (\lambda, \hanf{\kappa})$, $\pre \left(\isch{\kappa} (K)_{\ge \chi}^{\le 1}\right)$ is a good $(\le 1, \ge \chi)$-frame with underlying AEC $K_{\ge \chi}$.
  \end{enumerate}
\end{thm}
\begin{proof}
  Note that $K_\lambda$ has joint embedding so by Proposition \ref{jep-from-categ}, there exists $\chi_0 < \hanf{\LS (K)}$ such that $K_{\ge \chi_0}$ (and thus $K_{\ge \kappa}$) has joint embedding and no maximal models. By Fact \ref{categ-facts}.(\ref{categ-facts-1}), $K_{\ge \chi_0}$ is stable everywhere below $\lambda$. Since $\kappa = \beth_{\kappa}$, Fact \ref{op-facts} implies that $K$ does not have the $(<\kappa)$-order property of length $\kappa$.

  Let $\kappa \le \mu < \lambda$. By Fact \ref{categ-facts}.(\ref{categ-facts-bg}), $\clc{1} (\isch{\kappa} (K)_{\mu}, \leu) = \aleph_0$. Now using Proposition \ref{equiv-ss}, $K$ is $\kappa$-strongly $\mu$-superstable if and only if $\Ksatp{\kappa}_{\mu}$ is dense in $K_{\mu}$. If $\kappa < \mu$, then $\Ksatp{\kappa}_{\mu}$ is dense in $K_\mu$ (by stability), so $K$ is $\kappa$-strongly $\mu$-superstable. However we want $\kappa$-strong $\kappa$-superstability. We proceed in several steps.

  First, we show $K$ is $\mu$-superstable for \emph{some} $\mu < \lambda$. If $\lambda = \kappa^+$, then this follows directly from Proposition \ref{thm-ss-high-cof} with $\mu = \kappa$, so assume $\lambda > \kappa^+$. Then by the previous paragraph $K$ is $\kappa$-strongly $\mu$-superstable for $\mu := \kappa^+$.

Second, we prove (\ref{categ-frame-2}). We have already observed $K_{\ge \chi_0}$ is stable everywhere below $\lambda$. By Theorem \ref{superstab-frame}, $K$ is stable in every $\mu' \ge \mu$. In particular, it is stable in and above $\lambda$, so (\ref{categ-frame-2}) follows.

Third, we show (\ref{categ-frame-3}). Since $K$ is stable in $\lambda$, we can build a $\lambda_0^+$-saturated model of size $\lambda$ for all $\lambda_0 < \lambda$. Thus the model of size $\lambda$ is $\lambda_0^+$-saturated for all $\lambda_0 < \lambda$, and hence $\lambda$-saturated. 

Fourth, we prove (\ref{categ-frame-5}). Since the model of size $\lambda$ is saturated, it is $\kappa$-saturated. By Fact \ref{categ-facts}.(\ref{sat-transfer}), every model of size $\sup_{\kappa_0 < \kappa} h (\kappa_0) = \kappa$ is $\kappa$-saturated. By uniqueness of saturated models, $K$ is categorical in $\kappa$.

Fifth, observe that since every model of size $\kappa$ is saturated, $\Ksatp{\kappa}_\kappa = K_\kappa$ is dense in $K_\kappa$. By the second paragraph above, $K$ is $\kappa$-strongly $\kappa$-superstable so (\ref{categ-frame-1}) holds.

Finally, we prove (\ref{categ-frame-4}). We have seen that the model of size $\lambda$ is saturated, thus $\kappa^+$-saturated. By Fact \ref{categ-facts}.(\ref{sat-transfer}), every model of size $\ge \chi$ is $\kappa^+$-saturated. Now use (\ref{categ-frame-1}) with Theorem \ref{superstab-frame}.
\end{proof}
\begin{remark}
  If one just wants to get strong superstability from categoricity, we suspect it should be possible to replace the $\beth_\kappa = \kappa$ hypothesis by something more reasonable (maybe just asking for the categoricity cardinal to be above $2^{\kappa}$). Since we are only interested in eventual behavior here, we leave this to future work.
\end{remark}

As a final remark, we point out that it is always possible to get a good independence relation from superstability (i.e.\ even without categoricity) if one is willing to restrict the class to sufficiently saturated models:

\begin{fact}[Corollary 4.5 in \cite{bv-sat-v3}]\label{bv-sat-fact}
  Let $K$ be an AEC. If $K$ is $\kappa$-strongly $\mu$-$\ssp$, then whenever $\lambda > (\mu^{<\kappap})^{+}$, $\Ksatp{\lambda}$ is an AEC with $\LS (\Ksatp{\lambda}) = \lambda$.
\end{fact}
\begin{cor}
  Let $K$ be an AEC. If $K$ is $\kappa$-strongly $\mu$-$\ssp$, then $\Ksatp{(\mu^{<\kappap})^{+2}}$ is $(\le 1)$-good. Moreover the good frame is induced by $(<\kappa)$-coheir.
\end{cor}
\begin{proof}
  Combine Theorem \ref{superstab-frame}.(\ref{frame-cond-c}) and Fact \ref{bv-sat-fact}.
\end{proof}
\begin{remark}\label{superlimit-remark}
  Let $K$ be an AEC in $\lambda := \LS (K)$ with amalgamation, joint embedding, and no maximal models. If $\Ksatp{\lambda}$ is a nonempty AEC in $\lambda$, then the saturated model is superlimit (see \cite[Definition 1.13]{shelahaecbook}). Thus we even obtain a good frame in the sense of \cite[Chapter II]{shelahaecbook}.
\end{remark}

\section{Domination}\label{domin-sec}

Our next aim is to take a sufficiently nice good $\lambda$-frame (for types of length 1) and show that it can be extended to types of any length at most $\lambda$. To do this, we will give conditions under which a good $\lambda$-frame is \emph{weakly successful} (a key technical property of \cite[Chapter II]{shelahaecbook}, see Definition \ref{weakly-succ-def}), and even $\omega$-successful (Definition \ref{omega-succ-def}).

The hypotheses we will work with are: 

\begin{hypothesis}\label{ss-hyp-2} \
  \begin{enumerate}
    \item $\is = (K, \nf)$ is a $(<\infty, \ge \mu)$-independence relation.
    \item $\s := \pre (\is^{\le 1})$ is a type-full good $(\ge \mu)$-frame.
    \item $\lambda > \mu$ is a cardinal.
    \item\label{ss-hyp-2-4} For all $n < \omega$:
      \begin{enumerate}
      \item $\Ksatp{\lambda^{+n}}$ is an AEC\footnote{Thus we have a superlimit of size $\lambda^{+n}$, see Remark \ref{superlimit-remark}.} with Löwenheim-Skolem-Tarski number $\lambda^{+n}$.
      \item $\clc{\lambda^{+n}} (\is) = \lambda^{+{n + 1}}$.
      \end{enumerate}
    \item $\is$ has base monotonicity, $\pre (\is)$ has uniqueness.
    \item $\is$ has the left and right $(\le \mu)$-model-witness properties.
  \end{enumerate}
\end{hypothesis}

\begin{remark}
  We could have given more local hypotheses (e.g.\ by replacing $\infty$ by $\theta$ or only assuming (\ref{ss-hyp-2-4}) for $n$ below some fixed $m < \omega$) and made some of the required properties more precise (this is part of what should be done to improve ``short'' to ``diagonally tame'' in the main theorem, see the discussion in Section \ref{main-thm-sec}).
\end{remark}

The key is that we assume there is \emph{already} an independence notion for longer types. However, it is potentially weak compared to what we want. The next fact shows that the hypotheses above are reasonable.

\begin{fact}\label{ss-hyp-2-prop}
  Assume $K^0$ is a fully $(<\kappa)$-tame and short $\kappa$-strongly $\mu_0$-superstable AEC with amalgamation. Then for any $\mu \ge \left(\mu_0^{<\kappap}\right)^{+2}$ and any $\lambda > \mu$ with $\lambda = \lambda^{<\kappap}$, Hypothesis \ref{ss-hyp-2} holds for $K := \left(K^0\right)^{\mu\text{-sat}}$ and $\is := \isch{\kappa} (K^0) \rest K$.
\end{fact}  
\begin{proof}
  By Fact \ref{bv-sat-fact}, for any $\mu' \ge \mu$, $\Ksatp{\mu'}$ is an AEC with $\LS (\Ksatp{\mu'}) = \mu'$. By Theorem \ref{superstab-frame}.(\ref{frame-cond-c}), $(<\kappa)$-coheir induces a good $(\ge \mu)$-frame for $\mu$-saturated models. The other conditions follow directly from the definition of strong superstability and the properties of coheir (Fact \ref{coheir-syn}). For example, the local character condition holds because $\lambda^{<\kappap} = \lambda$ implies $\left(\lambda^{+n}\right)^{<\kappap} = \lambda^{+n}$ for any $n < \omega$.
\end{proof}

The next technical property is of great importance in Chapter II and III of \cite{shelahaecbook}. The definition below follows \cite[Definition 4.1.5]{jrsh875} (but as usual, we work only with type-full frames).

\begin{defin}\label{weakly-succ-def}
  Let $\ts$ be a type-full good $\lambda_{\ts}$-frame. 

  \begin{enumerate}
  \item\index{amalgam} For $M_0 \lea M_\ell$ in $K$, $\ell = 1,2$, an \emph{amalgam of $M_1$ and $M_2$ over $M_0$} is a triple $(f_1, f_2, N)$ such that $N \in K_{\ts}$ and $f_\ell : M_\ell \xrightarrow[M_0]{} N$.
  \item\index{equivalence of amalgam} Let $(f_1^x, f_2^x, N^x)$, $x = a,b$ be amalgams of $M_1$ and $M_2$ over $M_0$. We say $(f_1^a, f_2^a, N^a)$ and $(f_1^b, f_2^b, N^b)$ are \emph{equivalent over $M_0$} if there exists $N_\ast \in K_{\ts}$ and $f^x : N^x \rightarrow N_\ast$ such that $f^b \circ f_1^b = f^a \circ f_1^a$ and $f^b \circ f_2^b = f^a \circ f_2^a$, namely, the following commutes:

  \[
  \xymatrix{ & N^b \ar@{.>}[r]^{f^b} & N_\ast \\
    M_1 \ar[ru]^{f_1^b} \ar[rr]|>>>>>{f_1^a} & & N^a \ar@{.>}[u]_{f^a} \\
    M_0 \ar[u] \ar[r] & M_2 \ar[uu]|>>>>>{f_2^b}  \ar[ur]_{f_2^a} & \\
  }
  \]

  Note that being ``equivalent over $M_0$'' is an equivalence relation (\cite[Proposition 4.3]{jrsh875}).
\item\index{uniqueness triple}\index{$\Ktuqp{\ts}$|see {uniqueness triple}} Let $\Ktuqp{\ts}$ be the set of triples $(a, M, N)$ such that $M \lea N$ are in $K$, $a \in |N| \backslash |M|$ and for any $M_1 \gea M$ in $K$, there exists a unique (up to equivalence over $M$) amalgam $(f_1, f_2, N_1)$ of $N$ and $M_1$ over $M$ such that $\gtp (f_1 (a) / f_2[M_1] ; N_1)$ does not fork over $M$. We call the elements of $\Ktuqp{\ts}$ \emph{uniqueness triples}.
\item\index{existence property (for uniqueness triples)} $\Ktuqp{\ts}$ has the \emph{existence property} if for any $M \in K_{\ts}$ and any nonalgebraic $p \in \gS (M)$, one can write $p = \gtp (a / M; N)$ with $(a, M, N) \in \Ktuqp{\ts}$. We also talk about the \emph{existence property for uniqueness triples}.
\item\index{weakly successful} $\s$ is \emph{weakly successful} if $\Ktuqp{\ts}$ has the existence property.
  \end{enumerate}
\end{defin}

The uniqueness triples can be seen as describing a version of domination. They were introduced by Shelah for the purpose of starting with a good $\lambda$-frame and extending it to a good $\lambda^+$-frame. The idea is to first extend the good $\lambda$-frame to a forking notion for types of models of size $\lambda$ (and really this is what interests us here, since tameness already gives us a good $\lambda^+$-frame). Now, since we already have an independence notion for longer types, we can follow \cite[Definition 4.21]{makkaishelah} and give a more explicit version of domination that is exactly as in the first-order case.

\begin{defin}[Domination]\index{domination}\index{dominates|see {domination}}\index{model-dominates|see domination}
  Fix $N \in K$. For $M \lea N$, $B, C \subseteq |N|$, $B$ \emph{dominates} $C$ over $M$ in $N$ if for any $N' \gea N$ and any $D \subseteq |N'|$, $\nfs{M}{B}{D}{N'}$ implies $\nfs{M}{B \cup C}{D}{N'}$. 

  We say that $B$ \emph{model-dominates} $C$ over $M$ in $N$ if for any $N' \gea N$ and any $M \lea N_0' \lea N'$, $\nfs{M}{B}{N_0'}{N'}$ implies $\nfs{M}{B \cup C}{N_0'}{N'}$.
\end{defin}

Model-domination turns out to be the technical variation we need, but of course if $\is$ has extension, then it is equivalent to domination. We start with two easy ambient monotonicity properties:

\begin{lem}\label{domination-monot}
  Let $M \lea N$. Let $B, C \subseteq |N|$ and assume $B$ [model-]dominates $C$ over $M$ in $N$. Then:

  \begin{enumerate}
    \item If $N' \gea N$, then $B$ [model-]dominates $C$ over $M$ in $N'$.
    \item If $M \lea N_0 \lea N$ contains $B \cup C$, then $B$ [model-]dominates $C$ over $M$ in $N_0$.
  \end{enumerate}
\end{lem}
\begin{proof} We only do the proofs for the non-model variation but of course the model variation is completely similar.

  \begin{enumerate}
    \item By definition of domination.
    \item Let $N' \gea N_0$ and $D \subseteq |N'|$ be given such that $\nfs{M}{B}{D}{N'}$. By amalgamation, there exists $N'' \gea N$ and $f: N' \xrightarrow[N_0]{} N''$. By invariance, $\nfs{M}{B}{f[D]}{N''}$. By definition of domination, $\nfs{M}{B \cup C}{f[D]}{f[N']}$. By invariance again, $\nfs{M}{B \cup C}{D}{N'}$, as desired.
  \end{enumerate}
\end{proof}

The next result is key for us: it ties domination with the notion of uniqueness triples:  

\begin{lem}\label{domination-uq}
  Assume $M_0 \lea M_1$ are in $K_\lambda$, and $a \in M_1$ model-dominates $M_1$ over $M_0$ (in $M_1$). Then $(a, M_0, M_1) \in \Ktuqp{\s_\lambda}$.
\end{lem}
\begin{proof}
  Let $M_2 \gea M_0$ be in $K_\lambda$. First, we need to show that there exists $(b, M_2, N)$ such that $\gtp (b / M_2; N)$ extends $\gtp (a / M_0; M_1)$ and $\gtp (b / M_2; N)$ does not fork over $M_0$. This holds by the extension property of good frames.

  Second, we need to show that any such amalgam is unique: Let $(f_1^x, f_2^x, N^x)$, $x \in \{a, b\}$ be amalgams of $M_1$ and $M_2$ over $M_0$ such that $\nfs{M_0}{f_1^x (a)}{f_2^x[M_2]}{N^x}$. We want to show that the two amalgams are equivalent: we want $N_\ast \in K_\lambda$ and $f^x : N^x \rightarrow N_\ast$ such that $f^b \circ f_1^b = f^a \circ f_1^a$ and $f^b \circ f_2^b = f^a \circ f_2^a$, namely, the following commutes:

  \[
  \xymatrix{ & N^b \ar@{.>}[r]^{f^b} & N_\ast \\
    M_1 \ar[ru]^{f_1^b} \ar[rr]|>>>>>{f_1^a} & & N^a \ar@{.>}[u]_{f^a} \\
    M_0 \ar[u] \ar[r] & M_2 \ar[uu]|>>>>>{f_2^b}  \ar[ur]_{f_2^a} & \\
  }
  \]

  For $x = a,b$, rename $f_2^x$ to the identity to get amalgams $((f_1^x)', \id_{M_2}, (N^x)')$ of $M_1$ and $M_2$ over $M_0$. For $x = a, b$, the amalgams $((f_1^x)', \id_{M_2}, (N^x)')$ and $(f_1^x, f_2^x, N^x)$ are equivalent over $M_0$, hence we can assume without loss of generality that the renaming has already been done and $f_2^x = \id_{M_2}$

  Thus we know that $\nfs{M_0}{f_1^x (a)}{M_2}{N^x}$ for $x = a,b$. By domination, $\nfs{M_0}{f_1^x[M_1]}{M_2}{N^x}$. Let $\bar{M}_1$ be an enumeration of $M_1$. Using amalgamation, we can obtain the following diagram:

  \[
  \xymatrix{ & N^a \ar@{.>}[r]_{g^a} & N'\\
    & M_1 \ar[u]^{f_1^a} \ar[r]_{f_1^b} & N^b \ar@{.>}[u]_{g^b}
  }
  \]

  This shows $\gtp (f_1^a(\bar{M}_1) / M_0; N^a) = \gtp (f_1^b(\bar{M}_1) / M_0; N^b)$. By uniqueness, $\gtp (f_1^a(\bar{M}_1) / M_2; N^a) = \gtp (f_1^b(\bar{M}_1) / M_2; N^b)$. Let $N_\ast$ and $f^x: N^x \xrightarrow[M_2]{} N_\ast$ witness the equality. Since $f_2^x = \id_{M_2}$, $f^b \circ f_2^b = f^b \rest M_2 = \id_{M_2} = f^a \circ f_2^a$. Moreover, $(f^b \circ f_1^b) (\bar{M}_1) = f^b (f_1^b (\bar{M}_1)) = f^a (f_2^a (\bar{M}_1))$ by definition, so $f^b \circ f_1^b = f^a \circ f_1^a$. This completes the proof.
\end{proof}
\begin{remark}
  The converse holds if $\is$ has left extension.
\end{remark}
\begin{remark}
  The relationship of uniqueness triples with domination is already mentioned in \cite[Proposition 4.1.7]{jrsh875}, although the definition of domination there is different.
\end{remark}

Thus to prove the existence property for uniqueness triples, it will be enough to imitate the proof of \cite[Proposition 4.22]{makkaishelah}, which gives conditions under which the hypothesis of Lemma \ref{domination-uq} holds. We first show that we can work inside a local monster model.

\begin{lem}\label{monster-domination}
  Let $M \lea N$ and $B \subseteq |N|$. Let $\sea \ge N$ be $\|N\|^+$-saturated. Then $B$ model-dominates $N$ over $M$ in $\sea$ if and only if for any $M' \lea \sea$ with $M \lea M'$, $\nfs{M}{B}{M'}{\sea}$ implies $\nfs{M}{N}{M'}{\sea}$. Moreover if $\is$ has the right $(\le \mu)$-witness property, we get an analogous result for domination instead of model-domination. 
\end{lem}
\begin{proof}
  We prove the non-trivial direction for model-domination. The proof of the moreover part for domination is similar. Assume $\sea' \ge \sea$ and $M \lea M' \lea \sea'$ is such that $\nfs{M}{B}{M'}{\sea'}$. We want to show that $\nfs{M}{N}{M'}{\sea'}$. Suppose not. Then we can use the $(\le \mu)$-model-witness property to assume without loss of generality that $\|M'\| \le \mu + \|M\|$, and so we can find $N \lea N' \lea \sea'$ containing $M'$ with $\|N'\| = \|N\|$ and $\nfs{M}{B}{M'}{N'}$, $\nnfs{M}{N}{M'}{N'}$. By homogeneity, find $f: N' \xrightarrow[N]{} \sea$. By invariance, $\nfs{M}{B}{f[M']}{f[N']}$ but $\nnfs{M}{N}{f[M']}{f[N']}$. By monotonicity, $\nfs{M}{B}{f[M']}{\sea}$ but $\nnfs{M}{N}{f[M']}{\sea}$, a contradiction.
\end{proof}

\begin{lem}[Lemma 4.20 in \cite{makkaishelah}]\label{two-seq-technical}
  Let $\seq{M_i : i < \lambda^+}$, $\seq{N_i : i < \lambda^+}$ be increasing continuous in $K_\lambda$ such that $M_i \lea N_i$ for all $i < \lambda^+$. Let $M_{\lambda^+} := \bigcup_{i < \lambda^+} M_i$, $N_{\lambda^+} := \bigcup_{i < \lambda^+} N_i$. 

  Then there exists $i < \lambda^+$ such that $\nfs{M_i}{N_i}{M_{\lambda^+}}{N_{\lambda^+}}$.
\end{lem}
\begin{proof}
  For each $i < \lambda^+$, let $j_i < \lambda^+$ be least such that $\nfs{M_{j_i}}{N_i}{M_{\lambda^+}}{N_{\lambda^+}}$ (exists since $\clc{\lambda} (\is) = \lambda^+$). Let $i^\ast$ be such that $j_i < i^\ast$ for all $i < i^\ast$ and $\cf{i^\ast} \ge \mu^+$. By definition of $j_i$ and base monotonicity we have that for all $i < i^\ast$, $\nfs{M_{i^\ast}}{N_i}{M_{\lambda^+}}{N_{\lambda^+}}$. By the left $(\le \mu)$-model-witness property, $\nfs{M_{i^\ast}}{N_{i^\ast}}{M_{\lambda^+}}{N_{\lambda^+}}$.
\end{proof}

\begin{lem}[Proposition 4.22 in \cite{makkaishelah}]\label{domination-existence}
      Let $M \in K_\lambda$ be saturated. Let $\sea \gea M$ be saturated of size $\lambda^+$. Work inside $\sea$. Write $A \nf_M B$ for $\nfs{M}{A}{B}{\sea}$.
  
  \begin{itemize}
    \item There exists a saturated $N \lea \sea$ in $K_\lambda$ such that $M \lea N$, $N$ contains $a$, and $a$ model-dominates $N$ over $M$ (in $\sea$).
    \item In fact, if $M^\ast \lea M$ is in $K_{<\lambda}$, $a \nf_{M^\ast} M$, and $r \in \gS^{\le \lambda} (M^\ast a)$, then $N$ can be chosen so that it realizes $r$.
  \end{itemize}
\end{lem}
\begin{proof}
  Since $\slc{1} (\s) = \mu^+ \le \lambda$, it suffices to prove the second part. Assume it fails.

  \underline{Claim}: For any saturated $M' \gea M$ in $K_\lambda$, if $a \nf_M M'$, then the second part fails with $M'$ replacing $M$. \\

  \underline{Proof of claim}: By transitivity, $a \nf_{M^\ast} M'$. By uniqueness of saturated models, there exists $f: M' \cong_{M^\ast} M$, which we can extend to an automorphism of $\sea$. Thus we also have $f (a) \nf_{M^\ast} M$. By uniqueness, we can assume without loss of generality that $f$ fixes $a$ as well. Since the second part above is invariant under applying $f^{-1}$, the result follows.

  We now construct increasing continuous chains $\seq{M_i : i \le \lambda^+}$, $\seq{N_i : i \le \lambda^+}$ such that for all $i < \lambda^+$:

  \begin{enumerate}
    \item $M_0 = M$.
    \item $M_i \lea N_i$.
    \item $M_i \in K_\lambda$ is saturated.
    \item $a \nf_{M_0} M_i$. 
    \item $N_i \fork_{M_i} M_{i + 1}$.
  \end{enumerate}

  This is enough: the sequences contradict Lemma \ref{two-seq-technical}. This is possible: take $M_0 = M$, and $N_0$ any saturated model of size $\lambda$ containing $M_0$ and $a$ and realizing $r$. At limits, take unions (we are using that $\Ksat$ is an AEC). Now assume everything up to $i$ has been constructed. By the claim, the second part above fails for $M_i$, so in particular $N_i$ cannot be model-dominated by $a$ over $M_i$. Thus (implicitly using Lemma \ref{monster-domination}) there exists $M_i' \gea M_i$ with $a \nf_{M_i} M_i'$ and $N_i \fork_{M_i} M_i'$. By the model-witness property, we can assume without loss of generality that $\|M_i'\| \le \lambda$, so using extension and transitivity, we can find $M_{i + 1} \in K_\lambda$ saturated containing $M_i'$ so that $a \nf_{M_i} M_{i + 1}$. By monotonicity we still have $N_i \fork_{M_i} M_{i + 1}$. Let $N_{i + 1} \in K_\lambda$ be any saturated model containing $N_i$ and $M_{i + 1}$.
\end{proof}
\begin{thm}\label{weakly-successful}
  $\s_\lambda \rest \Ksat_\lambda$ is a \emph{weakly successful} type-full good $\lambda$-frame.
\end{thm}
\begin{proof}
  Since $\s_\lambda$ is a type-full good frame, $\s_\lambda \rest \Ksat_\lambda$ also is. To show it is weakly successful, we want to prove the existence property for uniqueness triples. So let $M \in \Ksat_\lambda$ and $p \in \gS (M)$ be nonalgebraic. Say $p = \gtp (a / M; N')$. Let $\sea$ be a monster model with $N' \lea \sea$. By Lemma \ref{domination-existence}, there exists $N \lea \sea$ in $\Ksat_\lambda$ such that $M \lea N$, $a \in |N|$, and $a$ dominates $N$ over $M$ in $\sea$. By Lemma \ref{domination-monot}, $a$ dominates $N$ over $M$ in $N$. By Lemma \ref{domination-uq}, $(a, M, N) \in \Ktuqp{\s_\lambda \rest \Ksat_\lambda}$. Now, $p = \gtp (a / M; N') = \gtp (a / M; \sea) = \gtp (a / M; N)$, as desired.
\end{proof}

The term ``weakly successful'' suggests that there must exist a definition of ``successful''. This is indeed the case:

\begin{defin}[Definition 10.1.1 in \cite{jrsh875}]\index{successful}
  A type-full good $\lambda_{\ts}$-frame $\ts$ is \emph{successful} if it is weakly successful and $\lea_{\lambda_{\ts}^+}^{\text{NF}}$ \index{$\lea_{\lambda_{\ts}^+}^{\text{NF}}$} has smoothness: whenever $\seq{N_i : i \le \delta}$ is a $\lea_{\lambda_{\ts}^+}^{\text{NF}}$-increasing continuous chain of saturated models in $(\Kup_{\ts})_{\lambda_{\ts}^+}$, $N \in (\Kup_{\ts})_{\lambda_{\ts}^+}$ is saturated and $i < \delta$ implies $N_i \lea_{\lambda_{\ts}^+}^{\text{NF}} N$, then $N_\delta \lea_{\lambda_{\ts}^+}^{\text{NF}} N$.
\end{defin}

We will not define $\lea_{\lambda_{\ts}^+}^{\text{NF}}$ (the interested reader can consult e.g.\ \cite[Definition 6.14]{jrsh875}). The only fact about it we will need is:

\begin{fact}[Theorem 7.8 in \cite{jarden-tameness-apal}]\label{jarden-ltnf}
  If $\ts$ is a weakly successful type-full good $\lambda_{\ts}$-frame, $(\Kup_\ts)_{[\lambda_{\ts}, \lambda_{\ts}^+]}$ has amalgamation and is $\lambda_{\ts}$-tame, then $\lea \rest (\Kup_{\ts})_{\lambda_{\ts}^+}^{\lambda_{\ts}^+\text{-sat}} = \lea_{\lambda_{\ts}^+}^{\text{NF}}$.
\end{fact}
\begin{cor}\label{cor-successful}
  $\s_\lambda \rest \Ksat_\lambda$ is a \emph{successful} type-full good $\lambda$-frame.
\end{cor}
\begin{proof}
  By Theorem \ref{weakly-successful}, $\s_\lambda \rest \Ksat_\lambda$ is weakly successful. To show it is successful, it is enough (by Fact \ref{jarden-ltnf}), to see that $\lea$ has smoothness. But this holds since $K$ is an AEC.
\end{proof}

For a good $\lambda_{\ts}$-frame $\ts$, Shelah also defines a $\lambda_{\ts}^+$-frame $\ts^+$ (\cite[Definition III.1.7]{shelahaecbook}). \index{successor of a frame}\index{$\ts^{+}$|see {successor of a frame}}\index{$\s^{+}$|see {successor of a frame}} He then goes on to show:

\begin{fact}[Claim III.1.9 in \cite{shelahaecbook}]\label{tsplus-fact}
  If $\ts$ is a successful good $\lambda_{\ts}$-frame, then $\ts^+$ is a good\footnote{Shelah proves that $\ts^+$ is actually $\text{good}^+$. There is no reason to define what this means here.} $\lambda_{\ts}^+$-frame.
\end{fact}
\begin{remark}
  This does \emph{not} use the weak continuum hypothesis.
\end{remark}

Note that in our case, it is easy to check that:

\begin{fact}
  $(\s_\lambda)^+ = \s_{\lambda^+} \rest \Ksatp{\lambda^+}_{\lambda^+}$.
\end{fact}

\begin{defin}[Definition III.1.12 in \cite{shelahaecbook}]\label{omega-succ-def}
  Let $\ts$ be a pre-$\lambda_{\ts}$-frame.
  \begin{enumerate}
    \item\index{$n$th successor of a frame}\index{$\s^{+n}$|see {$n$th successor of a frame}}\index{$\ts^{+n}$|see {$n$th successor of a frame}} By induction on $n < \omega$, define $\ts^{+n}$ as follows:
      \begin{enumerate}
        \item $\ts^{+0} = \ts$.
        \item $\ts^{+(n + 1)} = (\ts^{+n})^+$.
      \end{enumerate}
    \item\index{$n$-successful} By induction on $n < \omega$, define ``$\ts$ is $n$-successful'' as follows: 
      \begin{enumerate}
        \item $\ts$ is $0$-successful if and only if it is a good $\lambda$-frame.
        \item $\ts$ is $(n + 1)$-successful if and only if it is a successful good $\lambda$-frame and $\ts^+$ is $n$-successful.
    \end{enumerate}
    \item\index{$\omega$-successful} $\ts$ is \emph{$\omega$-successful} if it is $n$-successful for all $n < \omega$.
  \end{enumerate}
\end{defin}

Thus by Fact \ref{tsplus-fact}, $\ts$ is 1-successful if and only if it is a successful good $\lambda_{\ts}$-frame. More generally a good $\lambda_{\ts}$-frame $\ts$ is $n$-successful if and only if $\ts^{+m}$ is a successful good $\lambda_{\ts}^{+m}$-frame for all $m < n$.

\begin{thm}\label{succ-thm}
  $\s_\lambda \rest \Ksat_\lambda$ is an \emph{$\omega$-successful} type-full good $\lambda$-frame. 
\end{thm}

\begin{proof}
  By induction on $n < \omega$, simply observing that we can replace $\lambda$ by $\lambda^{+n}$ in Corollary \ref{cor-successful}.
\end{proof}

We emphasize again that we did \emph{not} use the weak continuum hypothesis (as Shelah does in \cite[Chapter II]{shelahaecbook}). We pay for this by using tameness (in Fact \ref{ss-hyp-2-prop}). Note that all the results of \cite[Chapter III]{shelahaecbook} apply to our $\omega$-successful good frame.

Recall that part of Shelah's point is that $\omega$-successful good $\lambda$-frames extend to $(\ge \lambda)$-frames. However this is secondary for us (since tameness already implies that a frame extends to larger models, see \cite{ext-frame-jml, tame-frames-revisited-v5}). Really, we want to extend the good frame to longer \emph{types}. We show that it is possible in the next section.

\section{A fully good long frame}\label{long-frame-sec}

\begin{hypothesis}\label{more-prop-hyp} 
  $\s = (K, \nf)$ is a weakly successful type-full good $\lambda$-frame.
\end{hypothesis}

This is reasonable since the previous section showed us how to build such a frame. Our goal is to extend $\s$ to obtain a fully good $(\le \lambda, \lambda)$-independence relation. Most of the work has already been done by Shelah:

\begin{fact}[Conclusion II.6.34 in \cite{shelahaecbook}]\label{nf-existence}\index{$\NF$}
  There exists a relation $\NF \subseteq \fct{4}{K}$ satisfying:

  \begin{enumerate}
    \item\label{nf-1} $\NF (M_0, M_1, M_2, M_3)$ implies $M_0 \lea M_\ell \lea M_3$ are in $K$ for $\ell = 1,2$.
    \item\label{nf-respects} $\NF (M_0, M_1, M_2, M_3)$ and $a \in |M_1| \backslash |M_2|$ implies $\gtp (a / M_2; M_3)$ does not $\s$-fork over $M_0$.
    \item\label{nf-2} Invariance: $\NF$ is preserved under isomorphisms.
    \item\label{nf-3} Monotonicity: If $\NF (M_0, M_1, M_2, M_3)$:
      \begin{enumerate}
        \item If $M_0 \lea M_\ell' \lea M_\ell$ for $\ell = 1,2$, then $\NF (M_0, M_1', M_2', M_3')$. 
        \item If $M_3' \lea M_3$ contains $|M_1| \cup |M_2|$, then $\NF (M_0, M_1, M_2, M_3')$.
        \item If $M_3' \gea M_3$, then $\NF (M_0, M_1, M_2, M_3')$.
      \end{enumerate}
    \item Symmetry: $\NF (M_0, M_1, M_2, M_3)$ if and only if $\NF (M_0, M_2, M_1, M_3)$.
    \item Long transitivity: If $\seq{M_i : i \le \alpha}$, $\seq{N_i : i \le \alpha}$ are increasing continuous and $\NF (M_i, N_i, M_{i + 1}, N_{i + 1})$ for all $i < \alpha$, then $\NF (M_0, N_0, M_\alpha, N_\alpha)$.
    \item Independent amalgamation: If $M_0 \lea M_\ell$, $\ell = 1,2$, then for some $M_3 \in K$, $f_\ell : M_\ell \xrightarrow[M_0]{} M_3$, we have $\NF (M_0, f_1[M_1], f_2[M_2], M_3)$.
    \item\label{nf-def-uq} Uniqueness: If $\NF (M_0^\ell, M_1^\ell, M_2^\ell, M_3^\ell)$, $\ell = 1,2$, $f_i : M_i^1 \cong M_i^2$ for $i = 0, 1, 2$, and $f_0 \subseteq f_1$, $f_0 \subseteq f_2$, then $f_1 \cup f_2$ can be extended to $f_3 : M_3^1 \rightarrow M_4^2$, for some $M_4^2$ with $M_3^2 \lea M_4^2$.
  \end{enumerate}
\end{fact}

\begin{notation}\index{$\nfs{M_0}{M_1}{M_2}{M_3}$}\index{$\nfs{M_0}{\ba}{M_2}{M_3}$}
  We write $\nfs{M_0}{M_1}{M_2}{M_3}$ instead of $\NF (M_0, M_1, M_2, M_3)$. If $\ba$ is a sequence, we write $\nfs{M_0}{\ba}{M_2}{M_3}$ for $\nfs{M_0}{\text{ran} (\ba)}{M_2}{M_3}$, and similarly if sequences appear at other places.
\end{notation}
\begin{remark}
  Shelah's definition of $\NF$ (\cite[Definition II.6.12]{shelahaecbook}) is very complicated. It is somewhat simplified in \cite{jrsh875}.
\end{remark}
\begin{remark}
  Shelah calls such an $\NF$ a \emph{nonforking relation which respects $\s$}\index{nonforking relation which respects $\s$} (\cite[Definition II.6.1]{shelahaecbook}). While there are similarities with this paper's definition of a good $(\le \lambda)$-frame, note that $\NF$ is only defined for types of models while we would like to make it into a relation for arbitrary types of length at most $\lambda$.
\end{remark}

We start by showing that uniqueness is really the same as the uniqueness property stated for frames. We drop Hypothesis \ref{more-prop-hyp} for the next lemma.

\begin{lem}\label{uq-equiv}
  Let $K$ be an AEC in $\lambda$ and assume $K$ has amalgamation. The following are equivalent for a relation $\NF \subseteq \fct{4}{K}$ satisfying (\ref{nf-1}), (\ref{nf-2}), (\ref{nf-3}) of Fact \ref{nf-existence}:

  \begin{enumerate}
    \item Uniqueness in the sense of Fact \ref{nf-existence}.(\ref{nf-def-uq}).
    \item Uniqueness in the sense of frames: If $\nfs{M_0}{M}{M_1}{N}$ and $\nfs{M_0}{M'}{M_1}{N'}$ for models $M, M' \in \K$, $\ba$ and $\ba'$ are enumerations of $M$ and $M'$ respectively, $p := \gtp (\ba / M_1; N)$, $q := \gtp (\ba' / M_1; N')$, and $p \rest M_0 = q \rest M_0$, then $p = q$.
  \end{enumerate}
\end{lem}
\begin{proof} \
  \begin{itemize}
  \item \underline{(1) implies (2):} Since $p \rest M_0 = q \rest M_0$, there exists $N'' \gea N'$ and $f: N \xrightarrow[M_0]{} N''$ such that $f (\ba) = \ba'$. Therefore by invariance, $\nfs{M_0}{\ba'}{f[M_1]}{N''}$. Let $f_0 := \id_{M_0}$, $f_1 := f^{-1} \rest f[M_1]$, $f_2 := \id_{M'}$. By uniqueness, there exists $N''' \gea N''$, $g \supseteq f_1 \cup f_2$, $g: N'' \rightarrow N'''$. Consider the map $h := g \circ f : N \rightarrow N'''$. Then $g \rest M_1 = \id_{M_1}$ and $h (\ba) = g (\ba') = \ba'$, so $h$ witnesses $p = q$. \\
  \item \underline{(2) implies (1):} By some renaming, it is enough to prove that whenever $\nfs{M_0}{M_2}{M_1}{N}$ and $\nfs{M_0}{M_2}{M_1}{N'}$, there exists $N'' \gea N'$ and $f: N' \xrightarrow[|M_1| \cup |M_2|]{} N''$. Let $\ba$ be an enumeration of $M_2$. Let $p := \gtp (\ba / M_1; N)$, $q := \gtp (\ba / M_1; N')$. We have that $p \rest M_0 = \gtp (\ba / M_1; M_2) = q \rest M_0$. Thus $p = q$, so there exists $N'' \gea N'$ and $f: N \xrightarrow[M_1]{} N''$ such that $f (\ba) = \ba$. In other words, $f$ fixes $M_2$, so is the desired map.
  \end{itemize}
\end{proof}

We now extend $\NF$ to take sets on the left hand side. This step is already made by Shelah in \cite[Claim III.9.6]{shelahaecbook}, for singletons rather than arbitrary sets. We check that Shelah's proofs still work.

\begin{defin}\index{$\NF'$}
  Define $\NF' (M_0, A, M, N)$ to hold if and only if $M_0 \lea M \lea N$ are in $K$, $A \subseteq |N|$, and there exists $N' \gea N$, $N_A \gea M$ with $N_A \lea N'$ and $\nfs{M_0}{N_A}{M}{N'}$. We abuse notation and also write $\nfs{M_0}{A}{M}{N}$ instead of $\NF' (M_0, A, M, N)$. We let $\ts := (K, \nf)$.
\end{defin}
\begin{remark}
  Compare with the definition of $\cl$ (Definition \ref{cl-def}).
\end{remark}

\begin{prop}\label{ts-basic-props} \
  \begin{enumerate}
  \item If $M_0 \lea M_\ell \lea M_3$, $\ell = 1,2$, then $\NF (M_0, M_1, M_2, M_3)$ if and only if $\NF' (M_0, M_1, M_2, M_3)$.
  \item $\ts$ is a (type-full) pre-$(\le \lambda, \lambda)$-frame.
  \item $\ts$ has base monotonicity, full symmetry, uniqueness, existence, and extension.
  \end{enumerate}
\end{prop}
\begin{proof} 
  Exactly as in \cite[Claim III.9.6]{shelahaecbook}. Shelah omits the proof of uniqueness, so we give it here. For notational simplicity, let us work in a local monster model $\sea \in \Ksatp{\lambda^+}_{\lambda^+}$, and write $A \nf_{M_0} M_1$ instead of $A \nf_{M_0}^{\sea} M_1$. Let $\alpha \le \lambda$ and assume that $p, q \in \gS^{\alpha} (M)$ are given such that $p = \gtp (\ba / M)$, $q = \gtp (\ba' / M)$. Assume that $M_0 \lea M$ is such that both $p$ and $q$ do not fork over $M_0$ (in the sense of NF'). We want to see that $p = q$.

  By definition, there exists $M_{\ba} \in \K_{\lambda}$ such that $M_0 \lea M_{\ba}$, $\ba \in \fct{\alpha}{|M_{\ba}|}$, and $M_{\ba} \nf_{M_0} M$. By symmetry for NF, $M \nf_{M_0} M_{\ba}$. Similarly, there exists a model $M_{\ba'} \in \K_{\lambda}$ containing $\ba'$ such that $M_0 \lea M_{\ba'}$ and $M \nf_{M_0} M_{\ba'}$.

  Since $p \rest M_0 = q \rest M_0$, there exists an automorphism $f$ of $\sea$ fixing $M_0$ such that $f (\ba) = \ba'$. By invariance, $M \nf_{M_0} M_{\ba'}$ and $f[M] \nf_{M_0} f[M_{\ba}]$, and both $M_{\ba'}$ and $f[M_{\ba}]$ contain $\ba'$. By Lemma \ref{uq-equiv} and the proof of \cite[Lemma 5.4.(3)]{bgkv-v3-toappear}, we have that (for some enumeration $\bc$ of $M$) $\gtp(\bc / M_0 \ba') = \gtp(f (\bc) / M_0 \ba')$. Thus we can pick an automorphism $g$ of $\sea$ fixing $M_0 \ba'$ and sending $f(\bc)$ back to $\bc$. Now $f \circ g^{-1}$ shows that $\gtp(\ba / M) = \gtp(\ba' / M)$, i.e.\ $p = q$ as needed.
\end{proof}

We now turn to local character. The key is:

\begin{fact}[Claim III.1.17 in \cite{shelahaecbook}]\label{limit-chain-fact}
  Let $\delta \le \lambda^+$ be a limit ordinal. Given $\seq{M_i : i \le \delta}$ increasing continuous, we can build $\seq{N_i : i \le \delta}$ increasing continuous such that for all $i \le j \le \delta$ with $j < \lambda^+$, $\nfs{M_i}{N_i}{M_{j}}{N_{j}}$ and $M_\delta \ltu N_\delta$.
\end{fact}
\begin{lem}\label{nf1-lc}
  For all $\alpha < \lambda$, $\clc{\alpha} (\ts) = |\alpha|^+ + \aleph_0$. Moreover if $\seq{M_i : i < \lambda^+}$ is increasing in $\K_{\lambda}$ and $p \in \gS^\lambda (\bigcup_{i < \lambda^+} M_i)$, there exists $i < \lambda^+$ such that $p \rest M_j$ does not fork over $M_i$ for all $j \ge i$.
\end{lem}
\begin{proof}
  Let $\alpha < \lambda$. Let $\seq{M_i : i \le \delta + 1}$ be increasing continuous with $\delta = \cf{\delta} > |\alpha|$. Let $A \subseteq |M_{\delta + 1}|$ have size $\le \alpha$. Let $\seq{N_i : i \le \delta}$ be as given by Fact \ref{limit-chain-fact}. By universality, we can assume without loss of generality that $M_{\delta + 1} \lea N_\delta$. Thus $A \subseteq |N_\delta|$ and by the cofinality hypothesis, there exists $i < \delta$ such that $A \subseteq |N_i|$. In particular, $\nfs{M_i}{A}{M_\delta}{N_\delta}$, so $\nfs{M_i}{A}{M_\delta}{M_{\delta + 1}}$, as needed. The proof of the moreover part is completely similar.
\end{proof}

\begin{remark}
  In \cite{jrsh875} (and later in \cite{jasi, jarden-prime, jarden-tameness-apal}), the authors have considered \emph{semi-good}\index{semi-good (frame)} $\lambda$-frames, where the stability condition is replaced by almost stability\index{almost stability} ($|\gS (M)| \le \lambda^+$ for all $M \in K_\lambda$), and an hypothesis called the conjugation property\index{conjugation property} is often added. Several of the above results carry through in that setup but we do not know if Lemma \ref{nf1-lc} would also hold. 
\end{remark}

We come to the last property: disjointness. The situation is a bit murky: At first glance, Fact \ref{nf-existence}.(\ref{nf-respects}) seems to give it to us for free (since we are assuming $\s$ has disjointness), but unfortunately we are assuming $a \notin |M_2|$ there. We will obtain it with the additional hypothesis of categoricity in $\lambda$ (this is reasonable since if the frame has a superlimit, see Remark \ref{superlimit-remark}, one can always restrict oneself to the class generated by the superlimit). Note that disjointness is never used in a crucial way in this paper (but it is always nice to have, as it implies for example disjoint amalgamation when combined with independent amalgamation).

\begin{lem}\label{disj-lem}
  If $K$ is categorical in $\lambda$, then $\ts$ has disjointness and $\ts^{\le 1} = \s$.
\end{lem}
\begin{proof}
  We have shown that $\ts^{\le 1}$ has all the properties of a good frame except perhaps disjointness so by the proof of Theorem \ref{good-frame-cor} (which never relied on disjointness), $\s = \ts^{\le 1}$. Since $\s$ has disjointness, $\ts^{\le 1}$ also does, and therefore $\ts$ has disjointness.
\end{proof}

What about continuity for chains? The long transitivity property seems to suggest we can say something, and indeed we can:

\begin{fact}\label{continuity-models}
  Assume $\lambda = \lambda_0^{+3}$ and there exists an $\omega$-successful good $\lambda_0$-frame $\s'$ such that $\s = (\s')^{+3}$.
  
  Assume $\delta$ is a limit ordinal and $\seq{M_i^\ell : i \le \delta}$ is increasing continuous in $K_\lambda$, $\ell \le 3$. If $\nfs{M_i^0}{M_i^1}{M_i^2}{M_i^3}$ for each $i < \delta$, then $\nfs{M_\delta^0}{M_\delta^1}{M_\delta^2}{M_\delta^3}$.
\end{fact}
\begin{proof}
By \cite[Claim III.12.2]{shelahaecbook}, all the hypotheses at the beginning of each section of Chapter III in the book hold for $\s$. Now apply Claim III.8.19 in the book.
\end{proof}
\begin{remark}
  Where does the hypothesis $\lambda = \lambda_0^{+3}$ come from? Shelah's analysis in chapter III of his book proceeds on the following lines: starting with an $\omega$-successful frames $\s$, we want to show $\s$ has nice properties like existence of prime triples, weak orthogonality being orthogonality, etc. They are hard to show in general, however it turns out $\s^{+}$ has some nicer properties than $\s$ (for example, $K_{\s^+}$ is always categorical)... In general, $\s^{+(n + 1)}$ has even nicer properties than $\s^{+n}$; and Shelah shows that the frame has all the nice properties he wants after going up three successors. 
\end{remark}

We obtain:

\begin{thm}\label{good-long-frame} \
  \begin{enumerate}
    \item If $K$ is categorical in $\lambda$, then $\ts$ is a good $(\le \lambda, \lambda)$-frame.
    \item If $\lambda = \lambda_0^{+3}$ and there exists an $\omega$-successful good $\lambda_0$-frame $\s'$ such that $\s = (\s')^{+3}$, then $\ts$ is a fully good $(\le \lambda, \lambda)$-frame.
  \end{enumerate}
\end{thm}
\begin{proof}
  $\ts$ is good by Proposition \ref{ts-basic-props}, Lemma \ref{nf1-lc}, and Lemma \ref{disj-lem}. The second part follows from Fact \ref{continuity-models} (note that by definition of the successor frame, $K$ will be categorical in $\lambda$ in that case).
\end{proof}

\begin{remark}
  In \cite[Corollary 6.10]{tame-frames-revisited-v5}, it is shown that $\lambda$-tameness and amalgamation imply that a good $\lambda$-frame extends to a good $(<\infty, \lambda)$-frame. However, the definition of a good frame there is not the same as it does \emph{not} assume that the frame is type-full (the types on which forking is defined are only the types of independent sequences). Thus the conclusion of Theorem \ref{good-long-frame} is much stronger (but uses more hypotheses).
\end{remark}

\section{Extending the base and right hand side}\label{up-transfer-sec}

\begin{hypothesis} \
  \begin{enumerate}
    \item $\is = (K, \nf)$ is a fully good $(\le \lambda, \lambda)$-independence relation.
    \item $K' := \Kup$ has amalgamation and is $\lambda$-tame for types of length less than $\lambda^+$.
  \end{enumerate}
\end{hypothesis}

In this section, we give conditions under which $\is$ becomes a fully good $(\le \lambda, \ge \lambda)$-independence relation. In the next section, we will make the left hand side bigger and get a fully good $(<\infty, \ge \lambda)$-independence relation. 

Recall that extending a $(\le 1, \lambda)$-frame to bigger models was investigated in \cite[Chapter II]{shelahaecbook} and \cite{ext-frame-jml, tame-frames-revisited-v5}. Here, most of the arguments are similar but the longer types cause some additional difficulties (e.g.\ in the proof of local character). 

\begin{notation}\index{$\is'$}\index{$\s'$}\index{$K'$}
  Let $\is' := \Kupp{\is}$ (recall Definition \ref{is-up-def}). Write $\s := \pre (\is)$, $\s' := \pre (\is')$, $K' := K_{\is'}$. We abuse notation and also denote $\nf_{\is'}$ by $\nf$.
\end{notation}

We want to investigate when the properties of $\is$ carry over to $\is'$.

\begin{lem}\label{easy-transfer} \
  \begin{enumerate}
    \item $\is'$ is a $(\le \lambda, \ge \lambda)$-independence relation. 
    \item $K'$ has joint embedding, no maximal models, and is stable in all cardinals.
    \item $\is'$ has base monotonicity, transitivity, uniqueness, and disjointness.
    \item $\is'$ has full model continuity.
  \end{enumerate}
\end{lem}
\begin{proof} \
  \begin{enumerate}
    \item By Proposition \ref{up-indep-rel}. 
    \item By \cite[Corollary 6.9]{tame-frames-revisited-v5}, $(\s')^{\le 1}$ is a good $(\ge \lambda)$-frame, so in particular $K'$ has joint embedding, no maximal models, and is stable in all cardinals.
    \item See \cite[Claim II.2.11]{shelahaecbook} for base monotonicity and transitivity. Disjointness is straightforward from the definition of $\is'$, and uniqueness follows from the tameness hypothesis and the definition of $\is'$.
    \item Assume $\seq{M_i^\ell : i \le \delta}$ is increasing continuous in $K'$, $\ell \le 3$, $\delta$ is regular, $M_i^0 \lea M_i^\ell \lea M_i^3$ for $\ell = 1,2$, $\|M_\delta^1\| < \lambda^+$ (recall the definition of full model continuity), $i < \delta$, and $\nfs{M_i^0}{M_i^1}{M_i^2}{M_i^3}$ for all $i < \delta$. Let $N := M_\delta^3$. By ambient monotonicity, $\nfs{M_i^0}{M_i^1}{M_i^2}{N}$ for all $i < \delta$. We want to see that $\nfs{M_\delta^0}{M_\delta^1}{M_\delta^2}{N}$. Since $\|M_\delta^1\| < \lambda^+$, $M_\delta^1$ and $M_\delta^0$ are in $K$. Thus it is enough to show that for all $M' \lea M_\delta^2$ in $K$ with $M_\delta^0 \lea M'$, $\nfs{M_\delta^0}{M_\delta^1}{M'}{N}$. Fix such an $M'$. We consider two cases:
      \begin{itemize}
        \item \underline{Case 1: $\delta < \lambda^+$}: Then we can find $\seq{M_i' :i \le \delta}$ increasing continuous in $K$ (as opposed to just in $K'$) such that $M_\delta' = M'$ and for all $i < \delta$, $M_i^0 \lea M_i' \lea M_i^2$. By monotonicity, for all $i < \delta$, $\nfs{M_i^0}{M_i^1}{M_i'}{N}$. By full model continuity in $K$, $\nfs{M_\delta^0}{M_\delta^1}{M'}{N}$, as desired.
        \item \underline{Case 2: $\delta \ge \lambda^+$}: Since $M_\delta^0, M_\delta^1 \in K$, the chains $\seq{M_i^\ell : i \le \delta}$ for $\ell = 0,1$ must be eventually constant, so we can assume without loss of generality that $M_\delta^0 = M_0^0$, $M_\delta^1 = M_0^1$. Since $\delta$ is regular, there exists $i < \delta$ such that $M' \lea M_i^2$. By assumption, $\nfs{M_0^0}{M_0^1}{M_i^2}{N}$, so by monotonicity, $\nfs{M_0^0}{M_0^1}{M'}{N}$, as needed.
      \end{itemize}
  \end{enumerate}
\end{proof}

We now turn to local character.

\begin{lem}\label{lc-lem-2}
  Assume $\seq{M_i :i  \le \delta}$ is increasing continuous, $p \in \gS^{\alpha} (M_\delta)$, $\alpha < \lambda^+$ a cardinal and $\delta = \cf{\delta} > \alpha$.

  \begin{enumerate}
  \item\label{lc-2} If $\alpha < \lambda$, then there exists $i < \delta$ such that $p$ does not fork over $M_i$.
  \item\label{lc-3} If $\alpha = \lambda$ and $\is$ has the left $(<\cf{\lambda})$-witness property, then there exists $i < \delta$ such that $p$ does not fork over $M_i$.
  \end{enumerate}
\end{lem}
\begin{proof} \
  \begin{enumerate}
    \item As in the proof of Lemma \ref{up-indep-prop}.(\ref{77-b}) Note that weak chain local character holds for free because $\alpha < \lambda$ and $\clc{\alpha} (\is) = \alpha^+ + \aleph_0$ by assumption. 
    \item By the proof of Lemma \ref{up-indep-prop}.(\ref{77-b}) again, it is enough to see that $\is$ has weak chain local character: Let $\seq{M_i : i < \lambda^+}$ be increasing in $K$ and let $M_{\lambda^+} := \bigcup_{i < \lambda^+} M_i$. Let $p \in \gS^\lambda (M_{\lambda^+})$. We will show that there exists $i < \lambda^+$ such that $p$ does not fork over $M_i$. Say $p = \gtp (\ba / M_{\lambda^+}; N)$ and let $A := \text{ran} (\ba)$. Write $A = \bigcup_{j < \cf{\lambda}} A_j$ with $\seq{A_j : i < \cf{\lambda}}$ increasing continuous and $|A_j| < \lambda$. By the first part, for each $j < \cf{\lambda}$ there exists $i_j < \lambda^+$ such that $\nfs{M_{i_j}}{A_j}{M_{\lambda^+}}{N}$. Let $i := \sup_{j < \cf{\lambda}} i_j$. We claim that $\nfs{M_i}{A}{M_{\lambda^+}}{N}$. By the $(<\cf{\lambda})$-witness property and the definition of $\is'$ (here we use that $M_i \in K$), it is enough to show this for all $B \subseteq A$ of size less than $\cf{\lambda}$. But any such $B$ is contained in an $A_j$, and so the result follows from base monotonicity.
  \end{enumerate}
\end{proof}

\begin{lem}\label{indep-amalgam-transfer}
 Assume $\is'$ has existence. Then $\is'$ has independent amalgamation.
\end{lem}
\begin{proof}
  As in, for example, the proof of \cite[Theorem 5.3]{ext-frame-jml}, using full model continuity.
\end{proof}

Putting everything together, we obtain:

\begin{thm}\label{main-thm-up}
  If $K$ is $(<\cf{\lambda})$-tame and short for types of length less than $\lambda^+$, then $\is'$ is a fully pre-good $(\le \lambda, \ge \lambda)$-independence relation.
\end{thm}
\begin{proof}
  We want to show that $\s'$ is fully good. The basic properties are proven in Lemma \ref{easy-transfer}. By Lemma \ref{cont-lem}, $\is$ has the left $(<\cf{\lambda})$-witness property. Thus by Lemma \ref{lc-lem-2}, for any $\alpha < \lambda^+$, $\clc{\alpha} (\is') = |\alpha|^+ + \aleph_0$. In particular, $\is'$ has existence, and thus by the definition of $\is'$ and transitivity in $\is$, $\slc{\alpha} (\is') = \lambda^+ = |\alpha|^+ + \lambda^+$. Finally by Lemma \ref{indep-amalgam-transfer}, $\is'$ has independent amalgamation and so by Proposition \ref{indep-props}.(\ref{indep-props-7}), $\is'$ has extension.
\end{proof}

\section{Extending the left hand side}\label{long-transfer-sec}

We now enlarge the left hand side of the independence relation built in the previous section. 

\begin{hypothesis}\label{long-transfer-hyp} \
  \begin{enumerate}
    \item $\is = (K, \nf)$ is a fully good $(\le \lambda, \ge \lambda)$-independence relation.
    \item $K$ is fully $\lambda$-tame and short.
  \end{enumerate}
\end{hypothesis}

\begin{defin}\index{elongation of an independence relation}\index{$\islong$|see {elongation of an independence relation}}
  Define $\islong = (K, \islongp{\nf})$ by setting $\islongp{\nf} (M_0, A, B, N)$ if and only if for all $A_0 \subseteq A$ of size less than $\lambda^+$, $\nfs{M_0}{A_0}{B}{N}$.
\end{defin}

\begin{remark}
The idea is the same as for \cite[Definition 4.3]{tame-frames-revisited-v5}: we extend the frame to have longer types. The difference is that $\islong$ is type-full.
\end{remark}
\begin{remark}
  We could also have defined extension to types of length less than $\theta$ for $\theta$ a cardinal or $\infty$ but this complicates the notation and we have no use for it here.
\end{remark}

\begin{notation}\index{$\is'$}
  Write $\is' := \islong$. We abuse notation and also write $\nf$ for $\islongp{\nf}$.
\end{notation}

\begin{lem}\label{basic-props} \
  \begin{enumerate}
    \item $\is'$ is a $(<\infty, \ge \lambda)$-independence relation.
    \item $K$ has joint embedding, no maximal models, and is stable in all cardinals.
    \item $\is'$ has base monotonicity, transitivity, disjointness, existence, symmetry, the left and right $\lambda$-witness properties, and uniqueness.
  \end{enumerate}
\end{lem}
\begin{proof} \
  \begin{enumerate}
    \item Straightforward.
    \item Because $\is$ is good.
    \item Base monotonicity, transitivity, disjointness, existence, and the left and right $\lambda$-witness property are straightforward (recall that $\is$ has the right $\lambda$-witness property by Lemma \ref{cont-lem}). Uniqueness is by the shortness hypothesis. Symmetry follows easily from the witness properties.
  \end{enumerate}
\end{proof}
\begin{lem}\label{full-mod-cont}
  Assume there exists a regular $\kappa \le \lambda$ such that $\is$ has the left $(<\kappa)$-model-witness property. Then $\is'$ has full model continuity.
\end{lem}
\begin{proof}
Let $\seq{M_i^\ell : i \le \delta}$, $\ell \le 3$ be increasing continuous in $K$ such that $M_i^0 \lea M_i^\ell \lea M_i^3$, $\ell = 1,2$, and $\nfs{M_i^0}{M_i^1}{M_i^2}{M_i^3}$. Without loss of generality, $\delta$ is regular. Let $N := M_\delta^3$. We want to show that $\nfs{M_\delta^0}{M_\delta^1}{M_\delta^2}{N}$. Let $A \subseteq |M_\delta^1|$ have size less than $\lambda^+$. Write $\mu := |A|$. By monotonicity, assume without loss of generality that $\lambda + \kappa \le \mu$. We show that $\nfs{M_\delta^0}{A}{M_\delta^2}{N}$, which is enough by definition of $\is'$. We consider two cases.

      \begin{itemize}

        \item \underline{Case 1: $\delta > \mu$}: By local character in $\is$ there exists $i < \delta$ such that $\nfs{M_i^2}{A}{M_\delta^2}{N}$. By right transitivity, $\nfs{M_i^0}{A}{M_\delta^2}{N}$, so by base monotonicity, $\nfs{M_\delta^0}{A}{M_\delta^2}{N}$. 
        \item \underline{Case 2: $\delta \le \mu$}: For $i \le \delta$, let $A_i := A \cap |M_i^1|$. Build $\seq{N_i : i \le \delta}$, $\seq{N_i^0 : i \le \delta}$ increasing continuous in $K_{\le \mu}$ such that for all $i < \delta$:
          \begin{enumerate}
            \item $A_i \subseteq |N_i|$.
            \item $N_i \lea M_i^1$, $A \subseteq |N_i|$.
            \item $N_i^0 \lea M_i^0$, $N_i^0 \lea N_i$.
            \item $\nfs{N_i^0}{N_i}{M_i^2}{N}$.
          \end{enumerate}

          \paragraph{\underline{This is possible}} 

          Fix $i \le \delta$ and assume $N_j, N_j^0$ have already been constructed for $j < i$. If $i$ is limit, take unions. Otherwise, recall that we are assuming $\nfs{M_i^0}{M_i^1}{M_i^2}{N}$. By Lemma \ref{lc-monot} (with $A_i \cup \bigcup_{j < i} |N_j|$ standing for $A$ there, this is where we use the $(<\kappa)$-model-witness property), we can find $N_i^0 \lea M_i^0$ and $N_i \lea M_i^1$ in $K_{\le \mu}$ such that $N_i^0 \lea N_i$, $\nfs{N_i^0}{N_i}{M_i^2}{N}$, $A_i \subseteq |N_i|$, $N_j \lea N_i$ for all $j < i$, and $N_j^0 \lea N_i^0$ for all $j < i$. Thus they are as desired.

          \paragraph{\underline{This is enough}}

          Note that $A_\delta = A$, so $A \subseteq |N_\delta|$. By full model continuity in $\is$, $\nfs{N_\delta^0}{N_\delta}{M_\delta^2}{N}$. By monotonicity, $\nfs{M_\delta^0}{A}{M_\delta^2}{N}$, as desired.
      \end{itemize}
\end{proof}
\begin{lem}\label{lc-long}
  Assume there exists a regular $\kappa \le \lambda$ such that $\is$ has the left $(<\kappa)$-model-witness property. Then for all cardinals $\mu$:
  \begin{enumerate}
    \item\label{lc-long-1} $\slc{\mu} (\is') = \lambda^+ + \mu^+$.
    \item\label{lc-long-2} $\clc{\mu} (\is') = \aleph_0 + \mu^+$.
  \end{enumerate}
\end{lem}
\begin{proof}
  By Lemma \ref{full-mod-cont}, $\is'$ has full model continuity. By Lemma \ref{lc-cont}, (\ref{lc-long-1}) holds. For (\ref{lc-long-2}), if $\mu \le \lambda$, this holds because $\is$ is good and if $\mu > \lambda$, this follows from Proposition \ref{indep-props}.(\ref{indep-props-lc}) and (\ref{lc-long-1}).
\end{proof}

We now turn to proving extension. The proof is significantly more complicated than in the previous section. We attempt to explain why and how our proof goes. Of course, it suffices to show independent amalgamation (Proposition \ref{indep-props}.(\ref{indep-props-7})). We work by induction on the size of the models but land in trouble when all models have the same size. Suppose for example that we want to amalgamate $M^0 \lea M^\ell$, $\ell = 1,2$ that are all in $K_{\lambda^+}$. If $M^1$ (or, by symmetry, $M^2$) had smaller size, we could use local character to assume without loss of generality that $M^0$ is in $K_\lambda$ and then imitate the usual directed system argument (as in for example the proof of \cite[Theorem 5.3]{ext-frame-jml}).

Here however it seems we have to take at least two resolutions at once so we fix $\seq{M_i^\ell : i < \lambda^+}$, $\ell = 0, 1$, satisfying the usual conditions. Letting $p := \gtp (M^1 / M^0; M^1)$ and its resolution $p_i := \gtp (M_i^1 / M_i^0; M^1)$, it is natural to build $\seq{q_i : i < \lambda^+}$ such that $q_i$ is the nonforking extension of $p_i$ to $M^2$. If everything works, we can take the direct limit of the $q_i$s and get the desired nonforking extension of $p$. However with what we have said so far it is not clear that $q_{i + 1}$ is even an extension of $q_i$! In the usual argument, this is the case since both $p_i$ and $p_{i + 1}$ do not fork over the same domain but we cannot expect it here. Thus we require in addition that $\nfs{M_i^0}{M_i^1}{M^0}{M^1}$ and this turns out to be enough for successor steps. To achieve this extra requirement, we use Lemma \ref{lc-monot}. Unfortunately, we also do not know how to go through limit steps without making one extra locality hypothesis:

\begin{defin}[Type-locality]\label{type-loc-def}\index{type-locality}\index{type-local|see {type-locality}}\index{densely type-local|see {type-locality}} \
  \begin{enumerate}
    \item Let $\delta$ be a limit ordinal, and let \index{$\bar{p}$} $\bar{p} := \seq{p_i : i < \delta}$ be an increasing chain of Galois types, where for $i < \delta$, $p_i \in \gS^{\alpha_i} (M)$ and $\seq{\alpha_i : i \le \delta}$ are increasing continuous. We say $\bar{p}$ is \emph{type-local} if whenever $p, q \in \gS^{\alpha_\delta} (M)$ are such that $p^{\alpha_i} = q^{\alpha_i} = p_i$ for all $i < \delta$, then $p = q$.
    \item We say $K$ is \emph{type-local} if every $\bar{p}$ as above is type-local.
    \item We say $K$ is \emph{densely type-local above $\lambda$} if for every $\lambda_0 > \lambda$, $M \in K_{\lambda_0}$, $p \in \gS^{\lambda_0} (M)$, there exists $\seq{N_i : i \le \delta}$ such that:

      \begin{enumerate}
        \item $\delta = \cf{\lambda_0}$.
        \item For all $i < \delta$, $N_i \in K_{<\lambda_0}$.
        \item $\seq{N_i : i \le \delta}$ is increasing continuous.
        \item $N_\delta \gea M$ is in $K_{\lambda_0}$.
        \item Letting $q_i := \gtp (N_i / M; N_\delta)$ (seen as a member of $\gS^{\alpha_i} (M)$, where of course $\seq{\alpha_i : i \le \delta}$ are increasing continuous), we have that $q_\delta$ extends $p$ and $\seq{q_j : j < i}$ is type-local for all limit $i \le \delta$.
      \end{enumerate}
      
      We say $K$ is \emph{densely type-local} if it is densely type-local above $\lambda$ for some $\lambda$.
  \end{enumerate}
\end{defin}

Intuitively, the relationship between type-locality and locality (see \cite[Definition 11.4]{baldwinbook09}) is the same as the relationship between type-shortness and tameness (in the later, we look at \emph{domain} of types, in the former we look at \emph{length} of types). We suspect that dense type-locality should hold in our context, see the discussion in Section \ref{main-thm-sec} for more. The following lemma says that increasing the elements in the resolution of the type preserves type-locality.

\begin{lem}\label{type-loc-monot}
  Let $\delta$ be a limit ordinal. Assume $\bar{p} := \seq{p_i : i < \delta}$ is an increasing chain of Galois types, $p_i \in \gS^{\alpha_i} (M)$ and $\seq{\alpha_i : i \le \delta}$ are increasing continuous. Assume $\bar{p}$ is type-local and assume $p_\delta \in \gS^{\alpha_\delta} (M)$ is such that $p^{\alpha_i} = p_i$ for all $i < \delta$. Say $p = \gtp (\ba_\delta / M; N)$ and let $\ba_i := \ba_\delta \rest \alpha_i$ (so $p_i = \gtp (\ba_i / M; N)$). 

  Assume $\seq{\bb_i : i \le \delta}$ are increasing continuous sequences such that $\ba_\delta = \bb_\delta$ and $\ba_i$ is an initial segment of $\bb_i$ for all $i < \delta$. Let $q_i := \gtp (\bb_i / M; N)$. Then $\bar{q} := \seq{q_i : i < \delta}$ is type-local.
\end{lem}
\begin{proof}
  Say $\bb_i$ is of type $\beta_i$. So $\seq{\beta_i : i \le \delta}$ is increasing continuous and $\alpha_\delta = \beta_\delta$.

  If $q \in \gS^{\beta_\delta} (M)$ is such that $q^{\beta_i} = q_i$ for all $i < \delta$, then $q^{\alpha_i} = (q_i)^{\alpha_i} = p_i$ for all $i < \delta$ so by type-locality of $\bar{p}$, $p = q$, as desired.
\end{proof}

Before proving Lemma \ref{ext-all}, let us make precise what was meant above by ``direct limit'' of a chain of types. It is known that (under some set-theoretic hypotheses) there exists AECs where some chains of Galois types do not have an upper bound, see \cite[Theorem 3.3]{non-locality}. However a \emph{coherent} chain of types (see below) always has an upper bound. We adapt Definition 5.1 in \cite{ext-frame-jml} (which is implicit already in \cite[Claim 0.32.2]{sh576} or \cite[Lemma 2.12]{tamenessthree}) to our purpose.

\begin{defin}\index{coherent chain of types}
  Let $\delta$ be an ordinal. An increasing chain of types $\seq{p_i : i < \delta}$ is said to be \emph{coherent} if there exists a sequence $\seq{(\ba_i, M_i, N_i) : i < \delta}$ and maps $f_{i, j}: N_i \rightarrow N_j$, $i \le j < \delta$, so that for all $i \le j \le k < \delta$:

  \begin{enumerate}
    \item $f_{j, k} \circ f_{i, j} = f_{i,k}$.
    \item $\gtp (\ba_i / M_i ; N_i) = p_i$.
    \item $\seq{M_i : i < \delta}$ and $\seq{N_i : i < \delta}$ are increasing.
    \item $M_i \lea N_i$, $\ba_i \in \fct{<\infty}{N_i}$.
    \item $f_{i, j}$ fixes $M_i$.
    \item $f_{i, j} (\ba_i)$ is an initial segment of $\ba_j$.
  \end{enumerate}

  We call the sequence and maps above a \emph{witnessing sequence} to the coherence of the $p_i$'s.
\end{defin}

Given a witnessing sequence $\seq{(\ba_i, M_i, N_i) :i  < \delta}$ with maps $f_{i, j} : N_i \rightarrow N_j$, we can let $N_\delta$ be the direct limit of the system $\seq{N_i, f_{i, j} : i \le j < \delta}$, $M_\delta := \bigcup_{i < \delta} M_i$, and $\ba_\delta := \bigcup_{i < \delta} f_{i, \delta} (\ba_i)$ (where $f_{i, \delta}: N_i \rightarrow N_\delta$ is the canonical embedding). Then $p := \gtp (\ba / M_\delta; N_\delta)$ extends each $p_i$. Note that $p$ depends on the witness but we sometimes abuse language and talk about ``the'' direct limit (where really some witnessing sequence is fixed in the background).

Finally, note that full model continuity also applies to coherent sequences. More precisely:

\begin{prop}\label{coher-seq-continuity}

  Assume $\is$ has full model continuity. Let $\seq{(\ba_i, M_i, N_i) : i < \delta}$, $\seq{f_{i, j} : N_i \rightarrow N_j, i \le j < \delta}$ be witnesses to the coherence of $p_i := \gtp (\ba_i / M_i; N_i)$. Assume that for each $i < \delta$, $\ba_i$ enumerates a model $M_i'$ and that $\seq{M_i^0 : i < \delta}$ are increasing such that $M_i^0 \lea M_i$, $M_i^0 \lea M_i'$, and $p_i$ does not fork over $M_i^0$. Let $p$ be the direct limit of the $p_i$s (according to the witnessing sequence). Then $p$ does not fork over $M_\delta^0 := \bigcup_{i < \delta} M_i^0$.
\end{prop}
\begin{proof}
  Use full model continuity inside the direct limit.
\end{proof}

\begin{lem}\label{ext-all}
  Assume $K$ is densely type-local above $\lambda$, and assume there exists a regular $\kappa \le \lambda$ such that $K$ is fully $(<\kappa)$-tame. Then $\is'$ has extension.
\end{lem}
\begin{proof}
  By Lemma \ref{cont-lem} and symmetry, $\is$ has the left $(<\kappa)$-model-witness property. By Lemmas \ref{full-mod-cont} and \ref{lc-long}, $\is'$ has full model continuity and the local character properties. Let $\lambda_0 \ge \lambda$ be a cardinal. We prove by induction on $\lambda_0$ that $\is'$ has extension for base models in $K_{\lambda_0}$. By Proposition \ref{indep-props}.(\ref{indep-props-7}), it is enough to prove independent amalgamation.

  Let $M^0 \lea M^\ell$, $\ell = 1,2$ be in $K$ with $\|M^0\| = \lambda_0$. We want to find $q \in \gS^{\lambda_0} (M^2)$ a nonforking extension of $p := \gtp (M^1 / M^0; M^1)$. Let $\lambda_\ell := \|M^\ell\|$ for $\ell = 1, 2$. 
 
  Assume we know the result when $\lambda_0 = \lambda_1 = \lambda_2$. Then we can work by induction on $(\lambda_1, \lambda_2)$: if they are both $\lambda_0$, the result holds by assumption. If not, we can assume by symmetry that $\lambda_1 \le \lambda_2$, find an increasing continuous resolution of $M^2$, $\seq{M_i^2 \in K_{<\lambda_2} : i < \lambda_2}$ and do a directed system argument as in \cite[Theorem 5.3]{ext-frame-jml} (using full model continuity and the induction hypothesis).

Now assume that $\lambda_0 = \lambda_1 = \lambda_2$. If $\lambda_0 = \lambda$, we get the result by extension in $\is$, so assume $\lambda_0 > \lambda$. Let $\delta := \cf{\lambda_0}$. By dense type-locality, we can assume (extending $M^1$ if necessary) that there exists $\seq{N_i : i \le \delta}$ an increasing continuous resolution of $M^1$ with $N_i \in K_{<\lambda_0}$ for $i < \delta$ so that $\seq{\gtp (N_j / M^0; M^1) : j < i}$ is type-local for all limit $i \le \delta$.

\paragraph{\textbf{Step 1}} Fix increasing continuous $\seq{M_i^\ell : i \le \delta}$ for $\ell < 2$ such that for all $i < \delta$, $\ell < 2$:

\begin{enumerate}
  \item $M^\ell = M_\delta^\ell$.
  \item $M_i^\ell \in K_{<\lambda_0}$.
  \item $N_i \lea M_i^1$.
  \item $M_i^0 \lea M_i^1$.
  \item $\nfs{M_i^0}{M_i^1}{M^0}{M^1}$.
\end{enumerate}

This is possible by repeated applications of Lemma \ref{lc-monot} (as in the proof of Lemma \ref{full-mod-cont}), starting with $\nfs{M^0}{M^1}{M^0}{M^1}$ which holds by existence.

\paragraph{\textbf{Step 2}}

Fix enumerations of $M_i^1$ of order type $\alpha_i$ such that $\seq{\alpha_i : i \le \delta}$ is increasing continuous, $\alpha_\delta = \lambda_0$ and $i < j$ implies that $M_i'$ appears as the initial segment up to $\alpha_i$ of the enumeration of $M_j'$. For $i \le \delta$, let $p_i := \gtp (M_i^1 / M_i^0; M^1)$ (seen as an element of $\gS^{\alpha_i} (M_i^0)$). We want to find $q \in \gS^{\lambda_0} (M^2)$ extending $p = p_\delta$ and not forking over $M^0$. Note that since for all $j < \delta$, $N_j \lea M_j^1$, we have by Lemma \ref{type-loc-monot} that $\seq{\gtp (M_j^1 / M^0; M^1) : j < i}$ is type-local for all limit $i \le \delta$. 

Build an increasing, coherent $\seq{q_i : i \le \delta}$ such that for all $i \le \delta$,

\begin{enumerate}
  \item $q_i \in \gS^{\alpha_i} (M^2)$.
  \item $q_i \rest M_i^0 = p_i$.
  \item $q_i$ does not fork over $M_i^0$.
\end{enumerate}

This is enough: then $q_\delta$ is an extension of $p = p_\delta$ that does not fork over $M_\delta^0 = M^0$.

This is possible: We work by induction on $i \le \delta$. While we do not make it explicit, the sequence witnessing the coherence is also built inductively in the natural way (see also \cite[Proposition 5.2]{ext-frame-jml}): at base and successor steps, we use the definition of Galois types. At limit steps, we take direct limits.

Now fix $i \le \delta$ and assume everything has been defined for $j < i$.

\begin{itemize}
  \item \textbf{Base step:} When $i = 0$, let $q_0 \in \gS^{\alpha_0} (M^2)$ be the nonforking extension of $p_0$ to $M_0^2$ (exists by extension below $\lambda_0$).
  \item \textbf{Successor step:} When $i = j + 1$, $j < \delta$, let $q_i$ be the nonforking extension (of length $\alpha_i$) of $p_i$ to $M^2$. We have to check that $q_i$ indeed extends $q_j$ (i.e.\ $q_i^{\alpha_j} = q_j$). Note that $q_j \rest M^0$ does not fork over $M_j^0$ so by step 1 and uniqueness, $q_j \rest M^0 = \gtp (M_j^1 / M^0; M^1)$. In particular, $q_j \rest M_i^0 = \gtp (M_j^1 / M_i^0; M^1)$. Since $q_i$ extends $p_i$, $q_i \rest M_i^0 = \gtp (M_i^1 / M_i^0; M^1)$ so $q_i^{\alpha_j} \rest M_i^0 = \gtp (M_j^1 / M_i^0; M^1) = q_j \rest M_i^0$. By base monotonicity, $q_j$ does not fork over $M_i^0$ so by uniqueness $q_i^{\alpha_j} = q_j$. A picture is below.

      
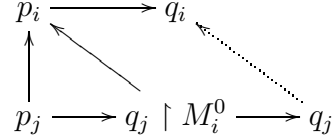
\[
  \xymatrix{ & p_i \ar[r] & q_i\\
    & p_j \ar[u] \ar[r] & q_j \rest M_i^0 \ar[lu] \ar[r] & q_j \ar@{.>}[lu] 
  }
  \]
  \item \textbf{Limit step:} Assume $i$ is limit. Let $q_i$ be the direct limit of the coherent sequence $\seq{q_j : j < i}$. Note that $q_i \in \gS^{\alpha_i} (M^2)$ and by Proposition \ref{coher-seq-continuity}, $q_i$ does not fork over $M_i^0$. It remains to see $q_i \rest M_i^0 = p_i$.

    For $j < i$, let $p_j' \in \gS^{\alpha_j} (M_i^0)$ be the nonforking extension of $p_j$ to $M_i^0$. By step 1, $p_j' = \gtp (M_j^1 / M_i^0; M^1)$. Thus $\seq{p_j' : j < i}$ is type-local. By an argument similar to the successor step above, we have that for all $j < i$, $p_i^{\alpha_j} = p_j'$. Moreover, for all $j < i$, $q_i^{\alpha_j} \rest M_j^0 = p_j$ and $q_j$ does not fork over $M_j^0$ so by uniqueness, $q_i^{\alpha_j} \rest M_i^0 = (q_i \rest M_i^0)^{\alpha_j} = p_j'$. By type-locality, it follows that $q_i \rest M_i^0 = p_i$, as desired.
\end{itemize} 

\end{proof}

Putting everything together, we get:

\begin{thm}\label{main-thm-long}
  If:

  \begin{enumerate}
    \item For some regular $\kappa \le \lambda$, $K$ is fully $(<\kappa)$-tame.
    \item $K$ is densely type-local above $\lambda$.
  \end{enumerate}
  
  Then $\is'$ is a fully good $(<\infty, \ge \lambda)$-independence relation.
\end{thm}
\begin{proof}
  Lemma \ref{basic-props} gives most of the properties of a good independence relation. By Lemma \ref{cont-lem} and symmetry, $\is$ has the left $(<\kappa)$-model-witness property. By Lemma \ref{full-mod-cont}, $\is'$ has full model continuity. By Lemma \ref{lc-long}, it has the local character properties. By Lemma \ref{ext-all}, $\is'$ has extension.
\end{proof}

We suspect that dense type-locality is not necessary, at least when $\is$ comes from our construction (see the proof of Theorem \ref{main-thm}). For example, by the proof below, it would be enough to see that $\pre(\is_{\mu}^{\le 1})$ is weakly successful for all $\mu \ge \lambda$. We delay a full investigation to a future work. For now, here is what we can say without dense type-locality:

\begin{thm}\label{main-thm-no-loc}
  Assume that for some regular $\kappa \le \lambda$, $K$ is fully $(<\kappa)$-tame. Then:

  \begin{enumerate}
    \item $\is'$ is a fully good independence relation, except perhaps for the extension property. Moreover, it has the right $\lambda$-witness property.
    \item Assume that\footnote{If for example $\is'$ is constructed as in the proof of Theorem \ref{main-thm}, this will be the case.} for all $\mu \ge \lambda$, $(\is')_{\ge \mu}$ satisfies Hypothesis \ref{ss-hyp-2}. Then $\is'$ has the extension property when the base is saturated.
  \end{enumerate}
\end{thm}
\begin{proof}
  The first part has been observed in the proof of Theorem \ref{main-thm-long} (see also Lemma \ref{basic-props}). To see the second part, let $\mu \ge \lambda$. By Theorem \ref{weakly-successful} and Theorem \ref{good-long-frame}, there exists a good $(\le \mu, \mu)$-independence relation $\is''$ with underlying class $\Ksatp{\mu}_\mu$. Using the witness properties and the arguments of \cite{bgkv-v3-toappear} we have that $(\is')^{\le \mu} \rest \Ksatp{\mu}_\mu = \is''$. By the proof of Lemma \ref{ext-all}, $\is'$ has extension when the base model is in $\Ksatp{\mu}_\mu$. 
\end{proof}

\section{The main theorems}\label{main-thm-sec}

Recall (Definition \ref{goodness-def-2}) that an AEC $K$ is fully good if there is a fully good independence relation with underlying class $K$. Intuitively, a fully good independence relation is one that satisfies all the basic properties of forking in a superstable first-order theory. We are finally ready to show that densely type-local fully tame and short superstable classes are fully good, at least on a class of sufficiently saturated models\footnote{The number 7 in (\ref{main-thm-1}) is possibly the largest natural number ever used in a statement about abstract elementary classes!}.

\begin{thm}\label{main-thm}
  Let $K$ be a fully $(<\kappa)$-tame and short AEC with amalgamation. Assume that $K$ is densely type-local above $\kappa$.

  \begin{enumerate}
    \item\label{main-thm-1} If $K$ is $\mu$-superstable, $\kappa = \beth_{\kappa} > \mu$, and $\lambda := \left(\mu^{<\kappap}\right)^{+7}$, then $\Ksatp{\lambda}$ is fully good.
    \item\label{main-thm-2} If $K$ is $\kappa$-strongly $\mu$-superstable and $\lambda := \left(\mu^{<\kappap}\right)^{+6}$, then $\Ksat$ is fully good.
    \item If $\kappa = \beth_{\kappa} > \LS (K)$, and $K$ is categorical in a $\mu > \lambda_0 := \left(\kappa^{<\kappap}\right)^{+5}$, then $K_{\ge \lambda}$ is fully good, where $\lambda := \min (\mu, \hanf{\lambda_0})$.
  \end{enumerate}
\end{thm}
\begin{proof}
  Given what has been proven already, the proofs are short. However to help the reader reflect on all the ground that was covered, we start by giving a summary in plain language of what the main steps in the construction are. Assume for example that $K$ is categorical in a high-enough cardinal $\mu > \kappa = \beth_\kappa > \LS (\K)$. By the results of Section \ref{sec-ss} (using results in \cite{bg-v9}, which ultimately rely on \cite{shvi635}), we get that $K$ is $\kappa$-strongly $\kappa$-superstable (note that, as opposed to \cite{sh394}, nothing is assumed about the cofinality of $\mu$). Thus coheir induces a good $(\le 1, \lambda)$-frame $\s$ with underlying class $\K_{\lambda}$, for $\lambda$ a high-enough cardinal. Moreover, coheir (seen as a global independence relation) has the properties in Hypothesis \ref{ss-hyp-2}. Thus from the material of Section \ref{domin-sec}, we conclude that the good frame is well-behaved: it is $\omega$-successful.

  By Section \ref{long-frame-sec}, this means that $\s$ can be extended to a good $(\le \lambda, \lambda)$-frame $\s'$ (so forking is defined not only for types of length one but for all types of length at most $\lambda$). With slightly more hypotheses on $\lambda$, we even can even make $\s'$ a \emph{fully} good $(\le \lambda, \lambda)$-frame, and by the ``minimal closure'' trick, into a fully good $(\le \lambda, \lambda)$ \emph{independence relation} $\is$. By Section \ref{up-transfer-sec}, $\is$ can be extended further to a fully good $(\le \lambda, \ge \lambda)$-independence relation $\is'$ (that is, forking is not only defined over models of size $\lambda$, but over models of all sizes at least $\lambda$). Finally, by Section \ref{long-transfer-sec}, we can extend $\is'$ to types of \emph{any} length (not just length at most $\lambda$), hence getting the desired global independence relation (a fully good $(<\infty, \ge \lambda)$-independence relation).

  Now on to the actual proofs:
  
  \begin{enumerate}
    \item By Theorem \ref{ss-strong-equiv} and Proposition \ref{ss-monot}, $K$ is $\kappa$-strongly $(2^{<\kappap})^+$-superstable. Now apply (\ref{main-thm-2}).
    \item By Fact \ref{ss-hyp-2-prop}, Hypothesis \ref{ss-hyp-2} holds for $\mu' := \left(\mu^{<\kappap}\right)^{+2}$, $\lambda$ there standing for $(\mu')^+$ here, and $K' := \Ksatp{\mu'}$. By Theorem \ref{succ-thm}, there is an $\omega$-successful type-full good $(\mu')^+$-frame $\s$ on $\Ksatp{(\mu')^+}$. By Theorem \ref{good-long-frame}, $\s^{+3}$ induces a fully good $(\le \lambda, \lambda)$-independence relation $\is$ on $\Ksatp{(\mu')^{+4}} = \Ksatp{\lambda}$. By Theorem \ref{main-thm-up}, $\is' := \cl(\pre (\is_{\ge \lambda}))$ is a fully good $(\le \lambda, \ge \lambda)$-independence relation on $\Ksat$. By Theorem \ref{main-thm-long}, $\islongp{(\is')}$ is a fully good $(<\infty, \ge \lambda)$-independence relation on $\Ksat$. Thus $\Ksat$ is fully good.
    \item\label{main-thm-3} By Theorem \ref{categ-frame}, $K$ is $\kappa$-strongly $\kappa$-superstable. By (\ref{main-thm-2}), $\Ksatp{\lambda_0^+}$ is fully good. By Fact \ref{categ-facts}.(\ref{sat-transfer}), all the models in $K_{\ge \lambda}$ are $\lambda_0^+$-saturated, hence $\Ksatp{\lambda_0^+}_{\ge \lambda} = K_{\ge \lambda}$ is fully good.
  \end{enumerate}
\end{proof}

We now discuss the necessity of the hypotheses of the above theorem. It is easy to see that a fully good AEC is $\ssp$. Moreover, the existence of a relation $\nf$ with disjointness and independent amalgamation directly implies disjoint amalgamation. An interesting question is whether there is a general framework in which to study independence without assuming amalgamation, but this is out of the scope of this paper. To justify full tameness and shortness, one can ask:

\begin{question}
  Let $K$ be a fully good AEC. Is $K$ fully tame and short?
\end{question} 

If the answer is positive, we believe the proof to be nontrivial. We suspect however that the shortness hypothesis of our main theorem can be weakened to a condition that easily holds in all fully good classes. In fact, we propose the following:

\begin{defin}\index{diagonally tame}\index{diagonally $(<\kappa)$-tame|see {diagonally tame}}
  An AEC $K$ is \emph{diagonally $(<\kappa)$-tame} if for any $\kappa' \ge \kappa$, $K$ is $(<\kappa')$-tame for types of length less than $\kappa'$. $K$ is \emph{diagonally $\kappa$-tame} if it is diagonally $(<\kappa^+)$-tame. $K$ is \emph{diagonally tame} if it is diagonally $(<\kappa)$-tame for some $\kappa$.
\end{defin}

It is easy to check that if $\is$ is a good $(<\infty, \ge \lambda)$-independence relation, then $K_{\is}$ is diagonally $\lambda$-tame. Thus we suspect the answer to the following should be positive:

\begin{question}
  In Theorem \ref{main-thm}, can ``fully $(<\kappa)$-tame and short'' be replaced by ``diagonally $(<\kappa)$-tame?
\end{question}

Finally, we believe the dense type-locality hypothesis can be removed\footnote{In fact, the result was initially announced without this hypothesis but Will Boney found a mistake in our proof of Lemma \ref{ext-all}. This is the only place where type-locality is used}. Indeed, chapter III of \cite{shelahaecbook} has several results on getting models ``generated'' by independent sequences. Since independent sequences exhibit a lot of finite character (see also \cite{tame-frames-revisited-v5}), we suspect the answer to the following should be positive. 

\begin{question}
  Is dense type-locality needed in Theorem \ref{main-thm}?
\end{question}

The construction also gives a more localized independence relation if we do not assume dense type-locality. Note that we can replace categoricity by superstability or strong superstability as in the proof of Theorem \ref{main-thm}.

\begin{thm}\label{good-frame-succ}
  Let $K$ be a fully $(<\kappa)$-tame and short AEC with amalgamation. Let $\lambda$, $\mu$ be cardinals such that:

  $$
  \LS (K) < \kappa = \beth_\kappa < \lambda = \beth_\lambda \le \mu
  $$

  Assume further that $\cf{\lambda} \ge \kappa$. If $K$ is categorical in $\mu$, then:

  \begin{enumerate}
    \item There exists an $\omega$-successful type-full good $\lambda$-frame $\s$ with $K_{\s} = K_\lambda$. Furthermore, the frame is induced by $(<\kappa)$-coheir: $\s = \pre (\isch{\kappa} (K)_\lambda^{\le 1})$.
    \item $K_\lambda$ is $(\le \lambda, \lambda)$-good.
    \item $\Ksatp{\lambda^{+3}}$ is fully $(\le \lambda^{+3})$-good.
    \item $\Ksatp{\lambda^{+3}}$ is fully good, except it may not have extension. Moreover it has extension over saturated models.
    \item Let $\is := \isch{\kappa} (\Ksatp{\lambda^{+4}})$. Then $\is$ is fully good, except it may not have extension.  Moreover it has extension over saturated models.
  \end{enumerate}
\end{thm}
\begin{proof}
  By cardinal arithmetic, $\lambda = \lambda^{<\kappap}$. By Fact \ref{ss-hyp-2-prop} and Theorem \ref{succ-thm}, there is an $\omega$-successful type-full good $\lambda$-frame $\s$ with $K_{\s} = \Ksatp{\lambda}_\lambda$. Now (by Theorem \ref{categ-frame} if $\mu > \lambda$), $K$ is categorical in $\lambda$. Thus $\Ksatp{\lambda} = K_{\ge \lambda}$. Theorem \ref{good-long-frame} and Theorem \ref{main-thm-up} give the next two parts. Theorem \ref{main-thm-no-loc} gives the fourth part. For the fifth part, let $\is'$ witness the fourth part. We use Theorem \ref{canon-coheir} with $\alpha$, $\lambda, \is$ there standing for $\lambda^{+4}$, $\lambda^{+3}$, $(\is')^{<\lambda^{+4}}$ here. We obtain that $(\is')^{<\lambda^{+4}} \rest \Ksatp{\lambda^{+4}} = \is^{<\lambda^{+4}}$. Since both $\is$ and $\is'$ have the left $\lambda^{+3}$-witness property, $\is' \rest \Ksatp{\lambda^{+4}} = \is$, as desired.
\end{proof}

\section{Applications}\label{examples-sec}

We give three contexts in which the construction of a global independence relation can be carried out. To simplify the statement of the results, we adopt the following convention:

\begin{notation}\index{high-enough}
  When we say ``For any high-enough cardinal $\lambda$'', this should be replaced by ``There exists an infinite cardinal $\lambda_0$ such that for all $\lambda \ge \lambda_0$''.
\end{notation}

\subsection{Fully $(<\aleph_0)$-tame and short AECs}\index{fully $(<\aleph_0)$-tame and short}

\begin{lem}\label{fully-tame-short-local}
  Let $\K$ be a fully $(<\aleph_0)$-tame and short AEC with amalgamation. Then $\K$ is type-local. 
\end{lem}
\begin{proof}
  Straightforward since types are determined by finite restrictions of their length.
\end{proof}

Note that the framework of fully $(<\aleph_0)$-tame and short AEC with amalgamation generalizes homogeneous model theory. It is more general since we are not assuming that all sets are amalgamation bases, nor that we are working in a class of models of a first-order theory omitting a set of types. As a result, we do not have the weak compactness of homogeneous model theory, so Corollary \ref{cor-aleph0-short} is (to the best of our knowledge) new.

\begin{cor}\label{cor-aleph0-short}
  Let $\K$ be a fully $(<\aleph_0)$-tame and short AEC with amalgamation. Assume that $\K$ is $\LS (\K)$-superstable. For any high-enough cardinal $\lambda$, $\Ksatp{\lambda}$ is fully good.
\end{cor}
\begin{proof}
  By Lemma \ref{fully-tame-short-local}, $K$ is type-local. Now apply Theorem \ref{main-thm}.
\end{proof}

\subsection{Fully tame and short eventually categorical AECs}\index{eventually categorical}\index{categorical on a tail}

An AEC is \emph{eventually categorical} (or \emph{categorical on a tail}) if it is categorical in any high-enough cardinal. Note that Theorem \ref{shelah-categ-cor} gives conditions under which this follows from categoricity in a single cardinal. In this context, we can also construct a global independence relation. To the best of our knowledge, this is new.

\begin{cor}
  Let $\K$ be a fully tame and short AEC with amalgamation. If $\K$ is eventually categorical, then for any high-enough cardinal $\lambda$, $\K_{\ge \lambda}$ is fully good. 
\end{cor}
\begin{proof}
  Let $\kappa$ be such that $\K$ is fully $(<\kappa)$-tame and short and let $\lambda_0$ be such that $\K$ is categorical in any $\lambda \ge \lambda_0$. We can make $\kappa$ bigger if necessary and replace $\K$ by $\K_{\ge \mu}$ for an appropriate $\mu$ to assume without loss of generality that $\LS (\K) < \kappa = \beth_\kappa$, and $\K$ is categorical in all $\lambda \ge \kappa$. We now apply Theorem \ref{good-frame-succ} to obtain that for any high-enough $\lambda$, $\K_{\ge \lambda}$ is fully good, except it may only have extension over saturated models. However any model is saturated by categoricity, so $\K_{\ge \lambda}$ is fully good.
\end{proof}

\subsection{Large cardinals}

Categoricity together with a large cardinal axiom implies that coheir is a well-behaved global independence relation. This was observed in \cite{makkaishelah} when the AEC is a class of models of an $\Ll_{\kappa, \omega}$ theory ($\kappa$ a strongly compact cardinal), and in \cite{bg-v9} for any AEC. Here we can improve on these results by proving that in this framework coheir is \emph{fully good}. In particular, it has full model continuity and $\clc{\alpha} (\is) = \alpha^+ +  \aleph_0$. Full model continuity is not discussed in \cite{makkaishelah, bg-v9}, and $\clc{\alpha} (\is) = \alpha^+ + \aleph_0$ is only proven when $\alpha < \kappa$ or $\alpha = \alpha^{<\kappa}$ (see \cite[Theorem 8.2.(3)]{bg-v9}). Further, the proof uses the large cardinal axiom whereas we use it only to prove that coheir has the extension property.

\begin{cor}
  Let $\K$ be an AEC and let $\kappa > \LS (\K)$ be a strongly compact cardinal. For any high-enough cardinal $\lambda > \kappa$, if $K$ is categorical in $\lambda$ then $\K_{\ge \lambda}$ is fully good as witnessed by coheir (that is, $\isch{\kappa} (\K_{\ge \lambda})$ is fully good).
\end{cor}
\begin{proof}
  By Fact \ref{boney-lc}, $\K$ is fully $(<\kappa)$-tame and short and $\K_{\ge \kappa}$ has amalgamation. By the last part of Theorem \ref{good-frame-succ}, there exists $\mu < \lambda$ such that $\is := \isch{\kappa} (\Ksatp{\mu^{+4}})$ is fully good, except perhaps for the extension property. By \cite[Theorem 8.2.(1)]{bg-v9} (using that $\kappa$ is strongly compact) $\is$ also has extension, hence it is fully good. Now the model of size $\lambda$ is saturated (Theorem \ref{categ-frame}), so $\K_{\ge \lambda} \subseteq \Ksatp{\mu^{+4}}$. Hence $\isch{\kappa} (\K_{\ge \lambda})$ is also fully good.
\end{proof}
\begin{remark}
  We can replace the categoricity hypothesis by amalgamation and $\kappa$-superstability. Moreover instead of asking for $\kappa$ to be a large cardinal, it is enough to assume that $\K$ has amalgamation, is fully $(<\kappa)$-tame and short, $\LS (\K) < \kappa = \beth_\kappa$, and coheir has the extension property (as in hypothesis (3) of \cite[Theorem 5.1]{bg-v9}).
\end{remark}

\printindex

\bibliographystyle{amsalpha}
\bibliography{independence-aecs}

\end{document}